\documentclass[reqno,10pt]{amsart}
\usepackage{amssymb,amsmath,amsthm}
\usepackage{mathrsfs}
\usepackage{amscd}
\usepackage{amsfonts}
\usepackage{amssymb}
\usepackage{latexsym}
\usepackage{color}
\usepackage{esint}
\usepackage{graphicx}
\usepackage{float}
\graphicspath{{Figures/}}

\usepackage{todonotes}
\usepackage[colorlinks,linkcolor=blue,anchorcolor=red,citecolor=blue]{hyperref}
\usepackage{cleveref}
\usepackage[margin=1in]{geometry} 
\usepackage{marginnote}

\usepackage{color}

\newcommand{\Blue}{\textcolor{blue}}

\newcommand{\dis}{\displaystyle}
\allowdisplaybreaks

\usepackage[makeroom]{cancel}

\newtheorem{theorem}{Theorem}[section]

\newtheorem{definition}{Definition}[section]
\newtheorem{lemma}{Lemma}[section]
\newtheorem{proposition}[theorem]{Proposition}
\newtheorem{remark}{Remark}[section]

\let\e=\varepsilon
\let\d=\delta
\let\h=v

\let\p=\partial

\let\O=\Omega

\numberwithin{equation}{section}

\let\hide\iffalse
\let\unhide\fi

\newcommand{\R}{\mathbb{R}}

\newcommand{\be}{\begin{equation}}
\newcommand{\bm}{\begin{multline}}
\newcommand{\ee}{\end{equation}}
\newcommand{\dd}{\mathrm{d}}

\newcommand{\xb}{x_{\mathbf{b}}}
\newcommand{\tb}{t_{\mathbf{b}}}

\newcommand{\Bes}{\begin{eqnarray*}}
\newcommand{\Ees}{\end{eqnarray*}}
\newcommand{\Be}{\begin{equation} }
\newcommand{\Ee}{\end{equation}}

\def\p{\partial}

\def\O{\Omega}
\def\R{\mathbb{R}}
\def\d{\mathrm{d}}
\def\B{\begin{equation}}
\def\E{\end{equation}}
\def\BN{\begin{eqnarray*}}
\def\EN{\end{eqnarray*}}

\usepackage{cite}

\makeatletter
\@namedef{subjclassname@2020}{%
  \textup{2020} Mathematics Subject Classification}
\makeatother

\begin{document}
\title[The Boltzmann equation in an infinite layer]{The Boltzmann equation in an infinite layer: spectrum and asymptotics toward the heat equation}

\author[H.-X. Chen]{Hongxu Chen}
\address[HXC]{Department of Mathematics, The Chinese University of Hong Kong,
Shatin, Hong Kong, P.R.~China}
\email{hongxuchen.math@gmail.com}

\author[R.-J. Duan]{Renjun Duan}
\address[RJD]{Department of Mathematics, The Chinese University of Hong Kong,
Shatin, Hong Kong, P.R.~China}
\email{rjduan@math.cuhk.edu.hk}

\author[S.-Q. Liu]{Shuangqian Liu}
\address[SQL]{{School of Mathematics and Statistics and and Key Lab NAA--MOE, Central China Normal University, Wuhan 430079, China}}
\email{sqliu@ccnu.edu.cn}


\begin{abstract}
In the paper, we develop spectral theory to analyze the sharp asymptotic behavior of solutions to the Boltzmann equation around global Maxwellians in a three-dimensional infinite layer $\mathbb{R}^2\times (-1,1)$. The isothermal diffuse reflection boundary condition is imposed on two parallel infinite planes at $x_3=\pm 1$. The main difficulties lie in the fact that the direct Fourier transform is not applicable to the vertical $x_3$-variable, and the linear collision operator $K$ loses its compactness on $L^2((-1,1)\times \R^3_v)$ although it is compact on $L^2(\R^3_v)$. By introducing a regularization operator $K_n$ via the finite-dimensional Fourier series truncation in $L^2(-1,1)$, we study the spectrum of the linearized initial-boundary value approximation problem, establish the resolvent estimates, and identify the leading diffusive eigenvalue. This spectral structure governs the sharp asymptotic dynamics of the original linear problem as $n\to \infty$, enabling us to construct the large-time behavior for the nonlinear problem and rigorously prove that the solution converges with a faster rate toward that of the two-dimensional heat equation in the horizontal direction. 
\end{abstract}

\date{\today}

\maketitle

\setcounter{tocdepth}{2}
\tableofcontents
\thispagestyle{empty}

\section{Introduction}

\subsection{Setting}
The Boltzmann equation provides a fundamental description of the dynamics of dilute gases, bridging microscopic particle interactions and macroscopic fluid behavior. A central theme in kinetic theory is to rigorously justify the emergence of macroscopic equations in appropriate asymptotic regimes. Among these, the diffusive limit and the characterization of large-time behavior have received particular attention, as they reveal how kinetic transport relaxes toward effective diffusion processes, cf.~\cite{Cercignani00,CIP}.

In this paper, we investigate the large-time asymptotics of solutions to the Boltzmann equation in a three-dimensional infinite layer $\Omega:=\mathbb{R}^2\times (-1,1)$: 
\begin{equation}\label{proF}
\p_t F + v\cdot\nabla_x F=Q(F,F), \ (t,x,v)\in [0,\infty) \times \O \times \mathbb{R}^3. 
\end{equation}
Here, $F=F(t,x,v)\geq 0$ stands for the velocity distribution function of gas particles with velocity $v=(v_1,v_2,v_3)\in \R^3$ at time $t\geq 0$ and position $x=(x_1,x_2,x_3)\in \Omega\subset \R^3$, and the initial and boundary conditions are to be specified later. The Boltzmann collision term is a bilinear integral operator acting only on velocity variable, and to the end, we only consider the hard sphere model, reading as 
\begin{equation*}
Q(F,G):=\int_{\mathbb{R}^3}\int_{\mathbb{S}^2}|(v-u)\cdot \omega|[F(u')G(v')-F(u)G(v)]\,\d\omega \dd u,
\end{equation*}
where the velocity pairs $(v,u)$ and $(v',u')$ satisfy 
\begin{equation*}
  v'=v+[(u-v)\cdot\omega]\omega,\quad u'=u-[(u-v)\cdot\omega]\omega,\quad \omega\in\mathbb{S}^2,
\end{equation*}
that's the $\omega$-representation in terms of the conservation of momentum and energy for elastic collisions between molecules:
\begin{equation*}
  v+u=v'+u',\quad |v|^2+|u|^2=|v'|^2+|u'|^2.
\end{equation*}

The gas particles also interact with the physical boundary $\partial\Omega$ at two parallel infinite planes $\mathbb{R}^2 \times \{x_3=\pm 1\}$. To describe the boundary condition, we split the boundary phase space $\partial\Omega\times \R^3_v$ into the outgoing set $\gamma_+$, incoming set $\gamma_-$ and grazing set $\gamma_0$ by
\begin{align*}
\gamma_{+}^\pm :=&\{(x,v)\in\mathbb{R}^2 \times \{x_3=\pm 1\}\times\mathbb{R}^{3}: v_3 \gtrless 0 \} , \ \gamma_+ := \gamma_+^+ \cup \gamma_+^-,  \\
\gamma_{-}^\pm :=&\{(x,v)\in\mathbb{R}^2 \times \{x_3=\pm 1\}\times\mathbb{R}^{3}:v_3 \lessgtr 0\}, \ \gamma_-:= \gamma_-^+ \cup \gamma_-^-, \\
\gamma_{0}^\pm :=&\{(x,v)\in\mathbb{R}^2 \times \{x_3=\pm 1\}\times\mathbb{R}^{3}:v_3=0\}, \ \gamma_0:= \gamma_0^+ \cup \gamma_0^-.
\end{align*}
We are interested in the infinite layer problem with the isothermal diffuse reflection boundary condition:
\begin{equation}\label{F.bc}
F(t,x,v)|_{\gamma_-^\pm}=c_\mu \mu(v)\int_{u_3\gtrless 0}F(t,x,u)|u_3| \,\dd u,
\end{equation}
where
\begin{equation}\label{def.mu}
\mu(v):=\frac{1}{(2\pi)^{3/2}}e^{-\frac{|v|^2}{2}}
\end{equation}
is a normalized global Maxwellian with zero bulk velocity and the constant $c_\mu=\sqrt{2\pi}$ is chosen to satisfy $\int_{v_3\gtrless 0}c_\mu\mu(v) |v_3|dv=1$ so that $c_\mu\mu(v) |v_3|$ is a probability measure on the half velocity spaces $\{v\in\R^3: v_3\gtrless 0\}$. Note that under the condition \eqref{F.bc}, the mass flux is vanishing at the boundaries, that is,
\begin{equation*}
\int_{\R^3} v_3F(t,x_1,x_2,x_3=\pm 1, v)\,\d v\equiv 0,\quad t\geq 0,x\in \p\Omega.
\end{equation*}

In the standard perturbation framework around the global Maxwellian $\mu$ in \eqref{def.mu}, we seek for the solution of the form $F = \mu + \sqrt{\mu}f$ to \eqref{proF} and \eqref{F.bc}. Then plugging this, the initial-boundary value problem on the perturbation $f$ is reformulated as
\begin{align}
\begin{cases}
     &\dis \p_t f + v\cdot \nabla_x f + \mathcal{L}f = \Gamma(f,f), \\
    &\dis  f(t,x,v)|_{\gamma_-^\pm } = c_\mu \sqrt{\mu(v)} \int_{u_3\gtrless 0} f(t,x,u)\sqrt{\mu(u)}|u_3|\,\dd u, \\   
    &\dis  f(0,x,v) = f_0(x,v) := (F(0,x,v)-\mu)/\sqrt{\mu}.  
\end{cases}\label{nonlinear_f}
\end{align}
Here $\mathcal{L}f$ and $\Gamma(f,f)$ denote the linearized collision term and nonlinear term, respectively:
\begin{align}
    &  \mathcal{L}f:=  -\mu^{-1/2}[Q(\mu,\sqrt{\mu}f) + Q(\sqrt{\mu}f,\mu)], \label{def.L}\\
    &   \Gamma(f,f) := \mu^{-1/2} Q(\sqrt{\mu}f,\sqrt{\mu}f).\label{def.Ga}
\end{align}
The linear operator $\mathcal{L}$ can be decomposed (see Lemma \ref{lemma:L_K} later on) as
\begin{align*}
    & \mathcal{L}f = \nu(v) f - Kf.
\end{align*}
We define two operators $A,B$ as
\begin{align}
    & Af = -v\cdot \nabla_x f - \nu(v)f, \label{def_A} \\
    &Bf = -v\cdot \nabla_x f - \nu(v) f + Kf. \label{def_B}
\end{align} 
The domain of $A,B$ and the corresponding semi-group $e^{At},e^{Bt}$ are to be discussed in Section \ref{sec:eAt_def}.

The global existence and time-decay of solutions to \eqref{nonlinear_f} was recently obtained in \cite{chen2025global} through the energy method; the sharp large-time behavior, in particular, proving the asymptotic toward diffusion waves along the two-dimensional horizontal direction, was left open. The goal of the current work is to develop a new method, based on spectral analysis, to determine the optimal large-time dynamics. We emphasize that the spectral theory of the Boltzmann equation in an infinite layer is a challenging topic in kinetic theory and of its own interest and importance.  

In what follows, we briefly review several main lines of research on the large-time behavior of cutoff Boltzmann equation and related fluid equations in the framework of spectral analysis.
\begin{itemize}
    \item The study of spectrum for the linearized Boltzmann collision operator $\mathcal{L}$ has a long history, which may date back to Wang Chang-Uhlenbeck \cite{chang1970propagation} in the case of Maxwell molecules; see also the monograph Cercignani \cite[Chapterm IV.6]{Cercignani} for more details on the classical works. In the spatially inhomogeneous setting, the spectral analysis of the linearized problem was developed by Ellis-Pinsky \cite{Ellis1975}. Later, through the spectral theory and bootstrap argument, the first global existence result on the nonlinear Boltzmann equation in both the whole space and the torus was pioneered by Ukai-Asano \cite{ukai1983steady,ukai1986steady}, and later refined by Ukai-Yang \cite{ukai2006boltzmann,ukai2006mathematical}. This classical approach focuses on characterizing the resolvent and spectrum of the operator $B$ defined in \eqref{def_B}. A key argument in this method is the exploitation of the compactness of the operator $K$ on $L^2(\mathbb{R}^3_v)$, which circumvents a detailed analysis of the essential spectrum. One crucial advantage of this spectral framework is the ability to reveal the precise structure of the leading order asymptotic behavior of solutions. 
    
    Moreover, this framework using the Boltzmann's spectrum method has been successfully applied to more complex systems with self-consistent forces in plasma physics by Li-Yang-Zhong, such as the Vlasov-Poisson-Boltzmann system and Vlasov-Maxwell-Boltzmann system \cite{li2016spectrum,li2021spectrum}. Another breakthrough was the factorization method developed by Gualdani–Mischler–Mouhot \cite{gualdani2010factorization} in the non-symmetric perturbation framework. We also mention the fundamental work  developed by Liu-Yu in \cite{liu2004green,liu2006green} via the method of Green's functions for the Boltzmann equation.\\

\item For the study of the large-time behavior of strong solutions in the whole space and on the torus, another powerful analysis tool is the energy method independently by Liu-Yang-Yu \cite{liu2004energy} and Guo \cite{guo2004boltzmann}. In addition, the hypocoercivity approach of Desvillettes-Villani \cite{desvillettes2005trend} provides a general theory of studying convergence to equilibrium states for large initial data. Closely related to the current work as well as our previous work \cite{chen2025global}, we also mention the $L^1_k$ method by Duan–Liu–Sakamoto–Strain \cite{duan2021global} on the torus and its counterpart in the whole space \cite{duan_SIMA} for the low-regularity solutions to the non-cutoff Boltzmann equation. By contrast, the spectral framework is able to uniquely capture the diffusion waves and hydrodynamic modes that govern the large-time dynamics, providing a detailed description of the solution beyond the mere decay rates that the usual energy approach gives.\\

    \item In the boundary value problem, high-order energy estimate is not applicable due to formation of singularity near the boundary \cite{GKTT,GKTT2,chen2024gradient,CK,kim2011formation}, and the Fourier transform technique is not available either due to the geometry and boundary condition.  Guo proposed a low-regularity framework using $L^2-L^\infty$ space in \cite{G}, and obtained the exponential decay structure of the solution under various boundary conditions. This breakthrough has led to substantial development of the boundary value problem in kinetic theory, including \cite{EGKM,EGKM2,liu2017initial,duan2019effects,CKL}.

    The $L^p-L^\infty$ framework has been successfully applied to the half space $\mathbb{T}^2\times \mathbb{R}^+$ or the exterior domain problems \cite{jung2025global,guo2024diffusive, jang2021incompressible, cao2023passage}. In such settings, the lack of Poincaré inequality presents a significant challenge, requiring new methods to address the combined difficulties of unboundedness and the presence of a boundary. A key argument in these studies involves leveraging the compactness of the boundary to establish crucial $L^6$ estimates. On the other hand, there are fewer results for domains with non-compact boundaries.

    As mentioned before, for the infinite layer domain $\mathbb{R}^2 \times (-1,1)$, a global solution with a decay rate nearly identical to the two-dimensional heat equation was constructed in our previous work \cite{chen2025global}. This argument is conducted in the $L^1_k L^\infty_{x_3,v}$ space through a Fourier transform in the horizontal directions($k$ is the horizontal Fourier variable). A related approach employing Fourier series in the horizontal directions was applied to the kinetic Couette flow in the domain $\mathbb{T}^2 \times (-1,1)$ in \cite {duan20243d}. We also mention the thermal transpiration problem in the infinite layer domain \cite{chen2007thermal}.

    Notably, the spectral analysis for boundary value problems remains largely undeveloped outside of specific cases. In the exterior domains \cite{ukai1983steady,ukai1986steady}, a perturbation argument shows that the derivation of eigenvalues and eigen-projections coincides with those of the whole space. Other than this, there are no spectral results for domains with non-compact boundaries in the Boltzmann theory, in particular for the infinite layer domain under consideration.\\

    \item However, through the Chapman–Enskog expansion, the Boltzmann equation can be approximated by the compressible Navier-Stokes equations. At the fluid level, in a series of papers \cite{kagei2007asymptotic,kagei2007resolvent,kagei2008large}, Kagei conducted a systematic study of the large-time asymptotics for the compressible Navier–Stokes system in an infinite layer, establishing resolvent estimates and semigroup asymptotics for the linearized problem and deriving diffusion profiles for the nonlinear system. These works give a comprehensive fluid-level description of diffusion phenomena in the infinite layer setting. Our present work can be regarded as a kinetic counterpart. Direct adoption of Kagei's approach seems not workable for the Boltzmann case. We expect to develop new techniques in this paper to understand the problem on the sharp asymptotic behavior of solutions toward diffusion waves in large time for the Boltzmann equation in the infinite layer. The difficulties and key ideas are to be described in Section \ref{sec:difficulty} later on for convenience of readers.

\end{itemize}

\subsection{Main result}\label{sec:result}
Before stating our main result, we introduce some notations and norms used throughout the paper. We denote $\mathbf{P}_0: L^2((-1,1)\times \mathbb{R}^3_v)\to L^2(\mathbb{R}^3_v)$ as the projection onto the average mass in $x_3$ direction:
\begin{align}
    & \mathbf{P}_0 f := \frac{1}{2}\int_{-1}^1 \int_{\mathbb{R}^3} f(x_3,u)\sqrt{\mu(u)}\, \dd u \dd x_3 \sqrt{\mu(v)},  
    \label{P0_def} 
\end{align}
for $f=f(x_3,v) \in L^2((-1,1)\times \mathbb{R}^3_v)$. We also denote $\mathbf{P}:L^2(\mathbb{R}^3_v)\to L^2(\mathbb{R}^3_v)$ as the standard projection onto the kernel space
$\ker\mathcal{L}$ $ = \text{span}\{\sqrt{\mu},v_1\sqrt{\mu},v_2\sqrt{\mu},v_3\sqrt{\mu},|v|^2\sqrt{\mu}\}$:
\begin{align}\label{abc-def}
\left\{\begin{array}{rll}
    &\mathbf{P}f: = a\sqrt{\mu} + \mathbf{b}\cdot v\sqrt{\mu} + c\frac{|v|^2-3}{2}\sqrt{\mu} ,\\[2mm]
    & a=\int_{\mathbb{R}^3}f\sqrt{\mu}\,\dd v, \ \mathbf{b} = (b_1,b_2,b_3),b_i = \int_{\mathbb{R}^3} v_if\sqrt{\mu}\,\dd v\, (i=1,2,3), \ c=\int_{\mathbb{R}^3} \frac{|v|^2-3}{2}f\sqrt{\mu}\,\dd v,
    \end{array}\right.
\end{align}
for $f=f(v)\in L^2(\mathbb{R}^3_v)$, where $a,\mathbf{b},c$ can be functions of $(t,x)$ whenever $f$ also depends on $(t,x)$.

Denote the horizontal spatial variable as
\begin{align*}
    x_{\parallel}:= (x_1,x_2)\in \R^2.
\end{align*}
For the linear estimate, we use the notation
\begin{align}\label{si-qm}
     \sigma_{q,m} := \frac{1}{q} - \frac{1}{2} + \frac{m}{2},
\end{align}
and denote the following functional space and norm:
\begin{align*}
    &    Z_q:= L^2_{x_3,v}L^q_{x_\parallel} : = L^2((-1,1)\times \mathbb{R}^3_v;L^q(\mathbb{R}^2_{x_\parallel})), \\
    &\Vert f(t)\Vert_{Z_q} : = \Vert f(t)\Vert_{L^2_{x_3,v}L^q_{x_\parallel}}.
\end{align*}
Here, $\sigma_{q,m}$ is the $L^q-L^2$ polynomial decay rate of $m$-order derivatives corresponding to the two-dimensional heat equation. 

For the nonlinear estimate, we denote the following functional spaces:
\begin{align*}
    &      L^\infty_{x_3,v}H^\ell_{x_\parallel}:= L^\infty((-1,1)\times \mathbb{R}^3_v;H^\ell(\mathbb{R}^2_{x_\parallel})), \\
    & L^\infty_{x_3,v}\dot{H}^\ell_{x_\parallel}:= L^\infty((-1,1)\times \mathbb{R}^3_v;\dot{H}^\ell(\mathbb{R}^2_{x_\parallel})), \\
    & L^\infty_{x_3,v}L^q_{x_\parallel}:= L^\infty((-1,1)\times \mathbb{R}^3_v;L^q(\mathbb{R}^2_{x_\parallel})).
\end{align*}
Here, $H^\ell(\mathbb{R}^2_{x_\parallel})$ and $\dot{H}^{\ell}(\mathbb{R}^2_{x_\parallel})$ represent the standard Sobolev space and homogeneous Sobolev space in $x_\parallel\in \mathbb{R}^2$, respectively. Correspondingly, we denote the following exponential velocity weighted norms:
\begin{align*}
&  w(v):=e^{\theta|v|^2}, \ 0<\theta<\frac{1}{4}, \\
    &\Vert f(t)\Vert_{H^{\theta,\ell}} :=  \Vert w f(t)\Vert_{L^\infty_{x_3,v}H^\ell_{x_\parallel}} = \Vert e^{\theta|v|^2} f(t)\Vert_{L^\infty_{x_3,v}H^\ell_{x_\parallel}} ,   \\
&    \Vert f(t)\Vert_{\dot{H}^{\theta,\ell}} := \Vert w f(t)\Vert_{L^\infty_{x_3,v}\dot{H}^\ell_{x_\parallel}} = \Vert e^{\theta|v|^2} f(t)\Vert_{L^\infty_{x_3,v}\dot{H}^\ell_{x_\parallel}}, \\
    &\Vert f(t)\Vert_{X^{\theta,q}} := \Vert wf(t)\Vert_{L^\infty_{x_3,v}L^q_{x_\parallel}} =  \Vert e^{\theta|v|^2}f(t)\Vert_{L^\infty_{x_3,v}L^q_{x_\parallel}}.
\end{align*}

The first main result of this paper is stated as follows.

\begin{theorem}\label{thm:asymptotic_stability}
There exist small constants $\delta>0$ and $c_0>0$ such that if the initial condition $F_0 := \mu + \sqrt{\mu}f_0\geq 0$ satisfies
\begin{align}
    \Vert f_0\Vert_{H^{\theta,\ell}} + \Vert f_0\Vert_{X^{\theta,q}} < \delta, \label{initial_condition}
\end{align}
then there exists a unique global solution $F := \mu + \sqrt{\mu}f\geq 0$ to the initial-boundary value problem \eqref{nonlinear_f} for $1\leq q < 2$ such that,
\begin{align*}
    &   \Vert f(t)\Vert_{H^{\theta,\ell}} \lesssim e^{-c_0 t}\Vert f_0\Vert_{H^{\theta,\ell}} +  (1+t)^{-\sigma_{q,0}} \Vert f_0\Vert_{X^{\theta,q}}.
\end{align*}

If we further assume $\mathbf{P}_0 f_0 \equiv 0$, then the existence and uniqueness of $F=\mu+\sqrt{\mu}f\geq 0$ hold for $1\leq q\leq 2$, and the extra decay rate holds:
\begin{align*}
    & \Vert f(t)\Vert_{H^{\theta,\ell}} \lesssim e^{-c_0 t}\Vert f_0\Vert_{H^{\theta,\ell}} + (1+t)^{-\sigma_{q,1}}\Vert f_0\Vert_{X^{\theta,q}}.
\end{align*}

\end{theorem}

Several remarks on Theorem \ref{thm:asymptotic_stability} are given below.

\begin{remark}\label{rmk:1}
    In the whole space, an extra decay rate requires the vanishing of all macroscopic components of the initial data, i.e., $\mathbf{P}f_0 = 0$. This comes from the fact that the operator $\hat{B}(k)$ has five eigenvalues in the space $L^2(\mathbb{R}_v^3)$, and the leading order of the eigen-projection is exactly $\mathbf{P}$(see \cite{ukai2006mathematical}). In our problem, the condition for extra decay reduces to $\mathbf{P}_0f_0 = 0$. This difference arises from the diffusive boundary effect. To be specific, the spectrum of the $\hat{B}(k)$ in the space $L^2((-1,1)\times \mathbb{R}^3_v)$ with the diffuse boundary condition contains only one eigenvalue. And the leading order term in the eigen-projection is exactly $\mathbf{P}_0$(see Proposition \ref{prop:eigenvalue}). 
\end{remark}

\begin{remark}
The $q<2$ constraint arises from the $(1+t)^{-(\frac{1}{q}-\frac{1}{2})}$ decay rate of the 2D heat equation. At $q=2$, the linear solution does not decay. In the nonlinear estimate, this becomes a borderline case due to logarithmic divergence $\int_0^t (1+t-s)^{-1} ds$ (see Lemma \ref{lemma:nonlinear_aprior_est}). Imposing the condition $\mathbf{P}_0 f_0 = 0$ provides extra decay for the linear solution. This resolves the borderline issue, yielding the construction of a global solution for the $q=2$ case.
\end{remark}

\begin{remark}
We expect that a similar analysis can be applied to the two-dimensional infinite layer $\mathbb{R} \times (-1,1)$ problem. In this setting, the linear decay rate is consistent with the one-dimensional heat equation, which is $(1+t)^{-\frac{1}{2}(\frac{1}{q}-\frac{1}{2})}$. For $q=1$, this becomes $(1+t)^{-1/4}$. It turns out that this rate is just sufficient to control the nonlinearity:
\begin{align*}
    & \int_0^t (1+t-s)^{-\frac{1}{4}-\frac{1}{2}}(1+s)^{-\frac{1}{4}\times 2} \,\dd s \lesssim (1+t)^{-\frac{1}{4}}.
\end{align*}
This yields the construction of a global solution that decays as $(1+t)^{-1/4}$ with $L^1(\mathbb{R}^1_{x_\parallel})$ initial data. 

We expect that such a result in the two-dimensional infinite layer problem can be proved using an analogous argument presented in the paper. This serves as a significant improvement over the method employed in \cite{chen2025global}, where the $L^1_k L^\infty_{x_3,v}$ framework could not yield a global solution due to the slow decay rates in low dimensional space.

\end{remark}

The leading asymptotic profile of the solution constructed in Theorem \ref{thm:asymptotic_stability} is characterized in the following result.

\begin{theorem}\label{thm:leading_behavior}
Let $1\leq q<2$ and the initial condition $f_0$ satisfy \eqref{initial_condition} with $\mathbf{P}_0f_0\neq 0$. Also, let $f$ be the solution constructed in Theorem \ref{thm:asymptotic_stability}. Then there exists a constant $\lambda^* > 0$, to be explicitly constructed in the proof, such that
\begin{align*}
    &   \Vert f(t,x,v)- \rho(t,x_\parallel)\sqrt{\mu(v)}\Vert_{L^2_{x,v}} \lesssim (1+t)^{-2\sigma_{q,0}} \log(2+t) [\Vert f_0\Vert_{H^{\theta,\ell}} + \Vert f_0\Vert_{X^{\theta,q}}].
\end{align*}
Here $\rho(t,x_\parallel)$ is the solution to the following heat equation in $\mathbb{R}^2_{x_\parallel}$:
\begin{align*}
\begin{cases}
       &  \p_t \rho = \lambda^* \Delta_{x_\parallel} \rho, \\
    & \rho(0,x_\parallel) = \rho_0(x_\parallel) . 
\end{cases}
\end{align*}
The initial condition $\rho_0(x_\parallel)$ is defined in terms of $\mathbf{P}_0f_0 = \rho_0(x_\parallel)\sqrt{\mu}$ as in \eqref{P0_def}.

\end{theorem}

Several remarks on Theorem \ref{thm:leading_behavior} are given below.

\begin{remark}
The identification of the leading asymptotic profile for the Boltzmann equation is one crucial benefit of the spectral method or the Green's function method, compared to other approaches.
    
\end{remark}

\begin{remark}
We first demonstrate that the leading order behavior of the nonlinear solution is governed by the linear dynamics, with an initial condition given by the $x_3$-averaged density (see \eqref{linear_leading}). The proof of this convergence towards the linear profile \eqref{linear_leading} raises the constraint $q<2$.

Furthermore, we show that the leading asymptotic profile of this linear equation is governed by a 2D heat equation. This structure can be intuitively explained by the spectral analysis in Proposition \ref{prop:eigenvalue}. As noted in Remark \ref{rmk:1}, there is only one eigenvalue in this infinite layer problem, a stark contrast to the case in the whole space. The diffusion coefficient $\lambda^*$ in the resulting heat equation corresponds to the leading order term of the eigenvalue; its construction is to be specified in Proposition \ref{prop:eigenvalue}, see \eqref{G1_eqn_l*} and \eqref{G1_eqn}. The leading order eigen-projection corresponding to this eigenvalue is represented by the density $\rho$, which corresponds to the $x_3$-averaged mass.
\end{remark}

\subsection{Difficulty and key ideas}\label{sec:difficulty}
We now discuss the principal difficulties and the methods employed to resolve them.

\medskip
\noindent$\bullet$ \textit{Loss of compactness}


As mentioned in the earlier literature review, the spectral analysis in the whole space heavily relies on the compactness of $K$ on $L^2(\mathbb{R}^3_v)$. In our problem, the presence of a boundary in the $x_3$-direction prevents the use of a Fourier transform in that direction. Consequently, we must work in the space $L^2((-1,1)\times \mathbb{R}^3_v)$ on which $K$ is no longer compact. This loss of compactness presents a significant obstacle to characterizing the essential spectrum of $\hat{B}(k)$. 

To overcome this, we introduce a regularization method in Section \ref{sec:regulariza} designed to recover the compactness. Specifically, we define $P_n f$ as the truncation of $f$'s Fourier series in the $x_3$ direction (see \eqref{Pnf}). This truncation integrates over $x_3$, and when composed with $K$, the resulting operator $K_n  := K P_n $ becomes compact on $L^2((-1,1)\times \mathbb{R}^3_v)$.

The regularized operators $K_n$ and $\mathcal{L}_n = \nu - K_n$ enjoy several good properties. As explained in Section \ref{sec:regulariza}, $\mathcal{L}_n$ maintains a coercive estimate analogous to that of $\mathcal{L}$ and also maintains orthogonality in the $L^2((-1,1)\times \mathbb{R}^3_v)$ inner product.

\medskip
\noindent$\bullet$ \textit{Resolvent estimate of large imaginary frequency}

We take the Fourier transform in the horizontal direction and carry out the spectrum analysis of the regularized operator $\hat{B}_n(k)$ in \eqref{def_hatBn}:
\begin{align*}
    & (\sigma+i\tau) \hat{f} + i(k_1v_1+k_2v_2) \hat{f} + v_3 \p_{x_3} \hat{f} + \nu \hat{f} - K_n \hat{f} = (\sigma+i\tau) \hat{f} - \hat{A}\hat{f} - K_n \hat{f} = g.
\end{align*}
However, with the presence of the transport operator in the $x_3$ direction, the imaginary part of the energy estimate does not provide any control as $\tau \to \infty$. On the other hand, the presence of $\tau$ in the macroscopic estimate in Proposition \ref{prop:macroscopic} yields a term that becomes unbounded in the limit.

The observation to overcome this lies in analyzing the resolvent $(\sigma+i\tau - \hat{B}_n(k))^{-1}$ through its connection to $(\sigma+i\tau - \hat{A}(k))^{-1}$ given by the second resolvent identity \eqref{A_B_related}. This relationship implies that $(\sigma+i\tau - \hat{B}_n(k))^{-1}$ exists and 
\begin{center}
$\lim_{|\tau|\to \infty}\Vert (\sigma+i\tau - \hat{B}_n(k))^{-1} \Vert_{L^2_{x_3,v}} \to 0$ if $\lim_{|\tau|\to \infty}\Vert (\sigma+i\tau - \hat{A}(k))^{-1} \Vert_{L^2_{x_3,v}} \to 0$. 
\end{center}
Although we cannot achieve such a convergence of $\hat{A}(k)$ from the energy estimate, we observe that $(\sigma+i\tau - \hat{A}(k))^{-1}$ admits an integral representation from the Laplace transform in \eqref{laplace_transform}. Interpreting this Laplace transform as a Fourier transform in $\tau$ allows us to establish the $L^2$ integrability with respect to $\tau$ of both the resolvent and its derivative, i.e.,
\begin{align*}
    &  \int_{-\infty}^\infty \Vert (\sigma+i\tau-\hat{A}(k))^{-1}\Vert_{L^2_{x_3,v}}^2 \dd \tau +  \int_{-\infty}^\infty \Vert \p_\tau(\sigma+i\tau-\hat{A}(k))^{-1}\Vert_{L^2_{x_3,v}}^2 \dd \tau < \infty.
\end{align*}
This integrability ensures that the resolvent norm converges to $0$ as $|\tau|\to\infty$. Therefore, the spectral analysis of $\hat{B}_n(k)$ can be relaxed to the region $|\tau| < \tau_0$, where the problematic term in the macroscopic estimate can be controlled(see Lemma \ref{lemma:large_tau_inv}).

\medskip
\noindent$\bullet$ \textit{Eigenvalue and eigen-projection of $\hat{B}_n(k)$ and semi-group representation of $e^{\hat{B}_n(k)t}$}

Our analysis on the eigenvalue and eigen-projection is inspired by the base case of $k=0$ for the eigenvalue problem:
\begin{align*}
    & (\sigma+ i\tau) \hat{f} + i(k_1v_1+k_2v_2) \hat{f} + v_3 \p_{x_3} \hat{f} + \mathcal{L}_n \hat{f} = 0, \ \hat{f}|_{\gamma_-} = P_\gamma \hat{f}.
\end{align*}
An energy estimate in $L^2((-1,1)\times \mathbb{R}^3_v)$ for $|k|=\sigma=0$ shows that any solution must satisfy 
$\Vert (\mathbf{I}-\mathbf{P})\hat{f}\Vert_{L^2_{x_3,v}} = 0$ and $|(I-P_\gamma)\hat{f}|_{L^2_{\gamma_+}} = 0$. 
Unlike the whole-space case, the extra boundary constraint excludes the modes $v\sqrt{\mu}$ and $|v|^2\sqrt{\mu}$. The only remaining eigenfunction is $M\sqrt{\mu}$, where $M$ is the average mass density defined by 
$M = \frac{1}{2}\int_{-1}^1 \int_{\mathbb{R}^3} \hat{f} \sqrt{\mu} dv dx_3$ (see Lemma \ref{lemma:eigenvalue_0}).

For $k \neq 0$, we expect the eigenvalue to be near $0$ and the corresponding eigenfunction to be near $M\sqrt{\mu}$. A key observation is that if $M=0$, the macroscopic estimate (Proposition \ref{prop:macroscopic}) combined with the energy estimate implies that the only solution for finite $|\tau|<\tau_0$ is the trivial one, $\hat{f} = 0$. Thus, non-trivial solutions can only exist for $M \neq 0$. To construct these solutions, we employ an asymptotic expansion in $|k|$ small enough: 
\begin{align*}
    &\hat{f} = M\sqrt{\mu} + M|k|G_1 + M |k|^2 G_2, \ \int_{-1}^1 \int_{\mathbb{R}^3} G_1 \sqrt{\mu} \,\dd v \dd x_3 = \int_{-1}^1 \int_{\mathbb{R}^3} G_2 \sqrt{\mu} \,\dd v \dd x_3 = 0,
\end{align*}
with the ansatz $\lambda = -\eta |k|^2$. The equation for $G_1$ is derived at order of $|k|$(see \eqref{G1_eqn}) and corresponds to a standard stationary problem with a source term and zero average mass condition. The equation for $G_2$ is given by the remaining terms (see \eqref{Gn2_eqn})
\begin{align*}
    &  -\eta \sqrt{\mu} -\eta|k|G_1 + i\frac{k_1v_1+k_2v_2}{|k|}G_1 + i(k_1v_1+k_2v_2) G_2 - \eta |k|^2 G_2 + v_3 \p_{x_3}G_2 + \mathcal{L}_nG_2 = 0.
\end{align*}
The zero average mass constraint for $G_2$ provides the key condition for determining the eigenvalue $\eta$. To be specific, integrating the above equation with $\sqrt{\mu}$ yields:
\begin{align*}
    &   \int_{-1}^1 \int_{\mathbb{R}^3} \Big[-\eta \mu  + i\frac{k_1v_1+k_2v_2}{|k|}G_1\sqrt{\mu} + i(k_1v_1+k_2v_2)G_2 \sqrt{\mu} \Big]\,\dd v \dd x_3 = 0.
\end{align*}
Consequently, $\eta$ consists of two parts:
\begin{itemize}
  \item[(a)] A leading order term governed explicitly by $i\iint \frac{k_1v_1+k_2v_2}{|k|}G_1 \sqrt{\mu} \,\d v \d x_3$, which is determined solely by the solution $G_1$ in \eqref{G1_eqn}. This term is proved to be real and strictly negative. 
  \item[(b)] A higher order remaining term governed implicitly by $i\iint (k_1v_1+k_2v_2)G_2 \sqrt{\mu}\, \d   v\d x_3$. To close the argument, we prove the well-posedness of $G_2$ via a contraction mapping argument, which allows us to control this higher order term and fully characterize $\eta$ (see Proposition \ref{prop:eigenvalue}).
\end{itemize}
This spectral analysis is crucial as it provides a semi-group representation of $e^{\hat{B}_n(k)t}$ through the inverse Laplace transform (Proposition \ref{prop:inverse_laplace}). The subsequent asymptotic behavior of the linear semi-group $e^{B_n t}$ is obtained via an inverse Fourier transform in $k$ (Lemma \ref{lemma:linear_decay}). Then, the long-time behavior can be characterized by the leading order eigenvalue $-\lambda^*|k|^2$ ($\lambda^*>0$) and the leading order eigen-projection onto the average density $\mathbf{P}_0 f$, as stated above. This structure is the crucial inspiration for the faster convergence towards the 2-dimensional heat equation in Theorem \ref{thm:leading_behavior}. The faster decay rate in both Theorem \ref{thm:asymptotic_stability} and Theorem \ref{thm:leading_behavior} occurs when $\mathbf{P}_0 f_0 = 0$, where the leading order term in the eigen-projection vanishes.

\medskip
\noindent$\bullet$ \textit{Asymptotic behavior of $e^{Bt}$ and nonlinear estimate.}

Due to the loss of compactness of $K$ in $L^2((-1,1)\times \mathbb{R}^3_v)$, the properties of essential spectrum to $\hat{B}(k)$ remain unclear. Consequently, a semi-group representation of $e^{\hat{B}(k)t}$ based on spectral theory is unavailable. To overcome this difficulty and determine the asymptotic behavior, we reformulate the problem by expressing $e^{Bt}$ as a perturbation of $e^{B_n t}$:
\begin{align*}
    & \p_t f + v\cdot \nabla_x f + \mathcal{L}_n f = \mathcal{L}_n f - \mathcal{L}f.
\end{align*}
While $e^{B_n t}$ itself provides the desired decay rate, the source term in the Duhamel formula converges to $0$: 
$$
\Vert\mathcal{L}_n f - \mathcal{L}f\Vert_{L^2_{x,v}} = \Vert K(f-P_nf)\Vert_{L^2_{x,v}}\to 0,
$$ 
as $n\to\infty$; see Proposition \ref{prop.Bdecay}.

For the nonlinear problem \eqref{nonlinear_f}, the presence of a boundary in $x_3$ and an unbounded domain in $x_\parallel$ leads us to adapt the function space with the norm $\Vert wf\Vert_{L^\infty_{x_3,v}H^\ell_{x_\parallel}}$. We then apply an $L^2((-1,1)\times \mathbb{R}^3_v)$-$L^\infty((-1,1)\times \mathbb{R}^3_v)$ bootstrap argument with the method of characteristic, as in \cite{G}. The high-order Sobolev space $H^\ell_{x_\parallel}$ is well-suited for the method of characteristics since the backward trajectories and stochastic cycles in Definition \ref{def:sto_cycle} are independent of $x_\parallel$, hence the Minkowski inequality can be applied(see Lemma \ref{lemma:bdr}).

\subsection{Outline} 

In Section \ref{sec:prelim}, we introduce several preliminaries, including the definitions and properties of the Boltzmann operators, the operators $A,\hat{A}$ and $B,\hat{B}$. In Section \ref{sec:regulariza}, we introduce the regularized Boltzmann operators $K_n,\mathcal{L}_n$ and the corresponding operators $B_n,\hat{B}_n(k)$. Several crucial properties of these regularized operators are discussed. In Section \ref{sec:macroscopic}, we prove a key macroscopic estimate to the eigenvalue problem of $\hat{B}_n(k)$. Section \ref{sec:linear} is devoted to the analysis of the spectrum and resolvent of $B$ and $B_n$, which further leads to the decay rate for the linear Boltzmann equation. In Section \ref{sec:nonlinear}, we employ an $L^2$–$L^\infty$ bootstrap argument to derive the decay rate for the nonlinear problem, and conclude Theorem \ref{thm:asymptotic_stability} at the end of Section \ref{sec:nonliear_asym}. Finally, we prove Theorem \ref{thm:leading_behavior} in Section \ref{sec:leading}.

\subsection{Other notations}
In addition to the notations introduced in Section~\ref{sec:result}, we will also use the following ones. In the spectral analysis, for given $a\in \mathbb{R}$, we define
\begin{align*}
    & \mathbb{C}_+(a) := \{\sigma+i\tau \in \mathbb{C}| \sigma>a\}.
\end{align*}

Given an operator $A$, we denote
\begin{align*}
    & \rho(A) := \text{resolvent of }A,\quad  \Sigma(A) := \text{spectrum of }A.
\end{align*}

We use the general norm:
\begin{align*}
    &  \Vert f\Vert_{L^2_\nu} := \Vert \nu^{1/2}f(v)\Vert_{L^2_v}=\Big(\int_{\R^3}\nu(v)|f(v)|^2\,\dd v\Big)^{1/2}.
\end{align*}
Here $\nu(v)$ is defined in Lemma \ref{lemma:L_K}.

When $f(x_3,v)$ is defined on $(-1,1)\times \mathbb{R}^3$, we slightly abuse the notation and define
\begin{align*}
\gamma_{+}^\pm :=&\{(x_3,v)\in \{x_3=\pm 1\}\times\mathbb{R}^{3}: v_3 \gtrless 0 \} , \ \gamma_+ := \gamma_+^+ \cup \gamma_+^-,  \\
\gamma_{-}^\pm :=&\{(x_3,v)\in \{x_3=\pm 1\}\times\mathbb{R}^{3}:v_3 \lessgtr 0\}, \ \gamma_-:= \gamma_-^+ \cup \gamma_-^- ,\\
\gamma_{0}^\pm :=&\{(x_3,v)\in  \{x_3=\pm 1\}\times\mathbb{R}^{3}:v_3=0\}, \ \gamma_0:= \gamma_0^+ \cup \gamma_0^-.
\end{align*}

Then when $f(x,v)$ is defined on $\O\times \mathbb{R}^3$ or $f(x_3,v)$ is defined on $(-1,1)\times \mathbb{R}^3$, on the boundary, let $P_\gamma$ denote the projection onto the diffusive reflection:
\begin{align*}
    & f(t,x,v)|_{\gamma_-} = P_\gamma f, \  P_{\gamma}f: = c_\mu \sqrt{\mu(v)} \int_{u_3\gtrless 0} f(t,x_1,x_2,x_3,u)\sqrt{\mu(u)}|u_3|\, \dd u \text{ for }x_3 = \pm 1, \\
    & f(t,x_3,v)|_{\gamma_-} = P_\gamma f, \ P_{\gamma}f: = c_\mu \sqrt{\mu(v)}\int_{u_3\gtrless 0} f(t,x_3,u)\sqrt{\mu(u)}|u_3|\, \dd u \text{ for }x_3 = \pm 1.
\end{align*}

When $f(x_3,v)$ is defined on $(-1,1)\times \mathbb{R}^3$, we define $f\in L^2_{\gamma_\pm}$ if $|f|_{L^2_{\gamma_\pm}}<\infty$, where
\begin{align*}
    & | f|_{L^2_{\gamma_\pm}} := \Big(\int_{v_3\gtrless 0} |f(1,v)|^2 |v_3| \,\dd v +  \int_{v_3 \lessgtr 0} |f(-1,v)|^2 |v_3| \,\dd v\Big)^{1/2}.  
\end{align*}
When $f(x,v)$ is defined on $\O\times \mathbb{R}^3$, we define $f\in L^2_{\gamma_\pm}$ if $|f|_{L^2_{\gamma_\pm}}<\infty$, where
\begin{align*}
    & |f|_{L^2_{\gamma_\pm}} := \Big(\int_{\mathbb{R}^2} \int_{v_3\gtrless 0}|f(x_1,x_2,1,v)|^2|v_3|\,\dd v\dd x_1\dd x_2 + \int_{\mathbb{R}^2} \int_{v_3\lessgtr 0}|f(x_1,x_2,-1,v)|^2|v_3|\,\dd v \dd x_1 \dd x_2  \Big)^{1/2}.
\end{align*}

In the end, $f \lesssim g$ means that there exists $C>1$ such that $f\leq C g$. Moreover, $f\leq o(1)g$ and $f \lesssim o(1)g$ both mean that there exists $0<\delta\ll 1$ such that $f\leq \delta g$.

\section{Preliminary}\label{sec:prelim}
In this section, we introduce several preliminary concepts and estimates. In Section \ref{sec:K_Gamma}, we outline the fundamental properties of the linear and nonlinear Boltzmann operators, $\mathcal{L}$ and $\Gamma$, defined in \eqref{def.L} and \eqref{def.Ga}. In Section \ref{sec:eAt_def}, we introduce the semi-group generated by the operator $A$ given in \eqref{def_A}.

\subsection{Properties of linear operator $K$ and nonlinear operator $\Gamma$}\label{sec:K_Gamma}

\begin{lemma}[\cite{R}]\label{lemma:L_K}
It holds that $\mathcal{L}=\nu(v)-K$, where 
\begin{equation*}
\nu(v):=\int_{\mathbb{R}^3}\int_{\mathbb{S}^2}|(v-u)\cdot \omega|\mu(u)\d\omega\d u,
\end{equation*} 
and
\begin{equation*}
Kf(v)=\int_{\mathbb{R}^3}\int_{\mathbb{S}^2}|(v-u)\cdot \omega|[\sqrt{\mu(v)\mu(u)}f(u)-\sqrt{\mu(u)\mu(u')}f(v')-\sqrt{\mu(u)\mu(v')}f(u')]\d\omega\d u.
\end{equation*}
Here, there is a positive constant $\nu_0>0$ such that the collision frequency $\nu(v)$ satisfies
\begin{equation*}
\nu(v) \geq \nu_0 \sqrt{|v|^2+1} \geq \nu_0.
\end{equation*}
The integral operator $K$ can be written as
\[
Kf(v)=\int_{\mathbb{R}^3}\mathbf{k}(v,u)f(u)\,\dd u,
\]
with a symmetric integral kernel $\mathbf{k}(v,u)=\mathbf{k}(u,v)$ satisfying
\Be\notag
 |\mathbf{k}  (v,u)| \lesssim e^{- \varrho |v-u|^2}/ |v-u|,
\Ee
for a constant $\varrho>0$.
\end{lemma}

\begin{lemma}[\cite{G}]\label{lemma:k_theta}
Let $0< \theta < \frac{1}{4}$, and $\mathbf{k}_\theta(v,u) := \mathbf{k}(v,u) \frac{e^{\theta |v|^2}}{e^{\theta |u|^2}}$, then there exists $\varrho>\tilde{\varrho}>0$ such that
\begin{align*}
    &  \mathbf{k}_\theta(v,u) \lesssim e^{-\tilde{\varrho}|v-u|^2}/|v-u|,
\end{align*}
and there exists $C_\theta > 0$ such that
\begin{equation*}
\int_{\mathbb{R}^3}  \mathbf{k}(v,u) \frac{e^{\theta |v|^2}}{e^{\theta |u|^2}}  \dd u  \leq \frac{C_\theta}{1+|v|}.
\end{equation*}
Moreover, for $N\gg 1$, we have
\begin{equation*}
\mathbf{k}_\theta(v,u) \mathbf{1}_{|v-u|> \frac{1}{N}} \leq C_N,
\end{equation*}
and
\begin{equation*}
\int_{|v|>N \text{ or } |u|>N \text{ or } |v-u|\leq \frac{1}{N}} \mathbf{k}_\theta(v,u) \dd u \lesssim \frac{1}{N} \leq o(1).
\end{equation*}

\end{lemma}

\begin{lemma}[\cite{ukai2006boltzmann} and \cite{G}]\label{lemma:gamma_property}
The following properties of the nonlinear operator $\Gamma$ \eqref{def.Ga} hold:
\begin{align*}
    & \Vert \Gamma(f,g)(t,\cdot,x_3,v)\Vert_{H^\ell_{x_\parallel}} \lesssim |\Gamma(f^0,g^0)(t,x_3,v)|, \ \text{ where } f^0(t,x_3,v) := \Vert f(t,\cdot,x_3,v)\Vert_{H^\ell_{x_\parallel}}, \\
        &  \Vert \nu^{-1}w\Gamma(f,g)\Vert_{L^\infty_{x,v}} \lesssim \Vert wf\Vert_{L^\infty_{x,v}} \Vert wg\Vert_{L^\infty_{x,v}}.
\end{align*}

\end{lemma}

\subsection{Semi-group $e^{At}$, $e^{Bt}$ and Fourier transform in horizontal direction}\label{sec:eAt_def}

We define the domain of the operator $A$ in \eqref{def_A} as:
\begin{align*}
    &     D(A) = \{f\in L^2(\O\times \mathbb{R}^3_v)| v\cdot \nabla_x f + \nu(v)f\in L^2(\O\times \mathbb{R}^3_v), \ f|_{\gamma_\pm} \in L^2_{\gamma_\pm}, \ f|_{\gamma_-} = P_\gamma f\}.
\end{align*}
By a similar argument as Ukai-Asano \cite{ukai1983steady} in the exterior domain, $A$ generates a strongly continuous semi-group on $L^2(\O\times \mathbb{R}^3_v)$. Here we note that the trace estimate $f|_{\gamma_\pm} \in L^2_{\gamma_\pm}$ is satisfied under the diffuse boundary condition due to the Ukai trace theorem presented in the appendix Section \ref{sec:trace}.

Since $K$ is a bounded operator on $L^2(\O\times \mathbb{R}^3_v)$, $B$ in \eqref{def_B} also generates a strongly continuous semi-group on $L^2(\O\times \mathbb{R}^3_v)$ with domain $D(B)=D(A)$. To study the asymptotic behavior of $e^{tB}$, we apply the Fourier transform in the horizontal direction. For $f\in D(B)=D(A)$, we obtain
\begin{align}
& \hat{A}(k)\hat{f}:= \mathcal{F}(Af) = -i(k_1v_1+k_2v_2) \hat{f} - v_3 \p_{x_3}\hat{f} - \nu \hat{f} , \label{def_hatA} \\
    &  \hat{B}(k)\hat{f} := \mathcal{F}(Bf) = -i(k_1v_1+k_2v_2) \hat{f} - v_3 \p_{x_3}\hat{f} - \nu \hat{f} + K \hat{f}. 
    \label{def_hatB}
\end{align}
Given $k\in \R^2$, we consider $\hat{A}(k),\hat{B}(k)$ as operators on the space $L^2((-1,1)\times \mathbb{R}^3_v)$ with the domain
\begin{align*}
    & D(\hat{A}(k)) = D(\hat{B}(k)) \\
    &:= \{\hat{f}\in L^2((-1,1)\times \mathbb{R}^3_v)| i(k_1v_1+k_2v_2) \hat{f} + v_3 \p_{x_3} \hat{f} +\nu \hat{f} \in L^2((-1,1)\times \mathbb{R}^3_v), \ \hat{f}|_{\gamma_\pm}\in L^2_{\gamma_\pm} , \ \hat{f}|_{\gamma_-} = P_\gamma \hat{f} \}.
\end{align*}
We can follow the argument by Ukai-Asano \cite{ukai1983steady}, with the observation that the extra term $i(k_1v_1+k_2v_2) \hat{f}$ does not contribute to the real part of the $L^2((-1,1)\times \mathbb{R}^3_v)$ energy estimate. Thus, similar to \eqref{def_A}, $\hat{A}(k)$ generates a strongly continuous semi-group on $L^2((-1,1)\times \mathbb{R}^3_v)$. Since $K$ is a bounded operator on $L^2((-1,1)\times \mathbb{R}^3_v)$, $\hat{B}(k)$ also generates a strongly continuous semi-group on $L^2((-1,1)\times \mathbb{R}^3_v)$. 

The resolvent properties of $\hat{A}(k)$ in \eqref{def_hatA} are given by the following lemma.

\begin{lemma}\label{lemma:A_resolvent}
It holds that
\[C_+(-\nu_0) \subset \rho(\hat{A}(k)),\]
and
\begin{align*}
    &  \Vert e^{t\hat{A}(k)}\Vert_{L^2_{x_3,v}} \lesssim e^{-\nu_0 t}.
\end{align*}

\end{lemma}

\begin{proof}
For $\lambda = \sigma + i \tau$, we consider
\begin{align*}
\begin{cases}
     & \lambda \hat{f} + i(k_1v_1+k_2v_2) \hat{f} + v_3 \p_{x_3} \hat{f} + \nu \hat{f} = g, \ g\in L^2_{x_3,v}, \\
    & \hat{f}|_{\gamma_-} = P_\gamma \hat{f}.  
\end{cases}
\end{align*}

We first derive the a priori estimate. Multiplying the equation by the complex conjugate $\bar{\hat{f}}$ and taking integration in $x_3,v$, we obtain the real part of the $L^2((-1,1)\times \mathbb{R}^3_v)$ energy estimate:
\begin{align*}
    & (\nu_0+\sigma)\Vert \hat{f}\Vert_{L^2_{x_3,v}}^2 + |(I-P_\gamma)\hat{f}|_{L^2_{\gamma_+}}^2 \leq \Vert g\Vert_{L^2_{x_3,v}}\Vert \hat{f}\Vert_{L^2_{x_3,v}}, \  \Vert \hat{f}\Vert_{L^2_{x_3,v}} \lesssim \frac{1}{\nu_0+\sigma} \Vert g\Vert_{L^2_{x_3,v}}  .
\end{align*}
Combining this with the trace estimate in Lemma \ref{lemma:trace_k} in the appendix, we obtain $\hat{f}\in D(\hat{A}(k))$.

The argument to prove the existence of such a solution can be found in \cite{ukai1983steady}. This justifies that $(\lambda-\hat{A}(k))^{-1}$ exists for $\sigma>-\nu_0$. 

The second statement in the lemma follows from the Hille-Yosida theorem. We conclude the lemma.
\end{proof}

\section{Regularization of $K$ and semi-group $e^{B_n t}$}\label{sec:regulariza}

As mentioned in the earlier discussion in Section \ref{sec:difficulty}, $K$ is not a compact operator on $L^2((-1,1)\times \mathbb{R}^3_v)$. We need to define the following regularized linear Boltzmann equation 
\begin{align}
\begin{cases}
       &\p_t f + v\cdot \nabla_x f + \mathcal{L}_nf =0, \\
       & f(0,x,v) = f_0(x,v) ,\\
    & f(t,x,v)|_{\gamma_-} = P_\gamma f . 
\end{cases} \notag
\end{align}
Here $\mathcal{L}_nf := \nu(v) f- K_n f$, and $K_n$ is the regularization operator of $K$, which is defined using truncation in the Fourier series:
\begin{align}
    & K_n f(x_3,v) := KP_n f(x_3,v), \notag\\
    &   P_n f(x_3,v):= \frac{1}{2}\int_{-1}^1 f(y_3,v)\dd y_3 + \sum_{m=1}^n \Big[\cos(m \pi x_3)\int_{-1}^1 f(y_3,v) \cos(m \pi y_3) \dd y_3 \notag\\
    & \ \  \ \ \ \ \ \ \ \ \ \ \ \ \ \ \ \ \ \ \ \ \ \ \ \ \  \ \  \ \ \ \ \ \ \ \ \ \ \ \ \ \ \ \ \ \ \ \ + \sin(m\pi x_3)\int_{-1}^1 f(y_3,v)\sin(m\pi y_3) \dd y_3\Big]. \label{Pnf}
\end{align}
The original $f$ can be decomposed into
\begin{align*}
    & f = P_n f + (I-P_n)f.
\end{align*}
Note that $(I-P_n)f$ satisfies the following elementary property:
\begin{align*}
    &   \int_{-1}^1 (I-P_n)f \dd x_3 = 0.
\end{align*}

To study the operator $B$ in \eqref{def_B}, we focus on the analysis of the regularized linear operator 
\begin{align}
    B_nf = -v\cdot \nabla_x f - \nu(v) f + K_nf.  \notag
\end{align}

Since $K_n$ is a bounded operator on $L^2(\O\times \mathbb{R}^3_v)$, similar to the case of $K$, $B_n$ also generates a strongly continuous semi-group on $L^2(\O\times \mathbb{R}^3_v)$ with $D(B_n)=D(A)$ in terms of discussions of Section \ref{sec:eAt_def}. We perform the spectral analysis on the Fourier transform version of $B_n$ corresponding to \eqref{def_hatB}:
\begin{align}
    &  \hat{B}_n(k)\hat{f} := \mathcal{F}(B_nf) = -i(k_1v_1+k_2v_2) \hat{f} - v_3 \p_{x_3}\hat{f} - \nu \hat{f} + K_n \hat{f}, \label{def_hatBn} 
\end{align}
with the domain
\begin{align*}
    & D(\hat{A}(k)) = D(\hat{B}_n(k)) \\
    &:= \{\hat{f}\in L^2((-1,1)\times \mathbb{R}^3_v)| i(k_1v_1+k_2v_2) \hat{f} + v_3 \p_{x_3} \hat{f} +\nu \hat{f} \in L^2((-1,1)\times \mathbb{R}^3_v), \ \hat{f}|_{\gamma_\pm}\in L^2_{\gamma_\pm} , \ \hat{f}|_{\gamma_-} = P_\gamma \hat{f} \}.
\end{align*}

To analyze the asymptotic behavior of $e^{\hat{B}_n(k)t}$, we will study the resolvent and spectrum of $\hat{B}_n(k)$ in Section \ref{sec:linear}. 

The properties of the regularized operators $K_n$ and $\mathcal{L}_n$ are summarized as follows.

The following coercive property for $\mathcal{L}_n f$ holds.

\begin{lemma}\label{lemma:Ln}
Let $f=f(x_3,v)$ be a complex function on $ (-1,1)\times \mathbb{R}^3$, then it holds that
\begin{align*}
    &   \int_{-1}^1 \int_{\mathbb{R}^3} \mathcal{L}_n f \bar{f} \dd v \dd x_3 \gtrsim \Vert (\mathbf{I}-\mathbf{P})P_n f\Vert_{L^2_{x_3,\nu}}^2 + \Vert \nu^{1/2}(f-P_n f)\Vert_{L^2_{x_3,v}}^2.
\end{align*}
Here $\bar{f}$ stands for the complex conjugate.

\end{lemma}

\begin{proof}
To simplify the notation, we denote
\begin{align}
    &  a_0(v) = \frac{1}{2}\int_{-1}^1  f(y_3,v) \dd y_3 \notag,\\
    & a_m(v) = \int_{-1}^1 f(y_3,v)\cos(m\pi y_3)\dd y_3, \label{am_notation}\\
    & b_m(v) = \int_{-1}^1 f(y_3,v)\sin(m\pi y_3) \dd y_3. \notag
\end{align}

We compute that
    \begin{align*}
    &     \int_{-1}^1\int_{\mathbb{R}^3} \mathcal{L}_n f \bar{f}\dd v \dd x_3= \int_{-1}^1 \int_{\mathbb{R}^3} [\mathcal{L}P_n f(x_3,v) + \nu(v)(f(x_3,v)-P_n f(x_3,v))]  \\
    &\times [P_n \bar{f}(x_3,v)+ \bar{f}(x_3,v)-P_n \bar{f}(x_3,v)]  \dd v \dd x_3 \\
    & \geq \Vert (\mathbf{I}-\mathbf{P})P_n f \Vert_{L^2_{x_3,\nu}}^2  + \Vert \nu^{1/2} (f-P_n f)  \Vert_{L^2_{x_3,v}}^2.
\end{align*}
The crossing terms are canceled due to
\begin{align*}
    &  \int_{-1}^1\int_{\mathbb{R}^3} \mathcal{L}P_n f(x_3,v)(\bar{f}(x_3,v)-P_n \bar{f}(x_3,v)) \dd v \dd x_3 \\
    & = \int_{-1}^1 \int_{\mathbb{R}^3} \Big\{\bar{f}(x_3,v)\mathcal{L}(a_0(v)) + \bar{f}(x_3,v)\sum_{m=1}^n \Big[ \cos (m\pi x_3) \mathcal{L}(a_m(v))  + \sin(m\pi x_3) \mathcal{L}(b_m(v)) \dd y_3 \Big) \Big] \Big\} \dd v \dd x_3 \\
    & -  \int_{\mathbb{R}^3} \Big\{ 2\bar{a}_0(v) \mathcal{L}(a_0(v)) + \sum_{m=1}^n \Big[ \bar{a}_m(v) \mathcal{L} (a_m(v))  + \bar{b}_m(v) \mathcal{L}(b_m(v)) \Big] \Big\} \dd v = 0,
\end{align*}
and
\begin{align*}
    &    \int_{-1}^1 \int_{\mathbb{R}^3}  [\nu(v)(f(x_3,v) - P_n f(x_3,v))] P_n \bar{f}(x_3,v) \dd v\dd x_3 \\
    &  = \int_{\mathbb{R}^3}\int_{-1}^1 \Big\{\nu(v) f(x_3,v) \bar{a}_0(v) + \sum_{m=1}^n \Big[ \nu(v) f(x_3,v)\cos(m\pi x_3) \bar{a}_m(v)  +\nu(v) f(x_3,v)\sin(m\pi x_3) \bar{b}_m(v) \Big]\Big\}\dd x_3\dd v \\
    & - \int_{\mathbb{R}^3} \Big\{2\nu(v) a_0(v)\bar{a}_0(v) +  \nu(v)\sum_{m=1}^n\Big[a_m(v)\bar{a}_m(v) + b_m(v)\bar{b}_m(v)\Big] \Big\}\dd v = 0.
\end{align*}
Here we used the orthogonality to cancel the crossing terms:
\begin{align*}
    &\int_{-1}^1 \cos (n\pi x_3) \cos (m\pi x_3) \dd x_3 = \delta_{nm}, \ \int_{-1}^1 \sin (n\pi x_3) \cos(m\pi x_3) \dd x_3 = \delta_{nm}, \\
    & \int_{-1}^1 \cos(n\pi x_3) \sin(n\pi x_3) \dd x_3 = 0.
\end{align*}
We conclude the lemma.
\end{proof}

The orthogonality property holds for $\mathcal{L}_n$.

\begin{lemma}\label{lemma:int_Ln}
Let $f=f(x_3,v)$ be a function defined on $ (-1,1)\times \mathbb{R}^3$, then it holds that
\begin{align*}
    & \int_{-1}^1\int_{\mathbb{R}^3} \mathcal{L}_n f \sqrt{\mu}\dd v \dd x_3 = 0.
\end{align*}

\end{lemma}
\begin{proof}
It is direct to compute
\begin{align*}
    &  \int_{\mathbb{R}^3}\int_{-1}^1 \mathcal{L}_n f\sqrt{\mu(v)} \dd x_3 \dd v = \int_{\mathbb{R}^3} \int_{-1}^1 (\mathcal{L} P_n f + K(f-P_n f) ) \sqrt{\mu(v)} \dd x_3 \dd v  \\
    & = \int_{\mathbb{R}^3} \sqrt{\mu(v)}K\Big( \int_{-1}^1 [f-P_n f] \dd x_3\Big) \dd v = 0.
\end{align*}
The lemma follows.
\end{proof}

It is important to observe that the compact property holds for $K_n$.

\begin{lemma}\label{lemma:compact}
$K_n$ is a compact operator on $L^2((-1,1)\times \mathbb{R}^3_v)$.

\end{lemma}

\begin{proof}

Let $f^i \rightharpoonup f$ weakly in $L^2((-1,1)\times \mathbb{R}^3_v)$. We will use the notation \eqref{am_notation} for $f$ and the following notation for $f^i$:
\begin{align}
    &  a_0^i(v) = \frac{1}{2}\int_{-1}^1  f^i(y_3,v) \dd y_3 \notag,\\
    & a_m^i(v) = \int_{-1}^1 f^i(y_3,v)\cos(m\pi y_3)\dd y_3, \notag\\
    & b_m^i(v) = \int_{-1}^1 f^i(y_3,v)\sin(m\pi y_3) \dd y_3. \notag
\end{align}

First we define 
\begin{align*}
    & K_n^N f:= \int_{\mathbb{R}^3} \mathbf{k}_N(v,u)P_n f(y_3,u) \dd u, \\
    &  \mathbf{k}_N(v,u) = \mathbf{1}_{|v-u|\geq \frac{1}{N}} \mathbf{1}_{|v|<N} \mathbf{k}(v,u).
\end{align*}
We first prove that $K_n^N$ is a compact operator in $L^2((-1,1)\times \mathbb{R}^3_v)$. 

We have
\begin{align*}
& K_n^N f^i =  \int_{\mathbb{R}^3} \mathbf{k}_N(v,u) \Big[ a_0^i(u) + \sum_{m=1}^n \Big(\cos(m\pi x_3)a^i_m(u) + \sin(m\pi x_3)b^i_m(u) \Big) \Big] \dd u \\
    &\lesssim_n \Big(\int_{-1}^1 \int_{\mathbb{R}^3} |\mathbf{k}_N(v,u)|^2 \dd u \dd y_3 \Big)^{1/2} \Big( \int_{-1}^1\int_{\mathbb{R}^3} |f^i|^2 \dd u \dd y_3 \Big)^{1/2} \lesssim \mathbf{1}_{|v|\leq N} \in L^1_{x_3,v}.
\end{align*}
By the dominating convergence theorem, 
\begin{align*}
& \lim_{i\to \infty} \int_{-1}^1 \int_{\mathbb{R}^3} |K_n^N f^i - K_n^N f|^2   \dd v \dd x_3 \\
    & =  \lim_{i\to\infty } \int_{-1}^1 \int_{\mathbb{R}^3} \Big|\int_{\mathbb{R}^3} \mathbf{k}_N(v,u)(a_0(u)-a^i_0(u)) \dd u +\sum_{m=1}^n \Big[\cos(m\pi x_3) \int_{\mathbb{R}^3} \mathbf{k}_N(v,u) (a_m(u)-a^i_m(u)) \dd u   \\
    &+ \sin(m\pi x_3)\int_{\mathbb{R}^3} \mathbf{k}_N(v,u) (b_m(u)-b_m^i(u))\dd u \Big]  \Big|^2 \dd v \dd x_3  \\
    & =   \int_{-1}^1 \int_{\mathbb{R}^3} \lim_{i\to\infty }\Big|\int_{\mathbb{R}^3} \mathbf{k}_N(v,u)(a_0(u)-a^i_0(u)) \dd u +\sum_{m=1}^n \Big[\cos(m\pi x_3) \int_{\mathbb{R}^3} \mathbf{k}_N(v,u) (a_m(u)-a^i_m(u)) \dd u   \\
    &+ \sin(m\pi x_3)\int_{\mathbb{R}^3} \mathbf{k}_N(v,u) (b_m(u)-b_m^i(u))\dd u \Big]  \Big|^2 \dd v \dd x_3  = 0.
\end{align*}
Here we used the weak convergence of $f^i$ in $L^2((-1,1)\times \mathbb{R}^3_v)$ such that
\begin{align*}
    & \int_{\mathbb{R}^3} \mathbf{k}_N(v,u) (a_m(u)-a^i_m(u)) \dd u = \int_{-1}^1 \int_{\mathbb{R}^3} \mathbf{k}_N(v,u) (f(y_3,u)-f^i(y_3,u)) \cos(m\pi y_3) \dd u \dd y_3 \to 0, \\
    & \int_{\mathbb{R}^3} \mathbf{k}_N(v,u) (b_m(u)-b^i_m(u)) \dd u  = \int_{-1}^1 \int_{\mathbb{R}^3} \mathbf{k}_N(v,u) (f(y_3,u)-f^i(y_3,u)) \sin(m\pi y_3) \dd u \dd y_3 \to 0.
\end{align*}

Next, we prove $\Vert K_n^N - K_n\Vert_{L^2_{x_3,v}} \to 0$. Once this is true, then we have
\begin{align*}
    &  \Vert K_n f^i - K_n f\Vert_{L^2_{x_3,v}} \leq \Vert K_n f^i - K_n^N f^i\Vert_{L^2_{x_3,v}} + \Vert K_n^N f^i - K_n^N f\Vert_{L^2_{x_3,v}} + \Vert K_n f- K_n^N f\Vert_{L^2_{x_3,v}} \to 0.
\end{align*}

We compute that 
\begin{align*}
    &   \Big(\int_{-1}^1 \int_{\mathbb{R}^3} |K_n^N f - K_n f|^2 \dd v \dd x_3 \Big)^{1/2} \\
    &= \Big\Vert    \int_{\mathbb{R}^3}  (\mathbf{k}_N(v,u) - \mathbf{k}(v,u))^{1/2} (\mathbf{k}_N(v,u) - \mathbf{k}(v,u))^{1/2} a_0(u)  \dd u \dd y_3 \\
    & + \sum_{m=1}^n \Big[ \cos(m\pi x_3) \int_{\mathbb{R}^3} (\mathbf{k}_N(v,u)-\mathbf{k}(v,u))^{1/2} (\mathbf{k}_N(v,u)-\mathbf{k}(v,u))^{1/2} a_m(u) \dd u  \\
    & + \sin(m\pi x_3) \int_{\mathbb{R}^3} (\mathbf{k}_N(v,u)-\mathbf{k}(v,u))^{1/2} (\mathbf{k}_N(v,u)-\mathbf{k}(v,u))^{1/2} b_m(u) \dd u \Big] \Big\Vert_{L^2_{x_3,v}}  \\
    &\lesssim \Big\Vert  \Big(\int_{-1}^1 \int_{\mathbb{R}^3} |\mathbf{k}_N(v,u)-\mathbf{k}(v,u)|\dd u\dd y_3 \Big)^{1/2} \Big( \int_{-1}^1 \int_{\mathbb{R}^3}  |\mathbf{k}_N(v,u)-\mathbf{k}(v,u)| |f(y_3,u)|^2\dd u \dd y_3 \Big)^{1/2} \Big\Vert_{L^2_{x_3,v}} \\
    & \lesssim \Big(\sup_v \int_{\mathbb{R}^3} |\mathbf{k}_N(v,u)-\mathbf{k}(v,u)| \dd u \Big)^{1/2} \Big[\int_{-1}^1 \int_{\mathbb{R}^3} |f(y_3,u)|^2 \dd u\dd y_3 \int_{\mathbb{R}^3} |\mathbf{k}_N(v,u) - \mathbf{k}(v,u)| \dd v \Big]^{1/2} \\
    & \lesssim \Vert f\Vert_{L^2_{x_3,v}} \Big[\int_{|v-u|\leq \frac{1}{N}} \frac{1}{|v-u|} e^{-\frac{|v-u|^2}{8}} + \frac{1}{1+|v|} \mathbf{1}_{|v|\geq N}\Big] \to 0.
\end{align*}

We conclude the lemma.
\end{proof}

\section{Macroscopic estimate to the eigenvalue problem}\label{sec:macroscopic}

In the spectral analysis of the regularized operator $\hat{B}_n(k)$ with $k\neq 0$ for \eqref{def_hatBn}, the following macroscopic estimate plays a crucial role. Recall $\mathcal{L}_n=\nu(v)-K_n$ with $K_n=KP_n$. 

\begin{proposition}\label{prop:macroscopic}
Let $g\in L^2((-1,1)\times \mathbb{R}^3_v)$ and $\lambda\in \mathbb{C}$. Suppose $\hat{f}\in D(\hat{B}_n(k))$ is a solution to the following stationary problem
\begin{align*}
\begin{cases}
        &  i(k_1 v_1 + k_2 v_2) \hat{f} + v_3 \p_{x_3} \hat{f} + \mathcal{L}_n\hat{f} = g  -\lambda \hat{f}, \\
        &  \hat{f}|_{\gamma_-} = P_\gamma \hat{f}.
\end{cases}
\end{align*}
Then, it holds that
\begin{align*}
    &  \Big\Vert \frac{|k|}{\sqrt{1+|k|^2}} \hat{a} \Big\Vert_{L^2_{x_3}} + \Vert \mathbf{\hat{\mathbf{b}}}\Vert_{L^2_{x_3}} + \Vert \hat{c}\Vert_{L^2_{x_3}} + \Vert \hat{a} - \hat{a}_M\Vert_{L^2_{x_3}} \\
    &\lesssim (1+|\lambda|)\Vert (\mathbf{I}-\mathbf{P})P_n \hat{f}\Vert_{L^2_{x_3,\nu}}  +(1+|\lambda|) \Vert \hat{f}(x_3,v) - P_n \hat{f}(x_3,v) \Vert_{L^2_{x_3,\nu}} \\
    &+ (1+|\lambda|)|(I-P_\gamma)\hat{f}|_{L^2_{\gamma_+}} + (1+|\lambda|)\Vert g\Vert_{L^2_{x_3,v}} + (1+|\lambda|)^2 \Vert(\mathbf{I}-\mathbf{P})\hat{f}\Vert_{L^2_{x_3,v}} .
\end{align*}
Here, $\hat{a}$, $\hat{\mathbf{b}}$ and $\hat{c}$ are defined as the macroscopic components $\mathbf{P}\hat{f}$ in \eqref{abc-def}, and
\begin{align*}
    & \hat{a}_M := \frac{1}{2}\int_{-1}^1 \hat{a}\dd x_3.
\end{align*}

\end{proposition}

\begin{proof}

In order to estimate the macroscopic component of $\hat{f}$, we employ the test function method proposed by Esposito-Guo-Kim-Marra \cite{EGKM,EGKM2}. The weak formulation with a test function $\psi$ to the eigenvalue problem reads
\begin{align}
    &  \underbrace{\int_{-1}^{1}\int_{\mathbb{R}^3} i(k_1v_1+k_2v_2) \hat{f} \psi \dd v \dd x_3 }_{\eqref{weak_formulation}_1} \underbrace{-  \int_{-1}^1 \int_{\mathbb{R}^3} v_3 \hat{f} \p_{x_3}\psi \dd v \dd x_3 }_{\eqref{weak_formulation}_2}+ \underbrace{  \int_{\mathbb{R}^3} v_3 [\hat{f}(k,1)  \psi(1) - \hat{f}(k,-1) \psi(-1) ]\dd v }_{\eqref{weak_formulation}_3} \notag\\
    & + \underbrace{ \int_{-1}^1 \int_{\mathbb{R}^3} \mathcal{L}_n\hat{f} \psi \dd v \dd x_3 }_{\eqref{weak_formulation}_4}  = \underbrace{\int_{-1}^1 \int_{\mathbb{R}^3} g \psi \dd v \dd x_3 }_{\eqref{weak_formulation}_5} - \underbrace{\int_{-1}^1\int_{\mathbb{R}^3} \lambda \hat{f} \psi \dd v \dd x_3}_{\eqref{weak_formulation}_6}. \label{weak_formulation}
\end{align}

\textit{Estimate of $\hat{a}$}. 

For the estimate of $\hat{a}$, we choose a test function as
\begin{align}
    &  \psi_a =   \sqrt{\mu} (|v|^2-10) (-i(k_1v_1+k_2v_2)  +  v_3 \p_{x_3}  )\phi_a.  \label{test_a}
\end{align}
Here $\phi_a$ satisfies the elliptic equation
\begin{align}
\begin{cases}
    & \dis (|k|^2 - \p_{x_3}^2)\phi_a(k,x_3) = -\bar{\hat{a}}(k,x_3) \frac{|k|^2}{1+|k|^2} , \ x_3\in (-1,1), \\
    &\dis    \p_{x_3} \phi_a(k, \pm 1) = 0.    
\end{cases}\label{elliptic_a}
\end{align}
Here $\bar{\hat{a}}$ stands for the complex conjugate of $\hat{a}$.

Multiplying \eqref{elliptic_a} by $\bar{\phi}_a$, the complex conjugate of $\phi_a$, we obtain
\begin{align*}
    &  |k|^2 \Vert \phi_a \Vert_{L^2_{x_3}}^2 + \Vert \p_{x_3}\phi_a \Vert_{L^2_{x_3}}^2 \lesssim  \frac{|k|^2}{1+|k|^2} \Vert \hat{a}\Vert_{L^2_{x_3}}^2 + o(1)\frac{|k|^2}{1+ |k|^2} \Vert \phi_a \Vert_{L^2_{x_3}}^2, \\
    & \Vert |k| \phi_a\Vert_{L^2_{x_3}}^2 + \Vert \p_{x_3}\phi_a\Vert_{L^2_{x_3}}^2 \lesssim \frac{|k|^2}{1+|k|^2} \Vert \hat{a}\Vert_{L^2_{x_3}}^2.
\end{align*}

Multiplying \eqref{elliptic_a} by $|k|^2 \bar{\phi}_a$, we obtain
\begin{align*}
    &    |k|^4 \Vert \phi_a\Vert_{L^2_{x_3}}^2 + |k|^2 \Vert \p_{x_3}\phi_a\Vert_{L^2_{x_3}}^2  \lesssim o(1) |k|^4 \Vert \phi_a \Vert_{L^2_{x_3}}^2 + \frac{|k|^4}{(1+|k|^2)^2}\Vert \hat{a}\Vert_{L^2_{x_3}}^2, \\
    & \Vert |k|^2 \phi_a\Vert_{L^2_{x_3}}^2 + \Vert |k|\p_{x_3}\phi_a\Vert_{L^2_{x_3}}^2 \lesssim \frac{|k|^4}{1+|k|^4} \Vert \hat{a}\Vert_{L^2_{x_3}}^2 \lesssim \frac{|k|^2}{1+|k|^2}\Vert \hat{a}\Vert_{L^2_{x_3}}^2.
\end{align*}

This leads to the estimate that
\begin{align}
    &    \Vert \p_{x_3}^2 \phi_a \Vert_{L^2_{x_3}} \lesssim |k|^2 \Vert \phi_a\Vert_{L^2_{x_3}} + \frac{|k|^2}{1+|k|^2} \Vert \hat{a}\Vert_{L^2_{x_3}} \lesssim \frac{|k|}{\sqrt{1+|k|^2}} \Vert \hat{a}\Vert_{L^2_{x_3}},  \notag  \\
    & \Vert ( |k|+ |k|^2) \phi_a\Vert_{L^2_{x_3}} + \Vert (1+|k|) \p_{x_3}\phi_a\Vert_{L^2_{x_3}} + \Vert \p_{x_3}^2 \phi_a\Vert_{L^2_{x_3}} \lesssim \frac{|k|}{\sqrt{1+|k|^2}} \Vert \hat{a}\Vert_{L^2_{x_3}}.\label{H2_est_a}
\end{align}
By trace theorem, using \eqref{H2_est_a} we conclude that
\begin{align}
    &    | |k| \phi_a(k,\pm 1)|+| \p_{x_3} \phi_a(k,\pm 1)| \lesssim    \frac{|k|}{\sqrt{1+|k|^2}} \Vert \hat{a}\Vert_{L^2_{x_3}}. \label{H1_trace_a}
\end{align}

We substitute \eqref{test_a} into \eqref{weak_formulation}. We decompose $\hat{f} = \mathbf{P}\hat{f} + (\mathbf{I}-\mathbf{P})\hat{f}$. Then we have
\begin{align*}
    &     \eqref{weak_formulation}_1  =  - \int_{-1}^1 \int_{\mathbb{R}^3} i (k_1v_1 + k_2v_2) \hat{a} \mu (|v|^2-10)i (k_1v_1 + k_2v_2) \phi_a \dd v \dd x_3 \\
    & +  \int^1_{-1} \int_{\mathbb{R}^3} i(k_1v_1 + k_2v_2) (\mathbf{I}-\mathbf{P})\hat{f} \sqrt{\mu} (|v|^2-10) (-i k_1v_1 -i k_2v_2 + v_3\p_{x_3}) \phi_a \dd v \dd x_3 .
\end{align*}
In the first line, the contribution of $v_3\p_{x_3}\phi_a$ and $\hat{\mathbf{b}},\hat{c}$ vanish from the oddness and 
\begin{align*}
    &  \int_{\mathbb{R}^3} v_i^2 (|v|^2-10)(\frac{|v|^2-3}{2}) \mu \dd v = 0.
\end{align*}
Then from $\int_{\mathbb{R}^3} v_i^2 (|v|^2-10)\mu \dd v = -5$, we further have
\begin{align*}
    &     \eqref{weak_formulation}_1 = -5 \int_{-1}^1  |k|^2 \phi_a \hat{a}  \dd x_3   \\
    &+ \underbrace{\int_{-1}^1 \int_{\mathbb{R}^3} i(k_1v_1+k_2v_2)(\mathbf{I}-\mathbf{P})\hat{f} \sqrt{\mu}(|v|^2-10) (-i(k_1v_1+k_2v_2) + v_3 \p_{x_3})\phi_a \dd v \dd x_3}_{E_1} . 
\end{align*}

Next, from the oddness we have
\begin{align*}
    &  \eqref{weak_formulation}_2 = - \int_{-1}^1 \int_{\mathbb{R}^3} v_3  \hat{a}\mu (|v|^2-10) v_3 \p_{x_3}^2 \phi_a \dd v \dd x_3 \\
    & -  \int_{-1}^1 \int_{\mathbb{R}^3} v_3 (\mathbf{I}-\mathbf{P})\hat{f} \sqrt{\mu} (|v|^2-10) (-i(k_1v_1+k_2v_2) + v_3 \p_{x_3}) \p_{x_3}\phi_a \dd v \dd x_3 \\
    &  = 5 \int_{-1}^1 \p_{x_3}^2 \phi_a \hat{a} \dd x_3  \underbrace{-  \int_{-1}^1 \int_{\mathbb{R}^3} v_3 (\mathbf{I}-\mathbf{P})\hat{f} \sqrt{\mu}(|v|^2-10) (-i(k_1v_1+k_2v_2) + v_3 \p_{x_3}) \p_{x_3} \phi_a \dd v \dd x_3}_{E_2} .
\end{align*}

Then $\eqref{weak_formulation}_1$ and $\eqref{weak_formulation}_2$ combine to be
\begin{align}
    & \eqref{weak_formulation}_1 + \eqref{weak_formulation}_2 = -5 \int_{-1}^1 (|k|^2 - \p_{x_3}^2) \phi_a \hat{a} \dd x_3  + E_1 + E_2  = 5\Big\Vert \frac{|k|}{\sqrt{1+|k|^2}} \hat{a} \Big\Vert_{L^2_{x_3}}^2 + E_1 + E_2.   \notag
\end{align}
Here $E_1+E_2$ corresponds to the contribution of $(\mathbf{I}-\mathbf{P})\hat{f}$, which is bounded as
\begin{align}
    &  |E_1+E_2| \lesssim \Vert (\mathbf{I}-\mathbf{P})\hat{f}\Vert_{L^2_{x_3,v}}^2 + o(1)\Big[\Vert |k|^2 \phi_a\Vert_{L^2_{x_3}}^2 + \Vert |k| \p_{x_3}\phi_a \Vert_{L^2_{x_3}}^2 + \Vert \p_{x_3}^2 \phi_a \Vert_{L^2_{x_3}}^2  \Big] \notag \\
    &\lesssim o(1)\Big\Vert \frac{|k|}{\sqrt{1+|k|^2}} \hat{a} \Big\Vert_{L^2_{x_3}}^2  + \Vert (\mathbf{I}-\mathbf{P})\hat{f}\Vert_{L^2_{x_3,v}}^2.  \notag
\end{align}
Here we have used \eqref{H2_est_a}.

Then we compute the boundary term $\eqref{weak_formulation}_3$. Without loss of generality, we only estimate the term at $x_3=1$. For the contribution of $P_\gamma \hat{f}$, we have
\begin{align*}
    &    \int_{\mathbb{R}^3} v_3   P_\gamma\hat{f}(k,1) \psi_a(1) \dd v  =   \int_{\mathbb{R}^3} v_3 P_\gamma \hat{f}(k,1) \sqrt{\mu} (|v|^2 -10) (-i(k_1v_1+k_2v_2) + v_3 \p_{x_3}) \phi_a  \dd v  = 0.
\end{align*}
Here we have used the oddness to have
\begin{align*}
    &     \int_{\mathbb{R}^3} v_3 \mu (|v|^2 - 10) (-i (k_1v_1+k_2v_2)) \phi_a \dd v = 0,
\end{align*}
and the boundary condition $\p_{x_3} \phi_a = 0$ to have
\begin{align*}
    &   \int_{\mathbb{R}^3} v_3^2 \mu (|v|^2-10) \p_{x_3} \phi_a \dd v = 0.
\end{align*}
Thus from the trace estimate \eqref{H1_trace_a}, we derive that
\begin{align}
    &    \Big| \int_{v_3>0} (I-P_\gamma)\hat{f}(k,1) \sqrt{\mu} (|v|^2-10) (-i(k_1v_1+k_2v_2) + v_3 \p_{x_3}) \phi_a \dd v   \Big|  \notag \\
    & \lesssim o(1) [| |k| \phi_a(k,1) | +  | \p_{x_3} \phi_a(k,1)|  ]^2   +    |(I-P_\gamma)\hat{f}|_{L^2_{\gamma_+}}^2 \notag \\
    & \lesssim o(1) \Big\Vert \frac{|k|}{\sqrt{1+|k|^2}} \hat{a} \Big\Vert_{L^2_{x_3}}^2 + |(I-P_\gamma)\hat{f}|_{L^2_{\gamma_+}}^2.   \notag
\end{align}

By the same computation to term at $x_3 = -1$, we obtain
\begin{align*}
    & | \eqref{weak_formulation}_3| \lesssim o(1) \Big\Vert \frac{|k|}{\sqrt{1+|k|^2}} \hat{a} \Big\Vert_{L^2_{x_3}}^2 + |(I-P_\gamma)\hat{f}|_{L^2_{\gamma_+}}^2. 
\end{align*}

Next we compute $\eqref{weak_formulation}_4, \eqref{weak_formulation}_5$ and $\eqref{weak_formulation}_6$ as
\begin{align}
    &    |\eqref{weak_formulation}_4| \lesssim \int_{-1}^1\int_{\mathbb{R}^3} \mu^{1/4} [|k\phi_a| + |\p_{x_3}\phi_a|] \big| \nu[\hat{f} -P_n \hat{f}] + \mathcal{L}P_n \hat{f} \big| \notag\\
    &\lesssim o(1)[\Vert k \phi_a \Vert_{L^2_{x_3}}^2 + \Vert \p_{x_3} \phi_a \Vert_{L^2_{x_3}}^2 ]   +  \Vert (\mathbf{I}-\mathbf{P})P_n \hat{f} \Vert_{L^2_{x_3,v}}^2  + \Vert \hat{f} - P_n \hat{f} \Vert_{L^2_{x_3,v}}^2\notag \\
    &\lesssim o(1)\Big\Vert \frac{|k|}{\sqrt{1+|k|^2}} \hat{a}  \Big\Vert^2_{L^2_{x_3}}  + \Vert (\mathbf{I}-\mathbf{P})P_n \hat{f} \Vert_{L^2_{x_3,v}}^2  + \Vert \hat{f} - P_n \hat{f} \Vert_{L^2_{x_3,v}}^2,   \notag
\end{align}
\begin{align}
    & |\eqref{weak_formulation}_5| =  \Big|    \int_{-1}^1 \int_{\mathbb{R}^3} g\psi_a \dd v \dd x_3  \Big|    \lesssim o(1)[\Vert k \phi_a \Vert_{L^2_{x_3}}^2 + \Vert \p_{x_3} \phi_a \Vert_{L^2_{x_3}}^2 ] + \Vert g\Vert_{L^2_{x_3,v}}^2 \notag\\
    & \lesssim  o(1)\Big\Vert \frac{|k|}{\sqrt{1+|k|^2}} \hat{a}  \Big\Vert^2_{L^2_{x_3}} + \Vert g\Vert_{L^2_{x_3,v}}^2,   
    \notag
\end{align}
and
\begin{align}
    &  |\eqref{weak_formulation}_6| = \Big|\int_{-1}^1 \int_{\mathbb{R}^3} \lambda \hat{f} [\sqrt{\mu}(|v|^2-5)(-i(k_1v_1+k_2v_2) + v_3\p_{x_3})\phi_a -5 \sqrt{\mu}(-i(k_1v_1+k_2v_2) + v_3 \p_{x_3})\phi_a] \dd v \dd x_3  \Big| \notag\\
    & \lesssim o(1)[\Vert k\phi_a\Vert_{L^2_{x_3}}^2 + \Vert \p_{x_3}\phi_a\Vert_{L^2_{x_3}}^2] + [|\lambda|^2 \Vert (\mathbf{I}-\mathbf{P})\hat{f}\Vert_{L^2_{x_3,v}}^2 + |\lambda|^2 \Vert \mathbf{b}\Vert_{L^2_{x_3}}^2 ] \notag \\
    & \lesssim o(1)\Big\Vert \frac{|k|}{\sqrt{1+|k|^2}} \hat{a}  \Big\Vert^2_{L^2_{x_3}} + |\lambda|^2 \Vert (\mathbf{I}-\mathbf{P})\hat{f}\Vert_{L^2_{x_3,v}}^2 + |\lambda|^2 \Vert \hat{\mathbf{b}}\Vert_{L^2_{x_3}}^2 .
    \notag
\end{align}

We conclude the estimate of $\hat{a}$: 
\begin{align*}
 \Big\Vert \frac{|k|}{\sqrt{1+|k|^2}} \hat{a}  \Big\Vert_{L^2_{x_3}}    &     \lesssim  \Vert (\mathbf{I}-\mathbf{P})P_n \hat{f} \Vert_{L^2_{x_3,v}}  + \Vert \hat{f} - P_n \hat{f}\Vert_{L^2_{x_3,v}}   + |(I-P_\gamma)\hat{f}|_{L^2_{\gamma_+}} + \Vert g\Vert_{L^2_{x_3,v}}  \\
    & + |\lambda| \Vert (\mathbf{I}-\mathbf{P})\hat{f}\Vert_{L^2_{x_3,v}} + |\lambda| \Vert \hat{\mathbf{b}}\Vert_{L^2_{x_3}}.
\end{align*}

\medskip
\textit{Estimate of $\hat{a}-\hat{a}_M$.}
We choose the same test function \eqref{test_a}, with different $\phi_a$:
\begin{align}
    \begin{cases}
        (|k|^2-\p_{x_3}^2) \phi_a(k,x_3) = -(\bar{\hat{a}}(k,x_3) - \bar{\hat{a}}_M(k)), \\
        \p_{x_3} \phi_a(k,\pm 1) = 0.
    \end{cases} \label{phi_aM}
\end{align}
Multiplying \eqref{phi_aM} by $\bar{\phi}_a$ and taking integration in $x_3$ we obtain
\begin{align*}
    &   \Vert |k| \phi_a \Vert_{L^2_{x_3}}^2 + \Vert \p_{x_3} \phi_a \Vert_{L^2_{x_3}}^2 \lesssim o(1) \Vert \phi_a\Vert_{L^2_{x_3}}^2 + \Vert \hat{a}-\hat{a}_M\Vert_{L^2_{x_3}}^2.
\end{align*}
Since $\int_{-1}^1 (\bar{\hat{a}}(k,x_3) - \bar{\hat{a}}_M(k))  \dd x_3 = 0$, from the Poincar\'e inequality, we further have
\begin{align}
    & \Vert (1+|k|)\phi_a\Vert^2_{L^2_{x_3}}+ \Vert \p_{x_3}\phi_a\Vert^2_{L^2_{x_3}} \lesssim \Vert \hat{a}-\hat{a}_M\Vert^2_{L^2_{x_3}}.  \label{phiaM_l2_k}
\end{align}
Multiplying \eqref{phi_aM} by $|k|^2 \bar{\phi}_a$, we obtain
\begin{align*}
    &  \Vert |k|^2\phi_a \Vert_{L^2_{x_3}}^2 + \Vert |k|\p_{x_3}\phi_a\Vert^2_{L^2_{x_3}} \lesssim o(1) \Vert |k|^2 \phi_a\Vert_{L^2_{x_3}}^2 +  \Vert \hat{a}-\hat{a}_M\Vert_{L^2_{x_3}}^2.
\end{align*}
Thus we conclude
\begin{align}
    & \Vert |k|^2 \phi_a\Vert_{L^2_{x_3}}^2 + \Vert |k|\p_{x_3}\phi_a\Vert^2_{L^2_{x_3}} \lesssim \Vert \hat{a}-\hat{a}_M\Vert^2_{L^2_{x_3}},   \notag \\
    & \Vert \p_{x_3}^2 \phi_a\Vert_{L^2_{x_3}}^2 \lesssim \Vert |k|^2 \phi_a\Vert_{L^2_{x_3}}^2 + \Vert \hat{a}-\hat{a}_M\Vert_{L^2_{x_3}}^2 \lesssim \Vert \hat{a}-\hat{a}_M\Vert^2_{L^2_{x_3}}.    \notag
\end{align}

By the trace theorem, we have
\begin{align}
    & | |k|\phi_a(k,\pm 1) |^2 \lesssim \Vert \hat{a}-\hat{a}_M\Vert^2_{L^2_{x_3}}, \quad | \p_{x_3}\phi_a(k,\pm 1)|^2 \lesssim \Vert \hat{a}-\hat{a}_M\Vert_{L^2_{x_3}}^2. \label{phiaM_trace_2}
\end{align}

The computation for \eqref{weak_formulation} is the same as the estimate of $\hat{a}$, except the following term:
\begin{align*}
    & \eqref{weak_formulation}_1 + \eqref{weak_formulation}_2 = -5\int_{-1}^1 (|k|^2- \p_{x_3}^2) \phi_a \hat{a} \dd x_3 + E_1 +E_2 \\
    & = 5\int_{-1}^1 (\bar{\hat{a}} - \bar{\hat{a}}_M) \hat{a} \dd x_3 + E_1 = 5\Vert \hat{a}-\hat{a}_M\Vert_{L^2_{x_3}}^2 + 5\int_{-1}^1 (\bar{\hat{a}}-\bar{\hat{a}}_M) \hat{a}_M \dd x_3 + E_1 +E_2\\
    & = 5\Vert \hat{a}-\hat{a}_M\Vert_{L^2_{x_3}}^2 + E_1+E_2.
\end{align*}
Here we used $ \int_{-1}^1 (\hat{a} - \hat{a}_M )\dd x_3 = 0.$

Thus we obtain the following estimate for $\hat{a} - \hat{a}_M$ by replacing $\frac{|k|}{\sqrt{1+|k|^2}} \hat{a}$ with $\hat{a}-\hat{a}_M$:
\begin{align*}
  \Vert \hat{a}-\hat{a}_M\Vert_{L^2_{x_3}}  &    \lesssim \Vert (\mathbf{I}-\mathbf{P})P_n \hat{f} \Vert_{L^2_{x_3,v}} + \Vert \hat{f} - P_n \hat{f}\Vert_{L^2_{x_3,v}} + |(I-P_\gamma)\hat{f}|_{L^2_{\gamma_+}} + \Vert g\Vert_{L^2_{x_3,v}} \\
  &+|\lambda| \Vert(\mathbf{I}-\mathbf{P})\hat{f}\Vert_{L^2_{x_3,v}} + |\lambda| \Vert \hat{\mathbf{b}}\Vert_{L^2_{x_3}}.
\end{align*}

\medskip 
\textit{Estimate of $\hat{\mathbf{b}}$.}

We choose a test function as
\begin{align}
    & \psi_b =  -\frac{3}{2}\Big(|v_1|^2 - \frac{|v|^2}{3} \Big)\sqrt{\mu} ik_1 \phi_b  - v_1v_2 \sqrt{\mu} ik_2 \phi_b + v_1 v_3 \sqrt{\mu} \p_{x_3} \phi_b \perp \ker\mathcal{L}, \notag
\end{align}
and let $\phi_b$ satisfy the elliptic system
\begin{align}
\begin{cases}
    &   [2|k_1|^2 + |k_2|^2 - \p_{x_3}^2] \phi_b = \bar{\hat{b}}_1, \ \ -1<x_3<1,\\
    & \phi_b(k,\pm1) = 0.
\end{cases} \notag
\end{align}
Since the Poincar\'e inequality holds, similar to \eqref{phiaM_l2_k}-\eqref{phiaM_trace_2}, we have the following regularity estimate and trace estimate:
\begin{align}
    & \Vert (1+|k|+|k|^2)\phi_b\Vert^2_{L^2_{x_3}}+ \Vert (1+|k|)\p_{x_3}\phi_b\Vert^2_{L^2_{x_3}} + \Vert \p_{x_3}^2 \phi_b\Vert_{L^2_{x_3}}^2 \lesssim \Vert \hat{b}_1\Vert^2_{L^2_{x_3}},  \label{phib_l2_k} \\
    &| |k|\phi_b(k,\pm 1) |^2 +  | \p_{x_3}\phi_b(k,\pm 1)|^2 \lesssim \Vert \hat{b}_1\Vert_{L^2_{x_3}}^2. \label{phib_trace_2}
\end{align}

We first compute
\begin{align*}
    & \eqref{weak_formulation}_1 =    \int_{-1}^1 \int_{\mathbb{R}^3} i (k_1v_1+k_2v_2) (\hat{\mathbf{b}}\cdot v)\sqrt{\mu} \psi_b  \dd v \dd x_3  + \underbrace{\int_{-1}^1\int_{\mathbb{R}^3} i(k_1v_1+k_2v_2) (\mathbf{I}-\mathbf{P})\hat{f} \psi_b \dd v \dd x_3}_{E_3} \\
    &  =  \int_{-1}^1 \int_{\mathbb{R}^3} i (k_1v_1 + k_2v_2)(\hat{b}_1 v_1 + \hat{b}_2 v_2 + \hat{b}_3v_3)\sqrt{\mu}\psi_b \dd v \dd x_3  + E_3 \\
    & =  \int_{-1}^1 \int_{\mathbb{R}^3} \Big[ \frac{3}{2}|k_1|^2 |v_1|^2 \hat{b}_1 \Big( |v_1|^2 - \frac{|v|^2}{3}\Big) \mu \phi_b + v_1^2 v_2^2 k_1 k_2 \hat{b}_2 \mu \phi_b + v_1^2 v_2^2 k_2^2 \hat{b}_1 \mu \phi_b \\
    & \ \ \ \ \ \ +\frac{3}{2} \Big(|v_1|^2 - \frac{|v|^2}{3} \Big) v_2^2 k_2 k_1 \hat{b}_2 \mu \phi_b \Big] \dd v \dd x_3  +   \int_{-1}^1 \int_{\mathbb{R}^3} i k_1 \hat{b}_3 v_1^2 v_3^2 \mu \p_{x_3}\phi_b \dd v \dd x_3  + E_3 \\
    &  =  \int_{-1}^1  [2 |k_1|^2 \hat{b}_1 + |k_2|^2 \hat{b}_1] \phi_b \dd x_3  + ik_1 \hat{b}_3 \phi_b  +  E_3.
\end{align*}
The contribution of $\hat{a}$ and $\hat{c}$ has vanished from the oddness.

Here, by \eqref{phib_l2_k},
\begin{align}
    &   |E_3| \lesssim o(1)[\Vert |k|^2\phi_b\Vert_{L^2_{x_3}}^2 + \Vert |k|\p_{x_3}\phi_b \Vert_{L^2_{x_3}}^2] + \Vert (\mathbf{I}-\mathbf{P})\hat{f}\Vert_{L^2_{x_3,v}}^2  \lesssim o(1)\Vert \hat{b}_1\Vert_{L^2_{x_3}}^2 + \Vert (\mathbf{I}-\mathbf{P})\hat{f}\Vert^2_{L^2_{x_3,v}}. \notag
\end{align}

Next, we compute
\begin{align*}
    &    \eqref{weak_formulation}_2 = - \int_{-1}^1 \int_{\mathbb{R}^3} v_3 (\hat{\mathbf{b}}\cdot v)\sqrt{\mu} \p_{x_3}\psi_b \dd v \dd x_3  \underbrace{-\int_{-1}^1 \int_{\mathbb{R}^3} v_3 (\mathbf{I}-\mathbf{P})\hat{f} \p_{x_3} \psi_b \dd v \dd x_3}_{E_4} \\
    & = - \int_{-1}^1 \int_{\mathbb{R}^3} v_1^2 v_3^2 \hat{b}_1  \mu \p_{x_3}^2 \phi_b  \dd v \dd x_3  +  \int_{-1}^1 \int_{\mathbb{R}^3}  \frac{3}{2}ik_1\hat{b}_3 v_3^2\Big( |v_1|^2- \frac{|v|^2}{3}\Big)\mu \dd v \dd x_3  + E_4 \\
    &= - \int_{-1}^1 \hat{b}_1 \p_{x_3}^2 \phi_b \dd x  - ik_1 \hat{b}_3 \phi_b + E_4.
\end{align*}
The contribution of $\hat{a},\hat{c}$ and $v_2$ has vanished from the oddness.

Here, by \eqref{phib_l2_k},
\begin{align}
    & |E_4| \lesssim o(1) [\Vert |k|\p_{x_3}\phi_b \Vert^2_{L^2_{x_3}} + \Vert \p_{x_3}^2 \phi_b\Vert^2_{L^2_{x_3}}] + \Vert (\mathbf{I}-\mathbf{P})\hat{f}\Vert^2_{L^2_{x_3,v}} \lesssim  o(1)\Vert \hat{b}_1\Vert_{L^2_{x_3}}^2 + \Vert (\mathbf{I}-\mathbf{P})\hat{f}\Vert^2_{L^2_{x_3,v}}.\notag
\end{align}
Then we have
\begin{align}
    &     \eqref{weak_formulation}_1 + \eqref{weak_formulation}_2 =  \int_{-1}^1 [2|k_1|^2 + |k_2|^2 - \p_{x_3}^2]\phi_b \hat{b}_1 \dd x_3  + E_3 + E_4   = \Vert \hat{b}_1\Vert_{L^2_{x_3}}^2 + E_3 + E_4. \notag
\end{align}

Then we compute the boundary term $\eqref{weak_formulation}_3$. At $x_3=1$, for the contribution of $P_\gamma \hat{f}$, we have
\begin{align*}
    &    \int_{\mathbb{R}^3} v_3   P_\gamma \hat{f}(k,1) \psi_b(1) \dd v  \\
    &=   \int_{\mathbb{R}^3} v_3 P_\gamma \hat{f}(k,1)\Big[ -\frac{3}{2}\Big(|v_1|^2 - \frac{|v|^2}{3} \Big)\sqrt{\mu} ik_1 \phi_b  - v_1v_2 \sqrt{\mu} ik_2 \phi_b + v_1 v_3 \sqrt{\mu} \p_{x_3} \phi_b \Big] \dd v  = 0.
\end{align*}
Here we have used the oddness.

For the part with $(I-P_\gamma)\hat{f}$, we derive that
\begin{align}
    & \Big| \int_{v_3>0} (I-P_\gamma)\hat{f}(k,1) \Big[ -\frac{3}{2}\Big(|v_1|^2 - \frac{|v|^2}{3} \Big)\sqrt{\mu} ik_1 \phi_b  - v_1v_2 \sqrt{\mu} ik_2 \phi_b + v_1 v_3 \sqrt{\mu} \p_{x_3} \phi_b \Big] \dd v \Big| \notag \\
    & \lesssim o(1) [| |k| \phi_b(k,1) |^2 +  | \p_{x_3} \phi_b(k,1)|^2  ]   +    |(I-P_\gamma)\hat{f}|_{L^2_{\gamma_+}}^2 \notag \\
    & \lesssim o(1) \Vert \hat{b}_1 \Vert_{L^2_{x_3}}^2 + |(I-P_\gamma)\hat{f}|_{L^2_{\gamma_+}}^2.  \notag
\end{align}
In the last line, we have used the trace estimate \eqref{phib_trace_2}.

Similarly, for $x_3=-1$ we have the same estimate. Thus we conclude that
\begin{align}
    &  |\eqref{weak_formulation}_3| \lesssim o(1)\Vert \hat{b}_1\Vert_{L^2_{x_3}}^2 + |(I-P_\gamma)\hat{f}|_{L^2_{\gamma_+}}^2.   \notag
\end{align}

Next we compute $\eqref{weak_formulation}_4$, $\eqref{weak_formulation}_5$ and $\eqref{weak_formulation}_6$ as
\begin{align}
    &    |\eqref{weak_formulation}_4 |\lesssim o(1)[\Vert k \phi_b \Vert_{L^2_{x_3}}^2 + \Vert \p_{x_3} \phi_b \Vert_{L^2_{x_3}}^2 ] +  \Vert (\mathbf{I}-\mathbf{P})P_n \hat{f} \Vert_{L^2_{x_3,v}}^2  + \Vert \hat{f} - P_n \hat{f} \Vert_{L^2_{x_3,v}}^2\notag \\
    & \lesssim o(1)\Vert \hat{b}_1\Vert_{L^2_{x_3}}^2  +  \Vert (\mathbf{I}-\mathbf{P})P_n \hat{f} \Vert_{L^2_{x_3,v}}^2  + \Vert \hat{f} - P_n \hat{f} \Vert_{L^2_{x_3,v}}^2,   \notag
\end{align}
\begin{align}
    & |\eqref{weak_formulation}_5 |=  \Big|    \int_{-1}^1 \int_{\mathbb{R}^3} g\psi_b \dd v \dd x_3  \Big| \lesssim o(1)[\Vert k \phi_b \Vert_{L^2_{x_3}}^2 + \Vert \p_{x_3} \phi_b \Vert_{L^2_{x_3}}^2 ] + \Vert g\Vert_{L^2_{x_3,v}}^2  \lesssim  o(1)\Vert \hat{b}_1\Vert_{L^2_{x_3}}^2  + \Vert g\Vert_{L^2_{x_3,v}}^2,
    \notag
\end{align}
and
\begin{align}
    &   |\eqref{weak_formulation}_6| = \Big|\int_{-1}^1 \int_{\mathbb{R}^3} \lambda \hat{f}\psi \dd x \dd v  \Big| = \Big|\int_{-1}^1 \int_{\mathbb{R}^3} \lambda (\mathbf{I}-\mathbf{P})\hat{f}\psi \dd x \dd v  \Big| \notag\\
    & \lesssim o(1)[\Vert k\phi_b\Vert^2_{L^2_{x_3}} + \Vert \p_{x_3}\phi_b\Vert_{L^2_{x_3}}^2] + |\lambda|^2 \Vert (\mathbf{I}-\mathbf{P})\hat{f}\Vert_{L^2_{x_3,v}}^2   \lesssim o(1)\Vert \hat{b}_1\Vert_{L^2_{x_3}}^2 + |\lambda|^2 \Vert (\mathbf{I}-\mathbf{P})\hat{f}\Vert_{L^2_{x_3,v}}^2 \notag.
\end{align}

We conclude the estimate of $\hat{b}_1$: 
\begin{align*}
    &     \Vert \hat{b}_1\Vert_{L^2_{x_3}} \lesssim    \Vert (\mathbf{I}-\mathbf{P})P_n \hat{f} \Vert_{L^2_{x_3,v}}  + \Vert \hat{f} - P_n \hat{f} \Vert_{L^2_{x_3,v}}  + |(I-P_\gamma)\hat{f}|_{L^2_{\gamma_+}}+ \Vert g\Vert_{L^2_{x_3,v}} + |\lambda| \Vert (\mathbf{I}-\mathbf{P})\hat{f}\Vert_{L^2_{x_3,v}}.
\end{align*}

For $\hat{b}_2$ and $\hat{b}_3$, we choose the test functions, respetively, as
\begin{align*}
\begin{cases}
      &\dis  \psi_b = -v_1v_2 \sqrt{\mu} ik_1 \phi_b - \frac{3}{2}\Big( |v_2|^2 -\frac{|v|^2}{3} \Big) \sqrt{\mu} ik_2 \phi_b + v_2 v_3 \sqrt{\mu} \p_{x_3} \phi_b, \\
    &\dis  [|k_1|^2 + 2|k_2|^2 - \p_{x_3}^2]\phi_b = \bar{\hat{b}}_2, \\
    &\dis  \phi_b = 0 \text{ when } x_3 = \pm 1,  
\end{cases}
\end{align*}
and
\begin{align*}
\begin{cases}
        &\dis \psi_b = -v_1 v_3 \sqrt{\mu} i k_1\phi_b - v_2 v_3 \sqrt{\mu} i k_2 \phi_b + \frac{3}{2} \Big( |v_3|^2 - \frac{|v|^2}{3} \Big) \sqrt{\mu} \p_{x_3} \phi_b, \\
    &\dis  [|k_1|^2 + |k_2|^2 - 2\p_{x_3}^2]\phi_b = \bar{\hat{b}}_3, \\
    &\dis  \phi_b = 0 \text{ when } x_3 = \pm 1.
\end{cases}
\end{align*}
Then, the estimates for $\hat{b}_2,\hat{b}_3$ can be done by the same computations.

In summary, we obtain the following estimate for $\hat{\mathbf{b}}$:
\begin{align}
    &     \Vert  \hat{\mathbf{b}}  \Vert_{ L^2_{x_3}} \leq    \Vert (\mathbf{I}-\mathbf{P})P_n \hat{f} \Vert_{L^2_{x_3,v}}  + \Vert \hat{f} - P_n \hat{f} \Vert_{L^2_{x_3,v}}  + |(I-P_\gamma)\hat{f}|_{L^2_{\gamma_+}}+ \Vert g\Vert_{L^2_{x_3,v}}+ |\lambda| \Vert (\mathbf{I}-\mathbf{P})\hat{f}\Vert_{L^2_{x_3,v}}. \notag
\end{align}

\medskip
\textit{Estimate of $\hat{c}$.}

We choose the test function as
\begin{align}
    &\dis \psi_c = (-ik_1v_1 \phi_c - ik_2 v_2 \phi_c + v_3 \p_{x_3}\phi_c)(|v|^2-5)\sqrt{\mu} \perp \ker \mathcal{L}, \notag
\end{align}
with $\phi_c$ satisfying
\begin{align}
\begin{cases}
    & \dis  |k|^2 \phi_c - \p_{x_3}^2 \phi_c = \bar{\hat{c}},  \\
    & \phi_c = 0 \ \text{ when } x_3 = \pm 1.   
\end{cases}\notag
\end{align}
Again due to the Poincar\'e inequality, similar to \eqref{phiaM_l2_k}-\eqref{phiaM_trace_2}, we have the following regularity estimate and trace estimate:
\begin{align}
    & \Vert (1+|k|+|k|^2)\phi_c\Vert^2_{L^2_{x_3}}+ \Vert (1+|k|)\p_{x_3}\phi_c\Vert^2_{L^2_{x_3}} + \Vert \p_{x_3}^2 \phi_c\Vert_{L^2_{x_3}}^2 \lesssim \Vert \hat{c}\Vert^2_{L^2_{x_3}},  \label{phic_l2_2} \\
    &| |k|\phi_c(k,\pm 1) |^2 +  | \p_{x_3}\phi_c(k,\pm 1)|^2 \lesssim \Vert \hat{c}\Vert_{L^2_{x_3}}^2. \label{phic_trace_2}
\end{align}

We first compute
\begin{align*}
    & \eqref{weak_formulation}_1  =  \int_{-1}^1 \int_{\mathbb{R}^3} i (k_1v_1 + k_2v_2) \Big( \hat{a}+\hat{c} \frac{|v|^2-3}{2}\Big) \sqrt{\mu} \psi_c \dd v \dd x_3  \underbrace{+ \int_{-1}^1 \int_{\mathbb{R}^3} i(k_1v_1+k_2v_2) (\mathbf{I}-\mathbf{P})\hat{f} \psi_c \dd v \dd x_3}_{E_5} \\
    &  =  \int_{-1}^1 \int_{\mathbb{R}^3}  (k_1v_1 + k_2v_2) \Big( \hat{a}+\hat{c} \frac{|v|^2-3}{2}\Big) \mu (v_1\phi_c + v_2 \phi_c)(|v|^2-5) \dd v \dd x_3  + E_5 \\
    & =  \int_{-1}^1  5|k|^2 \hat{c}\phi_c \dd x_3  + E_5   .
\end{align*}
In the first equality, the contribution of $\hat{\mathbf{b}}$ has vanished from the oddness. In the second equality, the contribution of $v_3$ in $\psi_c$ has vanished from the oddness. In the third equality, we used $\int_{\mathbb{R}^3} v_i^2 \frac{|v|^2-3}{2}(|v|^2-5)\mu \dd v = 5$, and the contribution of $\hat{a}$ has vanished by the orthogonality:
\begin{align*}
    &   \int_{\mathbb{R}^3} v_i^2 (|v|^2-5)\mu \dd v = 0, \ i=1,2,3.
\end{align*}

Note that $E_5$ corresponds to the contribution of $(\mathbf{I}-\mathbf{P})\hat{f}$. By \eqref{phic_l2_2}, it holds that
\begin{align}
    &   |E_5| \lesssim o(1)[\Vert |k|^2\phi_c\Vert_{L^2_{x_3,v}}^2 + \Vert |k|\p_{x_3}\phi_c \Vert_{L^2_{x_3}}^2] + \Vert (\mathbf{I}-\mathbf{P})\hat{f}\Vert_{L^2_{x_3,v}}^2 \lesssim o(1)\Vert \hat{c}\Vert_{L^2_{x_3}}^2 + \Vert (\mathbf{I}-\mathbf{P})\hat{f}\Vert^2_{L^2_{x_3,v}}. \notag
\end{align}

Next, we compute
\begin{align*}
    &    \eqref{weak_formulation}_2 = - \int_{-1}^1 \int_{\mathbb{R}^3} v_3 \Big(\hat{a} + \hat{c}\frac{|v|^2-3}{2} \Big)\sqrt{\mu} \p_{x_3}\psi_c \dd v \dd x_3   \underbrace{-\int_{-1}^1 \int_{\mathbb{R}^3} v_3 (\mathbf{I}-\mathbf{P})\hat{f}\p_{x_3} \psi_c \dd v \dd x_3}_{E_6} \\
    & = - \int_{-1}^1 \int_{\mathbb{R}^3} v_3^2 \hat{c}\frac{|v|^2-3}{2} (|v|^2-5) \mu \p_{x_3}^2 \phi_c  \dd v \dd x_3  + E_6 = -5 \int_{-1}^1 \hat{c} \p_{x_3}^2 \phi_c \dd x  + E_6.
\end{align*}
Here, by \eqref{phic_l2_2},
\begin{align}
    & |E_6| \lesssim o(1) [\Vert |k|\p_{x_3}\phi_c \Vert^2_{L^2_{x_3}} + \Vert \p_{x_3}^2 \phi_c \Vert^2_{L^2_{x_3}}] + \Vert (\mathbf{I}-\mathbf{P})\hat{f}\Vert^2_{L^2_{x_3,v}} \lesssim o(1) \Vert \hat{c}\Vert_{L^2_{x_3}}^2 + \Vert (\mathbf{I}-\mathbf{P})\hat{f}\Vert^2_{L^2_{x_3,v}}.\notag
\end{align}
Then we have
\begin{align}
    &     \eqref{weak_formulation}_1 + \eqref{weak_formulation}_2 = 5 \int_{-1}^1 [|k|^2  - \p_{x_3}^2]\phi_c \hat{c} \dd x_3  + E_5 + E_6  = 5\Vert \hat{c}\Vert_{L^2_{x_3}}^2 + E_5 + E_6. \notag
\end{align}

Then we compute the boundary term $\eqref{weak_formulation}_3$. At $x_3=1$, for the contribution of $P_\gamma \hat{f}$, we have
\begin{align*}
    &    \int_{\mathbb{R}^3} v_3   P_\gamma \hat{f}(k,1) \psi_c(1) \dd v  \\
    &=   \int_{\mathbb{R}^3} v_3 P_\gamma \hat{f}(k,1) (-ik_1v_1 \phi_c - ik_2 v_2 \phi_c + v_3 \p_{x_3}\phi_c)(|v|^2-5)\sqrt{\mu}  \dd v  = 0.
\end{align*}
Here we have used the oddness and $\int_{\mathbb{R}^3} v_3^2 (|v|^2-5)\mu \dd v = 0$.

For the part with $(I-P_\gamma)\hat{f}$, we derive that
\begin{align}
    & \int_{v_3>0} |(I-P_\gamma)\hat{f}(k,1) (-ik_1v_1 \phi_c - ik_2 v_2 \phi_c + v_3 \p_{x_3}\phi_c)(|v|^2-5)\sqrt{\mu}| \dd v  \notag \\
    & \lesssim o(1) [| |k| \phi_c(k,1) |^2 +  | \p_{x_3} \phi_c(k,1)|^2  ]   +    |(I-P_\gamma)\hat{f}|_{L^2_{\gamma_+}}^2 \notag \\
    & \lesssim o(1) \Vert \hat{c}\Vert_{L^2_{x_3}}^2 + |(I-P_\gamma)\hat{f}|_{L^2_{\gamma_+}}^2.  \notag
\end{align}
In the last line, we have used the trace estimate \eqref{phic_trace_2}.

Similarly, for $x_3=-1$ we have the same estimate. Thus we conclude that
\begin{align}
    &  |\eqref{weak_formulation}_3| \lesssim o(1) \Vert \hat{c}\Vert_{L^2_{x_3}}^2 + |(I-P_\gamma)\hat{f}|_{L^2_{\gamma_+}}^2.   \notag
\end{align}

Next we compute $\eqref{weak_formulation}_4$, $\eqref{weak_formulation}_5$ and $\eqref{weak_formulation}_6$ as
\begin{align}
    &    |\eqref{weak_formulation}_4| \lesssim o(1)[\Vert k \phi_c \Vert_{L^2_{x_3}}^2 + \Vert \p_{x_3} \phi_c \Vert_{L^2_{x_3}}^2 ] +  \Vert (\mathbf{I}-\mathbf{P})P_n \hat{f} \Vert_{L^2_{x_3,v}}^2  + \Vert \hat{f} - P_n \hat{f} \Vert_{L^2_{x_3,v}}^2  \notag \\
    & \lesssim o(1)\Vert \hat{c}\Vert_{L^2_{x_3}}^2  +  \Vert (\mathbf{I}-\mathbf{P})P_n \hat{f} \Vert_{L^2_{x_3,v}}^2  + \Vert \hat{f} - P_n \hat{f} \Vert_{L^2_{x_3,v}}^2,   \notag
\end{align}
\begin{align}
    & |\eqref{weak_formulation}_5| =  \Big|    \int_{-1}^1 \int_{\mathbb{R}^3} g\psi_c \dd v \dd x_3  \Big| \lesssim  o(1)[\Vert k \phi_c \Vert_{L^2_{x_3}}^2 + \Vert \p_{x_3} \phi_c \Vert_{L^2_{x_3}}^2 ] +  \Vert g\Vert_{L^2_{x_3,v}}^2   \lesssim  o(1)\Vert \hat{c}\Vert_{L^2_{x_3}}^2 + \Vert g\Vert_{L^2_{x_3,v}}^2,   \notag
\end{align}
and
\begin{align}
    & |\eqref{weak_formulation}_6| = \Big|\int_{-1}^1 \int_{\mathbb{R}^3} \lambda (\mathbf{I}-\mathbf{P})\hat{f}\psi_c \dd v \dd x_3  \Big|\lesssim  o(1)[\Vert k \phi_c \Vert_{L^2_{x_3}}^2 + \Vert \p_{x_3} \phi_c \Vert_{L^2_{x_3}}^2 ] +  |\lambda|^2\Vert (\mathbf{I}-\mathbf{P})\hat{f}\Vert_{L^2_{x_3,v}}^2   \notag\\
    & \lesssim  o(1)\Vert \hat{c}\Vert_{L^2_{x_3}}^2 +  |\lambda|^2\Vert (\mathbf{I}-\mathbf{P})\hat{f}\Vert_{L^2_{x_3,v}}^2 . \notag
\end{align}

We conclude the estimate of $\hat{c}$: 
\begin{align*}
    &    \Vert \hat{c}\Vert_{L^2_{x_3}} \lesssim   \Vert (\mathbf{I}-\mathbf{P})P_n \hat{f} \Vert_{L^2_{x_3,v}}  + \Vert \hat{f} - P_n \hat{f} \Vert_{L^2_{x_3,v}} + |(I-P_\gamma)\hat{f}|_{L^2_{\gamma_+}}+\Vert g\Vert_{L^2_{x_3,v}} + |\lambda| \Vert (\mathbf{I}-\mathbf{P})\hat{f}\Vert_{L^2_{x_3,v}}.
\end{align*}

Combining the estimates for $\hat{a},\hat{\mathbf{b}},\hat{c}$, we conclude the lemma.
\end{proof}

\section{Linear estimate}\label{sec:linear}
As mentioned in Section \ref{sec:regulariza}, the regularized operator $K_n$ is introduced to recover the compactness on $L^2((-1,1)\times \R^3_v)$. In this section, we focus on the analysis of the resolvent and spectrum of $\hat{B}_n(k)$ in Section \ref{sec:spectral_0} and Section \ref{sec:spectral_neq0}. In Section \ref{sec:linear_asymptotic}, we collect the spectral analysis of $\hat{B}_n(k)$ and derive a crucial semi-group representation of $e^{\hat{B}_n(k)t}$(see Proposition \ref{prop:inverse_laplace}). The decay rate of the semi-group $e^{B_nt}$ then follows through inverse Fourier transform(see Lemma \ref{lemma:linear_decay}). Last, we conclude the decay rate of $e^{Bt}$ in Proposition \ref{prop.Bdecay} by the decay rate of $e^{B_n t}$ with an approximation argument in $n$.

\subsection{Spectral analysis when $k=0$}\label{sec:spectral_0}
Recall \eqref{def_hatBn}. We start from the analysis when $k=0$.

\begin{lemma}\label{lemma:eigenvalue_0}
For $\hat{B}_n(k)$, there is no eigenvalue with a positive real part for any $k\in \R^2$.

Moreover, $\lambda = 0$ is the only eigenvalue on the imaginary axis in the only case of $k=0$, and its eigen-space is $\text{span}\{\sqrt{\mu}\}$.
\end{lemma}

\begin{proof}
We focus on the eigenvalue problem on $\hat{B}_n(k)$:
\begin{align}
\begin{cases}
        & (\sigma+i\tau)\hat{f} + i(k_1v_1+k_2v_2) \hat{f} + v_3 \p_{x_3} \hat{f} + \mathcal{L}_n\hat{f} = 0, \\
        &  \hat{f}|_{\gamma_-} = P_{\gamma}\hat{f}.
\end{cases} \label{eigenvalue_problem}
\end{align}

Suppose $\sigma>0$. Multiplying \eqref{eigenvalue_problem} by $\bar{\hat{f}}$ and taking integration in $x_3,v$, we obtain the real part of the $L^2_{x_3,v}$ energy estimate:
\begin{align*}
    &  \sigma \Vert \hat{f}\Vert_{L^2_{x_3,v}}^2 + |(I-P_\gamma)\hat{f}|_{L^2_{\gamma_+}}^2 + \Vert (\mathbf{I}-\mathbf{P})P_n \hat{f} \Vert_{L^2_{x_3,v}}^2  + \Vert \nu^{1/2} [\hat{f} - P_n \hat{f}]  \Vert_{L^2_{x_3,v}}^2= 0.
\end{align*} 
Here we have applied the coercive estimate of $\mathcal{L}_n$ in Lemma \ref{lemma:Ln}. Since $\sigma>0$, this leads to $\Vert \hat{f}\Vert_{L^2_{x_3,v}}^2 = 0$. Hence, there is no eigenvalue with a positive real part for any $k\in \R^2$. The first part is proved.

Suppose $\sigma=0$. The equality above becomes 
\begin{align*}
    &    | (I-P_\gamma)\hat{f}|_{L^2_{\gamma_+}}^2 +  \Vert (\mathbf{I}-\mathbf{P})P_n \hat{f} \Vert_{L^2_{x_3,v}}^2  + \Vert \nu^{1/2} [\hat{f} - P_n \hat{f}]  \Vert_{L^2_{x_3,v}}^2 = 0.
\end{align*}
We conclude that $(I-P_\gamma)\hat{f} = 0$ on the boundary, and
\begin{align*}
    &(\mathbf{I}-\mathbf{P})P_n \hat{f} = 0, \ \hat{f}(x_3,v) - P_n \hat{f}(x_3,v) = 0.
\end{align*}

This also implies that $(\mathbf{I}-\mathbf{P})\hat{f} = 0$. Thus, $\hat{f}$ is restricted to the following form:
\begin{align*}
    \hat{f}(x_3,v) = C_0(x_3) \sqrt{\mu} + C_1(x_3)v_1 \sqrt{\mu} + C_2(x_3)v_2 \sqrt{\mu} + C_3(x_3) v_3 \sqrt{\mu} + C_4(x_3) \frac{|v|^2-3}{2}\sqrt{\mu}.
\end{align*}
Since $(I-P_\gamma)\hat{f} = 0$ on the boundary, we have $C_1(1) = C_2(1)=C_3(1)=C_4(1) = 0$.

The original eigenvalue problem then becomes
\begin{align*}
    &    i(\tau + k_1v_1+k_2v_2) \Big[C_0(x_3)\sqrt{\mu} + C_1(x_3)v_1 \sqrt{\mu} + C_2(x_3)v_2 \sqrt{\mu} + C_3(x_3) v_3 \sqrt{\mu} + C_4(x_3) \frac{|v|^2-3}{2}\sqrt{\mu} \Big] \\
    &+ v_3 \Big[C_0'(x_3)\sqrt{\mu} + C_1'(x_3)v_1 \sqrt{\mu} + C_2'(x_3)v_2 \sqrt{\mu} + C_3'(x_3) v_3 \sqrt{\mu} + C_4'(x_3) \frac{|v|^2-3}{2}\sqrt{\mu} \Big] = 0.
\end{align*}
Since this holds for any $v$, comparing the coefficient of $v_3 \frac{|v|^2-3}{2}\sqrt{\mu}$, we first have $C_4'(x_3) = 0$ and $C_4(x_3) = 0$. Then we compare the coefficient of $v_3^2 \sqrt{\mu}$ and have $C'_3(x_3) = 0$ and thus $C_3(x_3) = 0$. Then we compare the coefficients of $v_2v_3\sqrt{\mu}$ and $v_1v_3\sqrt{\mu}$ and conclude $C'_1(x_3) = C'_2(x_3) = C_1(x_3) = C_2(x_3) = 0.$ Last we compare the coefficient of $v_3\sqrt{\mu}$ and conclude that $C_0'(x_3) = 0$ and thus $C_0(x_3) = C$, where $C\in \mathbb{R}$ is an arbitrary constant.

Thus the problem reduces to
\begin{align*}
    &    i(\tau + k_1v_1+k_2v_2)C\sqrt{\mu} = 0.
\end{align*}
Again, since this holds for any $v$, we conclude that $k=0$ and $\tau = 0$. This indicates that if $k\neq 0$, any possible eigenvalue must have a strictly negative real part. Moreover, $\lambda=0$ is the only eigenvalue on the imaginary axis only when $k=0$, with the corresponding eigen-space $\text{span}\{\sqrt{\mu}\}$.   We conclude the lemma.
\end{proof}

\hide

\begin{lemma}\label{lemma:continuous_spectrum}
When $k=0$, there exists $\sigma_0 \ll 1$ such that
\begin{align*}
    &   \sigma(\hat{B}_n(0)) \cap \{\sigma+i\tau||\sigma|<\sigma_0\} = \{0\}. 
\end{align*}

\end{lemma}

\begin{proof}
This is due to $K_n$ is a compact operator on $L^2_{x_3,v}$.

\end{proof}

\unhide

\subsection{Spectral analysis when $k\neq 0$}\label{sec:spectral_neq0}

Inspired by the spectral analysis when $k=0$, we expect an analogous behavior of the eigenvalue and eigen-space for $k$ near $0$. We need the following crucial estimate when the imaginary part of $\lambda$ is large enough.

\begin{lemma}\label{lemma:large_tau_inv}

There exists $\tau_0\gg 1$ such that when $|\tau|>\tau_0$ and $-\nu_0<\sigma<0$, then
\begin{align*}
    &   (\sigma+i\tau - \hat{B}_n(k) )^{-1} \text{ exists}.
\end{align*}
Here $\tau_0$ is independent of $n$.

Moreover, for any $-\nu_0<\sigma < 0$, it holds that
\begin{align*}
    &  \lim_{|\tau|\to \infty}\Vert (\sigma+i\tau - \hat{B}_n(k))^{-1}\Vert_{L^2_{x_3,v}} =\lim_{|\tau|\to \infty}\Vert (\sigma+i\tau - \hat{A}(k))^{-1}\Vert_{L^2_{x_3,v}} =0.
\end{align*}

\end{lemma}

\begin{proof}

We first focus on the resolvent estimate to $(\lambda - \hat{A}(k))^{-1}$, where $\lambda = \sigma + i \tau$ and $\sigma>-\nu_0$. By Lemma \ref{lemma:A_resolvent} and the Laplace transform, we have
\begin{align}
    &  (\lambda- \hat{A}(k))^{-1} = \int_0^\infty e^{-\lambda t} e^{t\hat{A}(k)} \dd t. \label{laplace_transform}
\end{align}
This can be written as a Fourier transform
\begin{align*}
    & (\sigma + i\tau - \hat{A}(k))^{-1} = (2\pi)^{-1/2}\int^\infty_{-\infty} e^{-i\tau t} (2\pi)^{1/2}\chi(t)e^{-\sigma t}e^{t\hat{A}(k)} \dd t, \\
    & \chi(t) = 1 \text{ for }t\geq 0 \text{ and }\chi(t) = 0 \text{ for }t<0.
\end{align*}

By the Plancherel theorem, we obtain
\begin{align}
    & \int_{-\infty}^\infty  \Vert (\sigma+i\tau - \hat{A}(k))^{-1}\Vert_{L^2_{x_3,v}}^2 \dd \tau = \int_{-\infty}^\infty \Vert (2\pi)^{1/2} \chi(t)e^{-\sigma t} e^{t\hat{A}(k)}  \Vert_{L^2_{x_3,v}}^2 \dd t \notag\\
    & = 2\pi \int_0^\infty e^{-2\sigma t}\Vert e^{t\hat{A}(k)}\Vert_{L^2_{x_3,v}}^2 \dd t \leq 2\pi \int_0^\infty e^{-2(\sigma+\nu_0)t}\dd t  \lesssim 1. \label{integrable}
 \end{align}
In the second line, we have applied Lemma \ref{lemma:A_resolvent}.
 
Moreover, since $\tau$ becomes the Fourier variable of $t$, we also have
\begin{align*}
    &   \p_\tau (\sigma+i\tau - \hat{A}(k))^{-1} = -(2\pi)^{-1/2} i\int_{-\infty}^\infty e^{-i\tau t} (2\pi)^{1/2} t\chi(t)e^{-\sigma t} e^{t\hat{A}(k)} \dd t.
\end{align*}
Then both estimates lead to
\begin{align*}
    & \int_{-\infty}^\infty \Vert \p_\tau (\sigma+i\tau-\hat{A}(k))^{-1} \Vert_{L^2_{x_3,v}}^2 \dd \tau \leq 2\pi \int_0^\infty e^{-2(\sigma+\nu_0)t}t \dd t  \lesssim 1.
\end{align*}
The estimates above conclude that
\begin{align}
    &  \lim_{|\tau|\to\infty}\Vert (\sigma+i\tau-\hat{A}(k))^{-1}\Vert_{L^2_{x_3,v}} = 0. \label{A_convergence_0}
\end{align}

This implies that for $\delta \ll 1$, there exists $\tau_0\gg 1$ such that when $|\tau|>\tau_0$, for a fixed $\sigma \in (-\nu_0,0)$, it holds that
\begin{align*}
    & \Vert (\sigma+i\tau-\hat{A}(k))^{-1}\Vert_{L^2_{x_3,v}} < \delta.
\end{align*}
For other $\sigma' \in (-\nu_0,0)$, we express the solution $ \hat{f}=(\sigma' + i\tau - \hat{A}(k))^{-1}h$ as
\begin{align*}
    &  (\sigma+i\tau) \hat{f} - \hat{A}(k)\hat{f}  = (\sigma-\sigma')\hat{f} + h.
\end{align*}
With $\delta\ll 1$, this leads to
\begin{align*}
    &    \Vert \hat{f}\Vert_{L^2_{x_3,v}} < \delta (\Vert h\Vert_{L^2_{x_3,v}} + |\nu_0|\Vert \hat{f}\Vert_{L^2_{x_3,v}}), \ \ \ \  \Vert \hat{f}\Vert_{L^2_{x_3,v}} < 2\delta \Vert h\Vert_{L^2_{x_3,v}}.
\end{align*}

Thus for all $\sigma\in (-\nu_0,0)$, when $|\tau|>\tau_0$ it holds that
\begin{align*}
    &   \Vert (\sigma+i\tau -\hat{A}(k))^{-1}\Vert_{L^2_{x_3,v}} < 2\delta.
\end{align*}

Now we turn to the solvability of $(\sigma+i\tau - \hat{B}_n(k))^{-1}$. We rewrite this term using the second resolvent identity as
\begin{align}
    &    (\sigma+ i \tau -\hat{B}_n(k))^{-1}  = (\sigma+ i\tau - \hat{A}(k))^{-1} (I-G(\sigma+i\tau))^{-1} , \notag\\
    & G(\sigma+i\tau) = KP_n(\sigma+i\tau-\hat{A}(k))^{-1}. \label{A_B_related}
\end{align}
Since $K$ is a bounded operator on $L^2((-1,1)\times \R^3_v)$ and
$ \Vert P_n \hat{f}\Vert_{L^2_{x_3,v}} \leq \Vert \hat{f}\Vert_{L^2_{x_3,v}}$, for $|\tau|>\tau_0$ we have
\begin{align*}
    &     \Vert G(\sigma+i\tau) \Vert_{L^2_{x_3,v}} < 2\delta\Vert K\Vert_{L^2_{x_3,v}}.
\end{align*}
Here the upper bound does not depend on $n$.

With $\delta\ll 1$ such that $\Vert G(\sigma+i\tau)\Vert_{L^2_{x_3,v}}< \frac{1}{2}$, this implies 
\begin{align*}
    & (I-G(\sigma+i\tau))^{-1} \text{ exists}.
\end{align*}
Since $(\sigma+i\tau-\hat{A}(k))^{-1}$ exists, when $|\tau|>\tau_0$, we conclude that
\begin{align*}
    & (\sigma + i\tau -\hat{B}_n(k))^{-1} \text{ exists.}
\end{align*}

Finally, we combine \eqref{A_convergence_0} and \eqref{A_B_related} to conclude that
\begin{align}
    &  \lim_{|\tau|\to\infty}\Vert (\sigma+i\tau-\hat{B}_n(k))^{-1}\Vert_{L^2_{x_3,v}} = 0.  \notag
\end{align}
We conclude the lemma.
\end{proof}

Lemma \ref{lemma:large_tau_inv} above implies that the eigenvalue of $\hat{B}_n(k)$ only possibly exists when $|\tau|<\tau_0$. Such an observation is crucial in the investigation of the eigenvalue and eigen-projection when $|k|\neq 0$. In the following proposition as well as its proof, we point out the explicit dependence of eigenvalue, eigen-function and eigen-projection on $n$.

\begin{proposition}\label{prop:eigenvalue}
There exists $0<\sigma_0 = \sigma_0(\tau_0)\ll 1$ such that for some $\kappa>0$ with $|k|<\kappa=\kappa(\sigma_0)\ll \sigma_0$, there is only one eigenvalue to $\hat{B}_n(k)$ on $\mathbb{C}_+(-\sigma_0)$. This unique eigenvalue has the following form:
\begin{align}
    & \lambda_n(k) = -\lambda^*_n|k|^2 + C^n_\lambda|k|^3. \label{eigenvalue}
\end{align}

\begin{enumerate}
    \item  In the leading part of \eqref{eigenvalue}, $\lambda_n^* > 0$ is positive, real, and does not depend on $k$. Moreover, the limit of $\lambda_n^*$ exists, and it is also positive and real:
\begin{align*}
    &   \lim_{n\to \infty} \lambda_n^* = \lambda^* > 0.
\end{align*}

\item In the higher order term of \eqref{eigenvalue}, the coefficient satisfies $|C^n_\lambda|\leq C_\lambda$, where $C_\lambda$ does not depend on $n$.

\item The eigen-function of \eqref{eigenvalue} has the form of $\hat{f}_n = M\sqrt{\mu} + M|k|G_{n,1} + M|k|^2 G_{n,2}.$ Here 
\begin{align*}
    &M = \frac{1}{2}\int_{-1}^1 \int_{\mathbb{R}^3} \hat{f}_n\sqrt{\mu} \dd v \dd x_3, \  \Vert G_{n,1}\Vert_{L^2_{x_3,v}} + \Vert G_{n,2}\Vert_{L^2_{x_3,v}} \lesssim 1.
\end{align*}
Here the upper bound in the inequality does not depend on $n$.

\item Equivalent to (3), the eigen-projection has the form
\begin{align*}
    & \tilde{\mathbf{P}}^n(k) = \mathbf{P}_0 + |k|\mathbf{P}^n_1(k) + |k|^2\mathbf{P}^n_2(k),
\end{align*}
where $\mathbf{P}_0$ is defined in \eqref{P0_def} and
\begin{align*}
    & \Vert \mathbf{P}^n_1(k)\Vert_{L^2_{x_3,v}} + \Vert \mathbf{P}^n_2(k)\Vert_{L^2_{x_3,v}} \lesssim 1.
\end{align*}
Here the upper bound in the inequality does not depend on $n$.
\end{enumerate}

\end{proposition}

\begin{proof}

We analyze the eigenvalue problem
\begin{align*}
\begin{cases}
        &   \lambda_n \hat{f}_n + i(k_1v_1 + k_2v_2 )\hat{f}_n + v_3 \p_{x_3} \hat{f}_n + \mathcal{L}_n\hat{f}_n = 0, \\
        & \hat{f}_n|_{\gamma_-} = P_\gamma \hat{f}_n.
\end{cases}
\end{align*}
From Lemma \ref{lemma:large_tau_inv}, the eigenvalue only possibly exists for $|Im(\lambda_n)|<\tau_0$. Below we let $|Im(\lambda_n)|<\tau_0$.

\medskip
\noindent \textit{Claim}: When $|Re(\lambda_n)|<\sigma_0$ with $\sigma_0$ to be specified later, there exists only a trivial solution when $\frac{1}{2}\int_{-1}^1 \int_{\mathbb{R}^3} \hat{f}_n\sqrt{\mu} \dd v \dd x_3 = 0$, and thus the non-trivial solution only possibly exists for $\frac{1}{2}\int_{-1}^1 \int_{\mathbb{R}^3} \hat{f}_n\sqrt{\mu} \dd v \dd x_3 \neq 0$.

\noindent \textit{Proof of claim:} When $\frac{1}{2}\int_{-1}^1 \int_{\mathbb{R}^3} \hat{f}_n\sqrt{\mu} \dd v \dd x_3 = 0$, it is clear to see $0\in D(\hat{B}_n(k))$ is a solution. Suppose $\hat{f}_n\in D(\hat{B}_n(k))$ is another solution, by Proposition \ref{prop:macroscopic}, we have $\hat{a}_M=0$, and thus the following macroscopic estimate holds:
\begin{align*}
    & \Vert \mathbf{P}\hat{f}_n\Vert_{L^2_{x_3,v}} \lesssim  (\nu_0+\tau_0) [\Vert (\mathbf{I}-\mathbf{P})P_n \hat{f}_n \Vert_{L^2_{x_3,v}}  + \Vert \hat{f}_n - P_n \hat{f}_n \Vert_{L^2_{x_3,v}} \\
    &+ |(I-P_\gamma)\hat{f}_n
    |_{L^2_{\gamma_+}}   + (\nu_0+\tau_0)\Vert (\mathbf{I}-\mathbf{P})\hat{f}_n\Vert_{L^2_{x_3,v}}] . 
\end{align*}
The real part of the $L^2_{x_3,v}$ energy estimate with Lemma \ref{lemma:Ln} leads to
\begin{align*}
    &     \Vert (\mathbf{I}-\mathbf{P})P_n \hat{f}_n \Vert_{L^2_{x_3,\nu}}  + \Vert \hat{f}_n - P_n \hat{f}_n \Vert_{L^2_{x_3,\nu}} + |(I-P_\gamma)\hat{f}_n|_{L^2_{\gamma_+}} \leq \sqrt{|Re(\lambda_n)|} \Vert \hat{f}_n\Vert_{L^2_{x_3,v}} . 
\end{align*}

Now we choose $|Re(\lambda_n)|<\sigma_0= \frac{c^2}{\nu_0^4 + \tau_0^4}$ with $c\ll 1$. Multiplying the first estimate by $\frac{\sqrt{c}}{\nu_0^2 + \tau_0^2}$ and adding it to the second estimate, we obtain that
\begin{align*}
    &     \frac{\sqrt{c}}{\nu_0^2 + \tau_0^2} \Vert \hat{f}_n\Vert_{L^2_{x_3,v}} \lesssim \frac{c}{\nu_0^2+\tau_0^2}\Vert \hat{f}_n\Vert_{L^2_{x_3,v}}.
\end{align*}
By taking $c\ll 1$, we conclude that $\hat{f}_n=0$, and there is no non-trivial solution.
\qed

\medskip
Therefore, to construct a non-trivial solution, we let $\frac{1}{2}\int_{-1}^1 \int_{\mathbb{R}^3} \hat{f}_n\sqrt{\mu} \dd v \dd x_3 = M\neq 0$. Such a solution, if exists, must be unique, since for any two solutions $\hat{f}_n,\hat{g}_n\in D(\hat{B}_n(k))$ with $\frac{1}{2}\int_{-1}^1 \int_{\mathbb{R}^3} \hat{f}_n\sqrt{\mu} \dd v \dd x_3 =\frac{1}{2} \int_{-1}^1 \int_{\mathbb{R}^3} \hat{g}_n\sqrt{\mu} \dd v \dd x_3=M$, the difference $\hat{f}_n - \hat{g}_n$ satisfies the same equation with the condition
\begin{align*}
    &    \frac{1}{2}\int_{-1}^1 \int_{\mathbb{R}^3} (\hat{f}_n-\hat{g}_n)\sqrt{\mu} \dd v \dd x_3 = 0.
\end{align*}
This implies $\hat{f}_n = \hat{g}_n$.

To find the eigenvalue in $\mathbb{C}_+(-\sigma_0)$, we let $|k|<\kappa\ll \sigma_0\ll 1$, where $\kappa$ will be specified at the end of the proof. To construct a nontrivial $\hat{f}_n\in D(\hat{B}_n(k))$ such that $\frac{1}{2}\int_{-1}^1\int_{\mathbb{R}^3} \hat{f}_n\sqrt{\mu} \dd v \dd x_3 = M,$ we seek for a solution in the form 
\begin{align*}
    \hat{f}_n = M\sqrt{\mu} + M|k|G_{n,1} + M|k|^2 G_{n,2},
\end{align*}
with
\begin{align*}
  \frac{1}{2}  \int_{-1}^1 \int_{\mathbb{R}^3} G_{n,1}\sqrt{\mu} \dd v \dd x_3 = 0, \  \frac{1}{2}\int_{-1}^1 \int_{\mathbb{R}^3} G_{n,2}\sqrt{\mu} \dd v \dd x_3 = 0 .
\end{align*}

We set the eigenvalue as $\lambda_n(k) = -\eta_n|k|^2$. Below we solve $G_{n,1}$ using the order of $|k|$, and solve $G_{n,2}$ using the remaining terms.

\begin{itemize}
    \item \textit{Solution to $G_{n,1}$}
\end{itemize}

 We let $G_{n,1}$ satisfy the equation given by the order of $|k|$ in the expansion:
\begin{align}
\begin{cases}
       &    i\frac{(k_1v_1 + k_2 v_2)}{|k|} \sqrt{\mu} + v_3 \p_{x_3} G_{n,1} + \mathcal{L}_n G_{n,1} = 0, \\
    & G_{n,1}|_{\gamma_-} = P_\gamma G_{n,1}. 
\end{cases} \label{Gn1_eqn}
\end{align}
From the oddness we conclude that $G_{n,1}(v_1,v_2) = -G_{n,1}(-v_1,-v_2)$. This implies that the boundary condition becomes $G_{n,1}|_{\gamma_-} = 0$. To get rid of the dependency on $k$ to the estimate of $G_{n,1}$, we introduce a rotation of coordinates in $(v_1,v_2)\in \mathbb{R}^2:$
\begin{align*}
  &\tilde{v}_1 = \frac{k_1v_1 + k_2v_2}{|k|}, \ \tilde{v}_2 = \frac{-k_2 v_1 + k_1 v_2}{|k|}, \\
  & v_1 = \frac{k_1\tilde{v}_1 - k_2\tilde{v}_2}{|k|}, \ v_2 = \frac{k_2\tilde{v}_1 + k_1\tilde{v}_2}{|k|}.
\end{align*}
 Define $\tilde{G}_{n,1}(\tilde{v}_1,\tilde{v}_2)=G_{n,1}\big(\frac{k_1\tilde{v}_1 - k_2\tilde{v}_2}{|k|},  \frac{k_2\tilde{v}_1 + k_1\tilde{v}_2}{|k|}\big),$ due to the rotational invariance of $\mathcal{L}_n$, $\tilde{G}_{n,1}(x_3,\tilde{v}_1,\tilde{v}_2,v_3)$ satisfies
\begin{align}
\begin{cases}
        &  i \tilde{v}_1 \sqrt{\mu(\tilde{v}_1,\tilde{v}_2,v_3)} + v_3\p_{x_3}\tilde{G}_{n,1} + \mathcal{L}_n \tilde{G}_{n,1}=0 \\
        & \tilde{G}_{n,1}|_{\gamma_-} = 0.
\end{cases}\label{G_tilde_eqn}
\end{align}
Clearly, $\tilde{G}_{n,1}$ is independent of $k$.

The existence and uniqueness for $G_{n,1}$ in \eqref{Gn1_eqn} is standard with $\int_{-1}^1 \int_{\mathbb{R}^3} G_{n,1}\sqrt{\mu} \dd v \dd x_3 = 0$(see \cite{EGKM}). From the uniqueness, we conclude that $G_{n,1}$ is pure imaginary, i.e, $Re(G_{n,1}) = 0$. Then we only compute the a priori estimate. 

The real part of the $L^2_{x_3,v}$ energy estimate leads to
\begin{align*}
    &   \Vert (\mathbf{I}-\mathbf{P})P_n G_{n,1}\Vert_{L^2_{x_3,\nu}}^2 + \Vert (I-P_n)G_{n,1}\Vert^2_{L^2_{x_3,\nu}} + |G_{n,1}|_{L^2_{\gamma_+}}^2 \lesssim o(1)\Vert G_{n,1}\Vert_{L^2_{x_3,v}}^2 + \Vert |v|\mu^{1/4}\Vert_{L^2_{x_3,v}}^2 \\
    & \lesssim o(1)\Vert G_{n,1}\Vert^2_{L^2_{x_3,v}} + 1.
\end{align*}
The macroscopic estimate from Proposition \ref{prop:macroscopic} leads to 
\begin{align*}
    &   \Vert \mathbf{P}G_{n,1}\Vert_{L^2_{x_3,v}} \lesssim \Vert (\mathbf{I}-\mathbf{P})P_nG_{n,1}\Vert_{L^2_{x_3,\nu}} + \Vert (I-P_n)G_{n,1}\Vert_{L^2_{x_3,\nu}} + |(I-P_{\gamma})G_{n,1}|_{L^2_{\gamma_+}} + 1.
\end{align*}
Here we emphasize that the coefficients in the inequality do not depend on $n$. Thus, combining the two estimates above, we conclude that
\begin{align*}
    &   \Vert G_{n,1}\Vert_{L^2_{x_3,\nu}}\lesssim 1, \ \int_{-1}^1 \int_{\mathbb{R}^3}     G_{n,1} \sqrt{\mu} \dd v \dd x_3 = 0.
\end{align*}
The upper bound does not depend on $n$.

This leads to the following estimate
\begin{align*}
    &  \int_{-1}^1 \int_{\mathbb{R}^3}  i \frac{k_1 v_1 + k_2 v_2}{|k|}\sqrt{\mu}  \bar{G}_{n,1} \dd v \dd x_3 = -|G_{n,1}|_{L^2_{\gamma_+}}^2 - \int_{-1}^1\int_{\mathbb{R}^3}\mathcal{L}_n G_{n,1}\bar{G}_{n,1}\dd v\dd x_3 .
\end{align*}
Taking the complex conjugate leads to
\begin{align}
    &  \int_{-1}^1 \int_{\mathbb{R}^3}  i \frac{k_1 v_1 + k_2 v_2}{|k|}\sqrt{\mu}  G_{n,1} \dd v \d x  = |G_{n,1}|_{L^2_{\gamma_+}}^2 + \int_{-1}^1\int_{\mathbb{R}^3}\mathcal{L}_n G_{n,1}\bar{G}_{n,1}\dd v\dd x_3  \notag \\
    &= \int_{-1}^1\int_{\mathbb{R}^3} i\tilde{v}_1 \sqrt{\mu} \tilde{G}_{n,1}(x_3,\tilde{v}_1,\tilde{v}_2,v_3) \dd \tilde{v}_1 \tilde{v}_2 \dd v_3 \dd x_3,\label{eta_G_tilde}
\end{align}
where we have applied the change of variable $(v_1,v_2)\to (\tilde{v}_1,\tilde{v}_2)$ with Jacobian $1$ in the second equality.

From \eqref{G_tilde_eqn}, we conclude that \eqref{eta_G_tilde} is independent of $k$.

\begin{itemize}
    \item \textit{Solution to $G_{n,2}$}
\end{itemize}

The equation of the $G_{n,2}$ is given by the remaining terms in the expansion:
\begin{align}
\begin{cases}
       &     -\eta_n  \sqrt{\mu} - \eta_n |k| G_{n,1} + i\frac{(k_1v_1+k_2v_2)}{|k|}G_{n,1} + i(k_1v_1+k_2v_2)G_{n,2} - \eta_n |k|^2 G_{n,2} + v_3 \p_{x_3} G_{n,2} + \mathcal{L}_n G_{n,2} = 0 ,\\
    & G_{n,2}|_{\gamma_-} = P_\gamma G_{n,2}. 
\end{cases}\label{Gn2_eqn}
\end{align}

We will apply the contraction mapping to prove the well-posedness of $G_{n,2}$. First, we denote $\eta_n = \lambda^*_n+ \gamma_n$, where $\lambda^*_n$ satisfies
\begin{align}
    & \int_{-1}^1\int_{\mathbb{R}^3}\Big[-\lambda^*_n \mu +i \frac{(k_1v_1+k_2v_2)}{|k|}G_{n,1}\sqrt{\mu} \Big] \dd v \dd x_3 =   0, \notag \\
    & \lambda_n^* = \frac{i}{2} \int_{-1}^1\int_{\mathbb{R}^3} \frac{k_1v_1+k_2v_2}{|k|}G_{n,1}\sqrt{\mu} \dd v \dd x_3. \label{lambda_n}
\end{align}
Then we denote
\begin{align*}
    & -g= -\lambda^*_n \sqrt{\mu}  + i\frac{k_1v_1+k_2v_2}{|k|}G_{n,1}, \ \int_{-1}^1\int_{\mathbb{R}^3} g\sqrt{\mu}\dd v \dd x_3 = 0.
\end{align*}
The problem \eqref{Gn2_eqn} is equivalent to
\begin{align*}
\begin{cases}
        &  -\gamma_n \sqrt{\mu} + i(k_1v_1 + k_2v_2)G_{n,2} - (\lambda^*_n+\gamma_n)|k|^2 G_{n,2} + v_3 \p_{x_3} G_{n,2} + \mathcal{L}_n G_{n,2} = g + (\lambda^*_n + \gamma_n)|k|G_{n,1}, \\
        &\int_{-1}^1 \int_{\mathbb{R}^3} g\sqrt{\mu} \dd v \dd x_3 = 0, \\
        & G_{n,2}|_{\gamma_-} = P_\gamma G_{n,2}. 
\end{cases}
\end{align*}

Since $G_{n,2}$ needs to satisfy the condition $\int_{-1}^1 \int_{\mathbb{R}^3} G_{n,2} \sqrt{\mu}\dd v  \dd x_3 =0$, we set 
\begin{align}
    & \gamma_n = \gamma_n(G_{n,2}): = \frac{1}{2}\int_{-1}^1 \int_{\mathbb{R}^3} i(k_1 v_1 + k_2 v_2) G_{n,2} \sqrt{\mu} \dd v \dd x_3. \label{gamma_n}
\end{align}

Let $\Lambda$ be a number such that
\begin{align*}
    & \max\{1,\lambda_n^*,\Vert \nu^{1/2}G_{n,1}\Vert_{L^2_{x_3,v}} \} < \Lambda.
\end{align*}

We prove the well-posedness through the Banach fixed-point theorem. For this purpose, we define the following Banach space
\begin{align}
    & \mathcal{X}:= \Big\{\Vert \nu^{1/2} G_{n,2}\Vert_{L^2_{x_3,v}} < C_1 \Lambda, \ \int_{-1}^1 \int_{\mathbb{R}^3} G_{n,2} \sqrt{\mu} \dd v \dd x_3 = 0 \Big\}, \label{banach_space}\\
    & \Vert G_{n,2}\Vert_{\mathcal{X}} : = \Vert \nu^{1/2}G_{n,2} \Vert_{L^2_{x_3,v}}. \notag
\end{align}
Here $C_1$ will be specified later.

We let $\tilde{G}_{n,2} \in \mathcal{X}$, then from \eqref{gamma_n}, it holds that 
\begin{align*}
    & \gamma_n(\tilde{G}_{n,2}) \lesssim |k|\Vert \nu^{1/2}\tilde{G}_{n,2}\Vert_{L^2_{x_3,v}} < o(1)\Lambda. 
\end{align*}

We consider the following system:
\begin{align*}
\begin{cases}
       &   v_3 \p_{x_3} G_{n,2} + \mathcal{L}_n G_{n,2} \\
       &=  g + (\lambda^*_n+\gamma_n(\tilde{G}_{n,2}))|k|G_{n,1} + \gamma_n(\tilde{G}_{n,2})\sqrt{\mu} - i(k_1v_1 + k_2v_2)\tilde{G}_{n,2} + (\lambda^*_n + \gamma_n(\tilde{G}_{n,2}))|k|^2 \tilde{G}_{n,2} ,\\
    & G_{n,2}|_{\gamma_-} = P_\gamma G_{n,2}.  
\end{cases}
\end{align*}
Denote the RHS as $h$, which satisfies
\begin{align*}
 \Vert \nu^{-1/2}h\Vert_{L^2_{x_3,v}}   &     \lesssim \lambda^*_n + \Vert \nu^{1/2}G_{n,1}\Vert_{L^2_{x_3,v}} + \Vert \nu^{1/2}\tilde{G}_{n,2}\Vert_{L^2_{x_3,v}} + \Lambda \Vert G_{n,1}\Vert_{L^2_{x_3,v}} \\
    &+ o(1)\Lambda + |k|\Vert \nu^{1/2}\tilde{G}_{n,2}\Vert_{L^2_{x_3,v}} + \Lambda|k|^2 \Vert \tilde{G}_{n,2}\Vert_{L^2_{x_3,v}} \lesssim \Lambda^2.
\end{align*}

The problem above is equivalent to a one-dimensional stationary problem in a bounded domain:
\begin{align*}
\begin{cases}
        &   v_3 \p_{x_3} G_{n,2} + \mathcal{L}_n G_{n,2} = h, \ \Vert \nu^{-1/2}h \Vert_{L^2_{x_3,v}}<\infty, \ \int_{-1}^1 \int_{\mathbb{R}^3} h \sqrt{\mu}\dd v \dd x_3 = 0, \\
    &  G_{n,2}|_{\gamma_-} = P_{\gamma}G_{n,2}.
\end{cases}
\end{align*}
The existence and uniqueness of this system is standard, and by the same argument to $G_{n,1}$, we can obtain the following estimate: $\int_{-1}^1 \int_{\mathbb{R}^3} G_{n,2} \sqrt{\mu}\dd v \dd x_3 = 0$, and
\begin{align}
    &   \Vert G_{n,2}\Vert_{L^2_{x_3,\nu}} + |(I-P_\gamma)G_{n,2}|_{L^2_{\gamma_+}} \lesssim \Vert \nu^{-1/2}h\Vert_{L^2_{x_3,v}}. \label{G2_bdd_h}
\end{align}
Here the coefficients in the upper bound do not depend on $n$.

This implies that there exists $C_0>0$ such that
\begin{align*}
    & \Vert G_{n,2}\Vert_{\mathcal{X}} < C_0 \Lambda^2.
\end{align*}
Now we choose $C_1$ in \eqref{banach_space} as $C_1:=C_0\Lambda$, and thus we verify that $G_{n,2}\in \mathcal{X}$.

To prove the contraction mapping property, we let $\tilde{H}_1,\tilde{H}_2\in \mathcal{X}$, and let $H_i,i=1,2$ satisfy the following equations:
\begin{align*}
\begin{cases}
       &   v_3 \p_{x_3} H_i + \mathcal{L}_n H_i =  g + (\lambda_n^*+\gamma_n(\tilde{H}_i))G_{n,1} + \gamma_n(\tilde{H}_i)\sqrt{\mu} - i(k_1v_1 + k_2v_2)\tilde{H}_i + (\lambda_n^* + \gamma_n(\tilde{H}_i))|k|^2 \tilde{H}_i \\
    & H_i|_{\gamma_-} = P_\gamma H_i.  
\end{cases}
\end{align*}
Denote $H= H_1-H_2$ and $\tilde{H} = \tilde{H}_1-\tilde{H}_2$, then $H$ satisfies
\begin{align*}
\begin{cases}
       &   v_3 \p_{x_3} H + \mathcal{L}_n H =  \gamma_n(\tilde{H})G_{n,1} + \gamma_n(\tilde{H})\sqrt{\mu} - i(k_1v_1 + k_2v_2)\tilde{H} + \lambda^*_n |k|^2 \tilde{H} + |k|^2 [\gamma_n(\tilde{H})\tilde{H}_1 +  \gamma_n(\tilde{H}_2)\tilde{H}] \\
    & H|_{\gamma_-} = P_\gamma H.  
\end{cases}
\end{align*}

For $\gamma_n(\tilde{H})$, from \eqref{gamma_n} we have
\begin{align*}
    & \gamma_n(\tilde{H}) \lesssim |k|\Vert \nu^{1/2}\tilde{H}\Vert_{L^2_{x_3,v}} = |k|\Vert \tilde{H}\Vert_\mathcal{X}.
\end{align*}
Applying this estimate and \eqref{G2_bdd_h}, we derive
\begin{align*}
    &   \Vert H\Vert_{\mathcal{X}} \lesssim |k|\Lambda \Vert \tilde{H}\Vert_\mathcal{X} + |k|\Vert \tilde{H}\Vert_{\mathcal{X}} + \Lambda |k|^2 \Vert \tilde{H}\Vert_\mathcal{X} + |k|^3 C_1 \Lambda\Vert \tilde{H}\Vert_{\mathcal{X}} <\frac{1}{2}\Vert \tilde{H}\Vert_\mathcal{X}.
\end{align*}
Here we have used $|k|\ll 1$ such that $|k|\Lambda + \Lambda |k|^2 + |k|^3 C_1 \Lambda \ll 1$. By fixed-point theorem, we verify a unique $G_{n,2}\in \mathcal{X}$ to \eqref{Gn2_eqn} with $\gamma_n(G_{n,2})$ given in \eqref{gamma_n}.

\begin{itemize}
    \item \textit{Choice, estimate and uniqueness of the eigenvalue $\lambda_n(k)= -\eta_n |k|^2$}
\end{itemize}

Collecting \eqref{lambda_n} and \eqref{gamma_n}, the eigenvalue $\eta_n$ is chosen to be 
\begin{align*}
 & \eta_n  := \frac{1}{2}\int_{-1}^1 \int_{\mathbb{R}^3} i \frac{(k_1v_1 +k_2v_2)}{|k|}G_{n,1} \sqrt{\mu} \dd v \dd x_3 + \frac{1}{2}\int_{-1}^1 \int_{\mathbb{R}^3} i(k_1v_1 + k_2v_2)G_{n,2} \sqrt{\mu}\dd v \dd x_3  \\  
 & =\frac{1}{2}|G_{n,1}|_{L^2_{\gamma_+}}^2 + \frac{1}{2}\int_{-1}^1 \int_{\mathbb{R}^3} \mathcal{L}_n G_{n,1}\bar{G}_{n,1} \dd v \dd x_3  +\frac{1}{2} \int_{-1}^1\int_{\mathbb{R}^3} i(k_1v_1 + k_2v_2)G_{n,2}\sqrt{\mu} \dd v \dd x_3.
\end{align*}
Such a choice guarantees the normalized property. We obtain the following form of the eigenvalue $\lambda_n(k) = -\eta_n  |k|^2$:
\begin{align*}
    &  \lambda_n(k) = -\lambda^*_n|k|^2  + O(|k|^3).
\end{align*}
Here $\lambda_n^* \in \mathbb{R}^+$ is positive, real, and does not depend on $k$ due to \eqref{eta_G_tilde}:
\begin{align*}
    &\lambda^*_n = \frac{1}{2} |G_{n,1}|^2_{L^2_{\gamma_+}} +\frac{1}{2} \int_{-1}^1 \int_{\mathbb{R}^3} \mathcal{L}_n G_{n,1} \bar{G}_{n,1} \dd v\dd x_3  > 0.
\end{align*}

On the other hand, when $\lambda_n(k) \neq -\eta_n|k|^2$, since $G_{n,1}$ is fixed, from the equation of $G_{n,2}$ in \eqref{Gn2_eqn},  $G_{n,2}$ no longer satisfies the normalized condition $\int_{-1}^1 \int_{\mathbb{R}^3} G_{n,2} \sqrt{\mu} \dd v \dd x_3 = 0$. The contraction arises from the assumption that $\frac{1}{2}\int_{-1}^1\int_{\mathbb{R}^3} \hat{f}_n \sqrt{\mu}\dd v \dd x_3 = M$. Therefore, there does not exist a nontrivial solution to the eigenvalue problem when $\lambda_n(k) \neq -\eta_n|k|^2$. This concludes the uniqueness of the eigenvalue.

Now we analyze the limit $\lim_{n\to \infty} \lambda_n^*$. Let $G_1$ satisfy
\begin{align}
\begin{cases}
       &    i\frac{(k_1v_1 + k_2 v_2)}{|k|} \sqrt{\mu} + v_3 \p_{x_3} G_1 + \mathcal{L} G_1 = 0, \\
    & G_1|_{\gamma_-} = 0. 
\end{cases} \label{G1_eqn}
\end{align}
Similar to $G_{n,1}$ in \eqref{Gn1_eqn}, the existence and uniqueness of solution to $G_1$ is standard.

We define 
\begin{align}\label{G1_eqn_l*}
    &  \lambda^* = \frac{1}{2} |G_1|^2_{L^2_{\gamma_+}} + \frac{1}{2}\int_{-1}^1 \int_{\mathbb{R}^3} \mathcal{L}G_1 \bar{G}_1 \dd v \dd x_3 > 0.
\end{align}
By introducing a similar rotation of coordinates as the argument for $\lambda^*_n$ in \eqref{eta_G_tilde}, $\lambda^*$ does not depend on $k$.

The difference $G_{n,1}-G_1$ satisfies the equation
\begin{align*}
    \begin{cases}
       &     v_3 \p_{x_3} (G_{n,1}-G_1) + \mathcal{L}(G_{n,1} - G_1) = (\mathcal{L}-\mathcal{L}_n)G_{n,1} = K(I-P_n)G_{n,1}, \\
    & [G_{n,1}-G_1]|_{\gamma_-} =0. 
\end{cases}
\end{align*}
Applying a similar argument, we have
\begin{align*}
    & |G_{n,1}-G_1|_{L^2_{\gamma_+}} + \Vert G_{n,1}-G_1\Vert_{L^2_{x_3,\nu}} \lesssim \Vert K(I-P_n)G_{n,1}\Vert_{L^2_{x_3,v}} \lesssim \Vert (I-P_n)G_{n,1}\Vert_{L^2_{x_3,v}}.
\end{align*}
Let $n\to \infty$, we conclude that
\begin{align*}
    & \lim_{n\to\infty}\Vert G_{n,1}-G_1\Vert_{L^2_{x_3,\nu}} = \lim_{n\to\infty} |G_{n,1}-G_1|_{L^2_{\gamma_+}} = 0.
\end{align*}
Thus, the difference of $\lambda_n^*$ and $\lambda^*$ can be estimated as
\begin{align*}
    & |\lambda^*_n - \lambda^*| \leq \int_{\gamma_+} |G_{n,1}-G_1| |G_{n,1}+G_1| |v_3| \dd v \\
    &+ \int_{-1}^1 \int_{\mathbb{R}^3}   \Big|\mathcal{L}_nG_{n,1} (\bar{G}_{n,1}-\bar{G}_1) + \bar{G}_1\mathcal{L}_n(G_{n,1}-G_1) + \bar{G}_1 K(P_n-I)G_1 \Big| \dd v \dd x_3 \\
    & \lesssim  \big[|G_{n,1}-G_1|_{L^2_{\gamma_+}} + \Vert G_{n,1}-G_1\Vert_{L^2_{x_3,\nu}} + \Vert (I-P_n)G_1\Vert_{L^2_{x_3,\nu}} \big] \big[\Vert G_1\Vert_{L^2_{x_3,\nu}} + \Vert G_{n,1}\Vert_{L^2_{x_3,\nu}} \big] .
\end{align*}
When $n\to \infty$, we conclude that $\lim_{n\to\infty}\lambda_n^* = \lambda^*$.

The rest terms in $\lambda_n(k)$ are of $O(|k|^3)$, and the upper bound of the coefficient of $|k|^3$ does not depend on $n$. This is due to
\begin{align*}
    &   \Big|\int_{-1}^1\int_{\mathbb{R}^3} i(k_1v_1+k_2v_2)G_{n,2}\sqrt{\mu} \dd v \dd x_3  \Big| \lesssim |k|\Vert G_{n,2}\Vert_{L^2_{x_3,v}} \lesssim C_1 \Lambda|k|.
\end{align*}
Here we have applied that $G_{n,2}\in \mathcal{X}$ in \eqref{banach_space}, and the upper bound does not depend on $n$. This concludes that the eigenvalue has the desired form in \eqref{eigenvalue}.

Finally, we choose $\kappa=\kappa(\sigma_0)\ll \sigma_0$ such that 
\begin{align*}
    & |\lambda_n(k)| < |\lambda_n^*|\kappa^2 + C_\lambda^n |\kappa|^3 < \sigma_0,
\end{align*}
where the inequality holds for all $n$ due to the uniform in $n$ control of $\lambda_n^*$ and $C_\lambda^n$. This concludes that for all $n$, $\lambda_n(k)\in \mathbb{C}_+(-\sigma_0)$. We then conclude the lemma.
\end{proof}

Proposition \ref{prop:eigenvalue} above leads to the following crucial property of the spectrum.

\begin{lemma}\label{lemma:kneq0_continuous}
Recall $\tau_0,\sigma_0,\kappa$ defined in Lemma \ref{lemma:large_tau_inv} and Proposition \ref{prop:eigenvalue}. For $0<|k|\leq \kappa\ll 1$, it holds that
\begin{align*}
    &   \Sigma(\hat{B}_n(k)) \cap \{\sigma + i\tau | |\sigma|<\sigma_0\} = \{\lambda_n(k)\}.
\end{align*}
In particular, for $\sigma = -\sigma_0$, with $|k|< \kappa \ll \sigma_0$ and $|\tau|<\tau_0$, it holds that
\begin{align}
    &   \Vert (-\sigma_0+i\tau - \hat{B}_n(k))^{-1} \Vert_{L^2_{x_3,v}} \lesssim \frac{1}{\sigma_0^{3/2}}.  \label{inverse_norm}
\end{align}
Here the upper bound does not depend on $n$.

\end{lemma}

\begin{proof}
The first statement is due to Proposition \ref{prop:eigenvalue} and the fact that $K_n $ is compact on $L^2((-1,1)\times \R^3_v)$ by Lemma \ref{lemma:compact}.

For \eqref{inverse_norm}, we focus on the resolvent problem
\begin{align*}
\begin{cases}
        &   (-\sigma_0+i\tau)\hat{f} + i(k_1v_1+k_2v_2)\hat{f} + v_3 \p_{x_3}\hat{f} + \mathcal{L}_n \hat{f} = g , \\
    & \hat{f}|_{\gamma_-} = P_\gamma \hat{f}.
\end{cases}
\end{align*}

First, we multiply the equation by $\sqrt{\mu}$ and take integration in $x_3,v$. Since the diffuse boundary condition preserves mass, with the orthogonality in Lemma \ref{lemma:int_Ln}, we obtain the average density as
\begin{align*}
    &  |\hat{a}_M| \lesssim o(1)\Vert \hat{\mathbf{b}}\Vert_{L^2_{x_3}} + \Vert g/|\sigma_0|\Vert_{L^2_{x_3,v}}, 
\end{align*}
where we have used $|k|<\kappa\ll \sigma_0$.

Next we compute the real part of the $L^2_{x_3,v}$ energy estimate
\begin{align*}
    & \Vert (\mathbf{I}-\mathbf{P})P_n \hat{f} \Vert_{L^2_{x_3,\nu}} + \Vert \hat{f} - P_n \hat{f}\Vert_{L^2_{x_3,\nu}} + |(I-P_\gamma)\hat{f}|_{L^2_{\gamma_+}} \\
    &\leq 2|\sigma_0|\Vert \hat{f}\Vert_{L^2_{x_3,v}} + \Vert g/|\sigma_0|\Vert_{L^2_{x_3,v}}.
\end{align*}
Then we apply Proposition \ref{prop:macroscopic} to derive the macroscopic estimate
\begin{align*}
    &   \Vert \mathbf{P}\hat{f}\Vert_{L^2_{x_3,v}} \lesssim \Vert \hat{\mathbf{b}}\Vert_{L^2_{x_3}} + \Vert \hat{c}\Vert_{L^2_{x_3}} + \Vert \hat{a} - \hat{a}_M\Vert_{L^2_{x_3}} + |\hat{a}_M| \\
    & \lesssim  \tau_0 \Vert (\mathbf{I}-\mathbf{P})P_n\hat{f}\Vert_{L^2_{x_3,\nu}} + \tau_0\Vert \hat{f} - P_n\hat{f}\Vert_{L^2_{x_3,\nu}}\\
    &+ \tau_0|(I-P_\gamma)\hat{f}|_{L^2_{\gamma_+}} + (\tau_0+\sigma_0^{-1})\Vert g\Vert_{L^2_{x_3,v}} + \tau_0^2 \Vert (\mathbf{I}-\mathbf{P})\hat{f}\Vert_{L^2_{x_3,v}} + o(1)\Vert \hat{\mathbf{b}}\Vert_{L^2_{x_3}}.
\end{align*}

Recall that we have chosen $|\sigma_0| = \frac{c_0^2}{\nu_0^4+\tau_0^4}$ with $c_0\ll 1$ in the proof of Proposition \ref{prop:eigenvalue}. Following the same computation, we conclude that
\begin{align*}
    & \frac{\sqrt{c_0}}{\nu_0^2 + \tau_0^2} \Vert \hat{f}\Vert_{L^2_{x_3,v}} \lesssim (\tau_0 + \sigma_0^{-1}) \Vert g\Vert_{L^2_{x_3,v}} \lesssim \sigma_0^{-1}\Vert g\Vert_{L^2_{x_3,v}}.
\end{align*}
We finally obtain that
\begin{align*}
    & \Vert \hat{f}\Vert_{L^2_{x_3,v}} \lesssim \frac{1}{\sigma_0^{3/2}} \Vert g\Vert_{L^2_{x_3,v}}.
\end{align*}
Here the upper only depends on $\sigma_0$, which is independent of $n$.

This concludes \eqref{inverse_norm}.
\end{proof}

\begin{remark}
    Such computation does not work when $|k|$ is not small. We have to deal with such a case using a direct energy estimate with an exponential time rate in the next lemma.
\end{remark}

\begin{lemma}\label{lemma:k_lower_bdd}
When $|k|>\kappa$, there exists $0<c_0< \kappa$ such that
\begin{align*}
    &   \Vert \mathbf{1}_{|k|>\kappa} e^{\hat{B}_n(k)t} \hat{f}_0\Vert_{L^2_{x_3,v}} \lesssim e^{-c_0 t} \Vert \hat{f}_0\Vert_{L^2_{x_3,v}}.  
\end{align*}
Here, the upper bound and $c_0$ do not depend on $n$.
\end{lemma}

\begin{proof}
The equation of $e^{c_0t}\hat{f}$ is given by
\begin{align*}
\begin{cases}
        &    \p_t (e^{c_0t}\hat{f}) + i(k_1 v_1 + k_2 v_2) e^{c_0t}\hat{f} + v_3 \p_{x_3} e^{c_0t}\hat{f} + \mathcal{L}_ne^{c_0t}\hat{f} = c_0 e^{c_0t} \hat{f} ,\\
    &   \hat{f}(0,x,v) = \hat{f}_0(x,v), \\
    & e^{c_0t}\hat{f}|_{\gamma_-} = P_\gamma e^{c_0t}\hat{f}.
\end{cases}
\end{align*}

The real part of the $L^2_{x_3,v}$ energy estimate with Lemma \ref{lemma:Ln} leads to
\begin{align*}
    & e^{2c_0t} \Vert \hat{f}(t)\Vert_{L^2_{x_3,v}}^2 + \int_0^t | e^{c_0s} (I-P_\gamma)\hat{f}(s)|^2_{L^2_{\gamma_+}} \dd s + \int_0^t \Vert e^{c_0s}(\mathbf{I}-\mathbf{P})P_n \hat{f}(s) \Vert_{L^2_{x_3,\nu}}^2 \dd s  \\
    &+  \int_0^t \Vert e^{c_0s}[ \hat{f} - P_n \hat{f}](s)\Vert_{L^2_{x_3,\nu}}^2 \dd s \leq c_0 \int_0^t \Vert e^{c_0s}\hat{f}(s)\Vert_{L^2_{x_3,v}}^2 \dd s + \Vert \hat{f}_0\Vert_{L^2_{x_3,v}}^2.
\end{align*}

Since $|k|>\kappa$, the following macroscopic estimate holds
\begin{align*}
    &    \int_0^t \Vert \kappa e^{c_0s}\mathbf{P}\hat{f}\Vert_{L^2_{x_3,v}}^2 \dd s \lesssim \int_0^t \Vert e^{c_0s}(\mathbf{I}-\mathbf{P})P_n \hat{f} \Vert^2_{L^2_{x_3,\nu}} \dd s  + \int_0^t \Vert e^{c_0s}[ \hat{f} - P_n \hat{f} ]\Vert^2_{L^2_{x_3,\nu}} \dd s \\
    & + \int_0^t |e^{c_0s} (I-P_\gamma)\hat{f}|^2_{L^2_{\gamma_+}} \dd s+ \Vert e^{c_0t} \hat{f}(t)\Vert_{L^2_{x_3,v}}^2 + \Vert \hat{f}_0\Vert_{L^2_{x_3,v}}^2 + c_0 \int_0^t \Vert e^{c_0s}\hat{f}\Vert_{L^2_{x_3,v}}^2 \dd s.
\end{align*}
We refer to Lemma 5 in \cite{chen2025global} for details of the proof. We note that one only needs an additional time derivative estimate $\p_t (e^{c_0t}\hat{f})$ in the weak formulation presented in Proposition \ref{prop:macroscopic}. 

Combining both estimates, with some $c_0< \kappa$, we obtain that
\begin{align*}
    &    \Vert e^{c_0t}\hat{f}(t)\Vert_{L^2_{x_3,v}}^2 + \int_0^t |e^{c_0s}(I-P_\gamma)\hat{f}|^2_{L^2_{\gamma_+}} \dd s + \int_0^t \Vert \frac{\kappa}{2} e^{c_0s} \mathbf{P}\hat{f}\Vert_{L^2_{x_3,v}}^2 \dd s \lesssim \Vert \hat{f}_0\Vert_{L^2_{x_3,v}}^2.
\end{align*}
This concludes the lemma.
\end{proof}

\subsection{Asymptotic behavior of $e^{Bt}$}\label{sec:linear_asymptotic}

Collecting Proposition \ref{prop:eigenvalue} and Lemma \ref{lemma:k_lower_bdd}, we derive the following crucial representation of the semi-group $e^{\hat{B}_nt}$.

\begin{proposition}\label{prop:inverse_laplace}
It holds that
\begin{align*}
    &     e^{\hat{B}_n(k)t} = e^{\lambda_n(k) t}(\mathbf{P}_0 + |k|\mathbf{P}^n_1(k) + |k|^2\mathbf{P}^n_2(k)) \mathbf{1}_{|k|<\kappa} + \Phi(k).
\end{align*}
Here
\begin{align*}
    &  \Vert \Phi(k)\Vert_{L^2_{x_3,v}} \lesssim e^{-c_0 t}.
\end{align*}

\end{proposition}
\begin{remark}
Recall that $c_0\ll \kappa$ in Lemma \ref{lemma:k_lower_bdd}, and $\kappa=\kappa(\sigma_0),\sigma_0 = \sigma_0(\tau_0)$ as specified in Proposition \ref{prop:eigenvalue}. The constant $\tau_0$ defined in Lemma \ref{lemma:large_tau_inv} is non-constructive. Therefore, the constant $c_0$ in the exponential decay rate is non-constructive.

\end{remark}

\begin{proof}[Proof of Proposition \ref{prop:inverse_laplace}]
We first decompose
\begin{align*}
    & e^{\hat{B}_n(k)t} = e^{\hat{B}_n(k)t}\mathbf{1}_{|k|<\kappa} + e^{\hat{B}_n(k)t} \mathbf{1}_{|k|> \kappa}.
\end{align*}
The second term is controlled by Lemma \ref{lemma:k_lower_bdd}. Then we focus on the first term, i.e, $|k|<\kappa$.

Denote
\begin{align}
    &Z(\lambda): = (\lambda - \hat{A}(k))^{-1} (I-G(\lambda))^{-1}G(\lambda), \ G(\lambda):= K_n(\lambda-\hat{A}(k))^{-1}. \label{Z}
\end{align}
For $|k|<\kappa$, we use the inverse Laplacian transform with $\sigma_0>0$ to have
\begin{align*}
    &   e^{\hat{B}_n(k)t} = e^{\hat{A}(k)t} + \lim_{a\to \infty}\frac{1}{2i\pi}U_{\sigma_0,a}, \\
    & U_{\sigma_0,a} = \int_{-a}^a e^{(\sigma_0 + i \tau)t}Z(\sigma_0+i\tau) \dd \tau.
\end{align*}

By Proposition \ref{prop:eigenvalue} and Lemma \ref{lemma:kneq0_continuous}, we apply Residue Theorem 
to have
\begin{align*}
    &   U_{\sigma_0,a} = 2\pi i e^{\lambda_n(k) t}\tilde{\mathbf{P}}^n(k) + H_a + U_{-\sigma_0,a}.
\end{align*}
Here
\begin{align*}
    & H_a = \Big(\int_{-\sigma_0+ia}^{\sigma_0+ia} - \int_{-\sigma_0-ia}^{\sigma_0-ia} \Big) e^{\lambda t} Z(\lambda) \dd \lambda. 
\end{align*}
By Lemma \ref{lemma:large_tau_inv}, we have 
\begin{align*}
    &   \lim_{a\to \infty}\Vert H_a\Vert_{L^2_{x_3,v}} = 0. 
\end{align*}

For $U_{-\sigma_0,a}$, we compute that
\begin{align}
 \lim_{a\to \infty}\Vert U_{-\sigma_0,a}\Vert_{L^2_{x_3,v}}   &  \leq \Big\Vert \int_{-\tau_0}^{\tau_0} e^{(-\sigma_0 + i\tau) t} [e^{\hat{B}_n(k)t} - e^{\hat{A}(k)t}] \dd \tau  \Big\Vert_{L^2_{x_3,v}} \label{U_minus_1} \\
    &\quad  +  \Big(\int_{-\infty}^{-\tau_0} + \int_{\tau_0}^\infty \Big) \Big\Vert e^{(-\sigma_0 + i\tau)t}  Z(-\sigma_0 + i\tau)\Big\Vert_{L^2_{x_3,v}}\dd \tau . \label{U_minus_2}
\end{align}
Here we have applied Minkowski inequality in \eqref{U_minus_2}.

By \eqref{inverse_norm}, we derive that
\begin{align*}
    &    \eqref{U_minus_1} \lesssim 2\tau_0 e^{-\sigma_0 t}[\sup_{|\tau|<\tau_0} \Vert (-\sigma_0+i\tau-\hat{B}_n(k))^{-1}\Vert_{L^2_{x_3,v}} + \sup_{|\tau|<\tau_0} \Vert (-\sigma_0 + i\tau - \hat{A}(k))^{-1}\Vert_{L^2_{x_3,v}} ]\lesssim   \frac{\tau_0}{\sigma_0^{3/2}} e^{-\sigma_0 t} .
\end{align*}

For \eqref{U_minus_2}, given $|\tau|>\tau_0$, by Lemma \ref{lemma:large_tau_inv}, we have
\begin{align*}
    &  \Vert G(\lambda)\Vert_{L^2_{x_3,v}} \leq \Vert P_n\Vert_{L^2_{x_3,v}}\Vert (\lambda-\hat{A}(k))^{-1}\Vert_{L^2_{x_3,v}} < \frac{1}{2}, \ \Vert (I-G(\lambda))^{-1}\Vert_{L^2_{x_3,v}} < 2.
\end{align*}
This leads to
\begin{align}
    & \eqref{U_minus_2} \lesssim \Big( \int_{-\infty}^{-\tau_0} + \int_{\tau_0}^\infty \Big) e^{-\sigma_0 t} \Vert Z(-\sigma_0+i\tau) \Vert_{L^2_{x_3,v}} \dd \tau \notag\\
    & \lesssim e^{-\sigma_0 t}\Vert K_n\Vert_{L^2_{x_3,v}} \sup_{|\tau|>\tau_0}\Vert (I-G(-\sigma_0+i\tau))^{-1}\Vert_{L^2_{x_3,v}} \int_{-\infty}^\infty \Vert (-\sigma_0+i\tau- \hat{A}(k))^{-1}\Vert_{L^2_{x_3,v}}^2  \dd \tau \lesssim e^{-\sigma_0 t}. \label{tau_0_integrable}
\end{align}
Here we have applied the definition of $Z$ in \eqref{Z} and the computation in \eqref{integrable}.

We conclude that
\begin{align*}
    &  \lim_{a\to \infty} \Vert U_{-\sigma_0,a}\Vert_{L^2_{x_3,v}} \lesssim e^{-\sigma_0 t}.
\end{align*}

We define $U_{-\sigma_0,\infty} : = \int_{-\infty}^\infty  e^{(-\sigma+i\tau)t}Z(-\sigma_0+i\tau) \dd \tau$. Note that $U_{-\sigma_0,a}$ is a well defined operator for $a<\infty$. We estimate the difference of $U_{-\sigma_0,\infty}$ and $U_{-\sigma_0,a}$ as
\begin{align*}
    & \lim_{a\to \infty}\Vert   U_{-\sigma_0,\infty}-U_{-\sigma_0,a}   \Vert_{L^2_{x_3,v}} = \lim_{a\to \infty}\Big\Vert   \Big(\int_{-\infty}^{-a} + \int_{a}^\infty \Big)e^{(-\sigma_0+i\tau)t}Z(-\sigma_0+i\tau)\dd \tau \Big\Vert_{L^2_{x_3,v}} \\
    & \leq \lim_{a\to \infty} \Big(\int_{-\infty}^{-a} + \int_a^{\infty}\Big) e^{-\sigma_0 t}\Vert Z(-\sigma_0+i\tau)\Vert_{L^2_{x_3,v}} \dd \tau = 0. 
\end{align*}
The limit becomes $0$ due to the integrability in \eqref{tau_0_integrable}. Hence, $U_{-\sigma_0,\infty}$ is a well-defined operator and it satisfies
\begin{align*}
    & \Vert U_{-\sigma_0,\infty}\Vert_{L^2_{x_3,v}} \lesssim e^{-\sigma_0 t}.
\end{align*}

Finally we define 
\begin{align*}
    &\Phi(k) = e^{\hat{A}(k)t}\mathbf{1}_{|k|<\kappa} + U_{-\sigma_0,\infty}  + e^{\hat{B}_n(k)t}\mathbf{1}_{|k|>\kappa}.
\end{align*}
Applying Lemma \ref{lemma:k_lower_bdd} to the last term, we complete the proof from the fact that $c_0\ll \kappa \ll \sigma_0$.   
\end{proof}

\hide

We consider the solution of $e^{\hat{B}t}$, we have the following conclusion to
\begin{align*}
    \p_t \hat{f} + i\bar{v}\cdot k \hat{f} + v_3 \p_{x_3} \hat{f} + \mathcal{L}\hat{f} = 0.
\end{align*}
\begin{lemma}\label{lemma:B_semigroup}
\begin{align*}
    &  \Vert \hat{f}(t)\Vert_{L^2_{x_3,v}} \lesssim e^{-\sigma_0 t}\Vert \hat{f}_0\Vert_{L^2_{x_3,v}} + \mathbf{1}_{|k|<\kappa} e^{\lambda_n(k) t}\Vert \hat{f}_0\Vert_{L^2_{x_3,v}}.
\end{align*}

When $\mathbf{P}_0 \hat{f}_0 = 0$, we have
\begin{align*}
    &     \Vert \hat{f}(t)\Vert_{L^2_{x_3,v}} \lesssim e^{-\sigma_0 t}\Vert \hat{f}_0\Vert_{L^2_{x_3,v}} + \mathbf{1}_{|k|<\kappa} |k|e^{\lambda_n(k) t}\Vert \hat{f}_0\Vert_{L^2_{x_3,v}}.
\end{align*}

\end{lemma}

\begin{proof}
We rewrite the equation as
\begin{align*}
    &   \p_t \hat{f} + i\bar{v}\cdot k \hat{f} + v_3 \p_{x_3}\hat{f} + \mathcal{L}_n \hat{f} = \mathcal{L}_n \hat{f} - \mathcal{L}\hat{f}.
\end{align*}
By the Duhammel principle, we further have
\begin{align*}
    &  \hat{f}(t,x_3,v) = e^{\hat{B}_n(k)t}\hat{f}_0  +  \int^t_0 e^{\hat{B}_n(t-s)} (\mathcal{L}_n - \mathcal{L})f(s) \dd s.
\end{align*}

Apply Proposition \ref{prop:inverse_laplace}, we have
\begin{align*}
    &     \Vert \hat{f}(t)\Vert_{L^2_{x_3,v}} \lesssim e^{-\sigma_0 t} \Vert \hat{f}_0\Vert_{L^2_{x_3,v}} + \mathbf{1}_{|k|<\kappa} e^{\lambda_n(k) t}\Vert \hat{f}_0\Vert_{L^2_{x_3,v}} + \int_0^t \Vert K_n \hat{f}(s)- K \hat{f}(s)\Vert_{L^2_{x_3,v}} \dd s .
\end{align*}
Since this holds for any $n$, let $n\to \infty$, by dominating convergence theorem, we conclude that
\begin{align*}
    & \Vert \hat{f}(t)\Vert_{L^2_{x_3,v}} \lesssim e^{-\sigma_0 t}\Vert \hat{f}_0\Vert_{L^2_{x_3,v}} + \mathbf{1}_{|k|<\kappa} e^{\lambda_n(k) t} \Vert \hat{f}_0\Vert_{L^2_{x_3,v}}.
\end{align*}

When $\mathbf{P}_0 \hat{f}_0 = 0$, the only difference lies in $e^{\hat{B}_n(k)t}\hat{f}_0$:
\begin{align*}
    &     \Vert \hat{f}(t)\Vert_{L^2_{x_3,v}} \lesssim e^{-\sigma_0 t} \Vert \hat{f}_0\Vert_{L^2_{x_3,v}} + \mathbf{1}_{|k|<\kappa} e^{\lambda_n(k) t} |k|\Vert \hat{f}_0\Vert_{L^2_{x_3,v}} + \int_0^t \Vert K_n \hat{f}(s)- K \hat{f}(s)\Vert_{L^2_{x_3,v}} \dd s .
\end{align*}
Again, this holds for any $n$, we send $n\to \infty$ and conclude the lemma.

\end{proof}

\unhide

From Proposition \ref{prop:inverse_laplace}, we express $e^{B_n t}$ as 
\begin{align*}
    &  e^{B_nt} = E_1(t) + E_2(t) , \\
    & E_1(t) = \mathcal{F}^{-1} \Big\{ e^{\lambda_n(k) t} (\mathbf{P}_0 + |k|\mathbf{P}^n_1(k) + |k|^2 \mathbf{P}^n_2(k))\mathbf{1}_{|k|<\kappa} \Big\} \mathcal{F}, \\
    & E_2(t) = \mathcal{F}^{-1}\Big\{ \Phi(k) \Big\} \mathcal{F} .
\end{align*}

Recall the norm $\Vert \cdot \Vert_{Z_q}$ and the quantity $\sigma_{q,m}$ defined in \eqref{si-qm} in Section \ref{sec:result}.

\begin{lemma}\label{lemma:linear_decay}
We have the following algebraic decay for $E_1(t)$:
\begin{align*}
    &   \Vert \p_{x_\parallel}^\alpha E_1(t)u\Vert_{L^2_{x,v}}  \lesssim (1+t)^{-\sigma_{q,m}} \Vert \p_{x_\parallel}^{\alpha'} u\Vert_{Z_q}, \\
    & \Vert \p_{x_\parallel}^\alpha E_1(t)(\mathbf{I}-\mathbf{P})u\Vert_{L^2_{x,v}}  \lesssim (1+t)^{-\sigma_{q,m+1}} \Vert \p_{x_\parallel}^{\alpha'} u\Vert_{Z_q}.
\end{align*}
Here $m = |\alpha| - |\alpha'|\geq 0$.

In particular, when $\mathbf{P}_0 u = 0$, we have
\begin{align*}
    & \Vert \p_{x_\parallel}^\alpha E_1(t)u\Vert_{L^2_{x,v}} \lesssim (1+t)^{-\sigma_{q,m+1}}\Vert \p_{x_\parallel}^{\alpha'} u\Vert_{Z_q}.
\end{align*}

And we have the following exponential decay for $E_2(t)$:
\begin{align*}
    &   \Vert E_2(t)u\Vert_{L^2_{x,v}} \lesssim e^{-c_0 t} \Vert u\Vert_{L^2_{x,v}}.
\end{align*}  
\end{lemma}

\begin{proof}
The proof is standard by integrating over the frequency variable $k$. The extra decay to $(\mathbf{I}-\mathbf{P})u
$ comes from $\mathbf{P}_0[(\mathbf{I}-\mathbf{P})u] = 0$. We refer to the detailed proof in Theorem 2.2.15 of \cite{ukai2006mathematical}.
\end{proof}

Now we analyze the asymptotic behavior of the semi-group $e^{Bt}$, which corresponds to the
solution of the linearized Boltzmann equation.
\begin{align}\label{linear_f}
\begin{cases}
     &\dis \p_t f + v\cdot \nabla_x f + \mathcal{L}f = 0, \\
    &\dis  f(t,x,v)|_{\gamma_-} = P_\gamma f, \\   
    &\dis  f(0,x,v) = f_0(x,v) := (F(0,x,v)-\mu)/\sqrt{\mu}. 
\end{cases}
\end{align}

\begin{proposition}\label{prop.Bdecay}
Let $m=|\alpha| - |\alpha'|\geq 0$. It holds that
\begin{align*}
    & \Vert  \p_{x_\parallel}^\alpha e^{Bt} u\Vert_{L^2_{x,v}} \lesssim (1+t)^{-\sigma_{q,m}}\Vert \p_{x_\parallel}^{\alpha'} u\Vert_{Z_q} + e^{-c_0 t} \Vert  \p_{x_\parallel}^\alpha u\Vert_{L^2_{x,v}} , \\
    & \Vert \p_{x_\parallel}^\alpha e^{Bt}(\mathbf{I}-\mathbf{P})u\Vert_{L^2_{x,v}}  \lesssim (1+t)^{-\sigma_{q,m+1}} \Vert \p_{x_\parallel}^{\alpha'} u\Vert_{Z_q} + e^{-c_0 t} \Vert  \p_{x_\parallel}^\alpha u\Vert_{L^2_{x,v}} .
\end{align*}
In particular, when $\mathbf{P}_0 u = 0$, we have
\begin{align*}
    & \Vert \p_{x_\parallel}^\alpha e^{Bt}u\Vert_{L^2_{x,v}} \lesssim (1+t)^{-\sigma_{q,m+1}}\Vert \p_{x_\parallel}^{\alpha'} u\Vert_{Z_q} + e^{-c_0 t} \Vert  \p_{x_\parallel}^\alpha u\Vert_{L^2_{x,v}} .
\end{align*}
\end{proposition}

\begin{proof}
Since $B=A-K$, where $K$ is bounded operator on $L^2 (\Omega\times \R^3_v)$, we have, for any $t<\infty$, 
\begin{align*}
    &   \Vert e^{Bt}f_0\Vert_{L^2_{x,v}} \lesssim e^{(-\nu_0 + \Vert K\Vert_{L^2_{x,v}})t} \Vert f_0\Vert_{L^2_{x,v}}<\infty.
\end{align*}

Then for $e^{Bt}f_0$, we rewrite the corresponding equation in \eqref{linear_f} as 
\begin{align*}
    &   \p_t f + v\cdot \nabla_x f + \mathcal{L}_n f = \mathcal{L}_n f- \mathcal{L}f, \ f|_{\gamma_-} = P_\gamma f, \  \ f(0,x,v) = f_0(x,v).
\end{align*}
By the Duhammel principle, we further have
\begin{align*}
    & f(t,x,v) = e^{B_n t} f_0 + \int^t_0 e^{B_n(t-s)}(\mathcal{L}_n - \mathcal{L})f(s)\dd s.
\end{align*}
Then applying Lemma \ref{lemma:linear_decay} we have
\begin{align*}
    & \Vert f(t)\Vert_{L^2_{x,v}} \lesssim (1+t)^{-\sigma_{q,0}} \Vert f_0\Vert_{Z_q} + e^{-c_0 t} \Vert f_0\Vert_{L^2_{x,v}} + \int_0^t e^{(-\nu_0 + \Vert K_n\Vert_{L^2_{x,v}})(t-s)}\Vert (\mathcal{L}_n - \mathcal{L})f(s)\Vert_{L^2_{x,v}} \dd s \\
    & \lesssim  (1+t)^{-\sigma_{q,0}} \Vert f_0\Vert_{Z_q} + e^{-c_0 t} \Vert f_0\Vert_{L^2_{x,v}} + \int_0^t e^{(-\nu_0 + \Vert K\Vert_{L^2_{x,v}})(t-s)}\Vert (\mathcal{L}_n - \mathcal{L})f(s)\Vert_{L^2_{x,v}} \dd s. 
\end{align*}
Since this holds for any $n$, letting $n\to \infty$, by the dominated convergence theorem, we have
\begin{align*}
    & \lim_{n\to \infty} \int_0^t e^{(-\nu_0 + \Vert K\Vert_{L^2_{x,v}})(t-s)}\Vert (\mathcal{L}_n - \mathcal{L})f(s)\Vert_{L^2_{x,v}} \dd s = \lim_{n\to \infty}\int_0^t e^{(-\nu_0 + \Vert K\Vert_{L^2_{x,v}})(t-s)} \Vert (K_n-K)f(s) \Vert_{L^2_{x,v}} \dd s  \\
    & =\int_0^t e^{(-\nu_0 + \Vert K\Vert_{L^2_{x,v}})(t-s)}\lim_{n\to \infty} \Vert (K_n-K)f(s)\Vert_{L^2_{x,v}} \dd s= 0. 
\end{align*}
We conclude the lemma for $\alpha = \alpha' = 0$. The proof for $\alpha-\alpha'\neq 0$ is similar.
\end{proof}

\section{Nonlinear estimate}\label{sec:nonlinear}

In this section, we combine the linear spectral analysis from Section \ref{sec:linear} with an $L^2–L^\infty$ bootstrap argument to analyze the nonlinear problem \eqref{nonlinear_f}, thereby proving both Theorem \ref{thm:asymptotic_stability} and Theorem \ref{thm:leading_behavior}. In Section \ref{sec:asymptotic_infty}, we establish the linear decay rate in the norm $\Vert \cdot \Vert_{H^{\theta,\ell}}$(see Proposition \ref{prop:linfty_decay_rate}) introduced in Section \ref{sec:result}. This linear decay rate is then applied in Section \ref{sec:nonliear_asym} to study the nonlinear problem \eqref{nonlinear_f}, leading to the proof of Theorem \ref{thm:asymptotic_stability}. Finally, in Section \ref{sec:leading}, we derive the leading asymptotic profile of the solution to the nonlinear problem and complete the proof of Theorem \ref{thm:leading_behavior}.

\subsection{Asymptotic behavior under the norm  $\Vert \cdot \Vert_{H^{\theta,\ell}}$}\label{sec:asymptotic_infty}
We recall the norms $\Vert \cdot \Vert_{H^{\theta,\ell}}$, $\Vert \cdot \Vert_{X^{\theta,q}}$,  and $\Vert \cdot \Vert_{\dot{H}^{\theta,\ell}}$ as well as the polynomial rate index $\sigma_{q,\ell}$ in Section \ref{sec:result}.

\begin{proposition}\label{prop:linfty_decay_rate}
Let the initial condition satisfy
\begin{align*}
    & \Vert f_0\Vert_{H^{\theta,\ell}} + \Vert f_0\Vert_{X^{\theta,q}} < \infty.
\end{align*}
Then there exists a unique solution to the linear problem \eqref{linear_f} such that
\begin{align*}
    &   \Vert f(t)\Vert_{\dot{H}^{\theta,\ell}} \lesssim e^{-c_0 t} \Vert f_0\Vert_{\dot{H}^{\theta,\ell}} + (1+t)^{-\sigma_{q,\ell}} \Vert f_0\Vert_{X^{\theta,q} } ,
\end{align*}
and thus
\begin{align*}
    & \Vert f(t)\Vert_{H^{\theta,\ell}} \lesssim e^{-c_0 t} \Vert f_0\Vert_{H^{\theta,\ell}} + (1+t)^{-\sigma_{q,0}}  \Vert f_0\Vert_{X^{\theta,q} }.
\end{align*}

If we further assume $\mathbf{P}_0f_0 = 0$, then the unique solution further satisfies
\begin{align*}
    &   \Vert f(t)\Vert_{\dot{H}^{\theta,\ell}} \lesssim e^{-c_0 t}\Vert f_0\Vert_{\dot{H}^{\theta,\ell}} + (1+t)^{-\sigma_{q,\ell+1}} \Vert f_0\Vert_{X^{\theta,q} },
\end{align*}
and thus
\begin{align*}
    &   \Vert f(t)\Vert_{H^{\theta,\ell}} \lesssim e^{-c_0 t}\Vert f_0\Vert_{H^{\theta,\ell}} + (1+t)^{-\sigma_{q,1}} \Vert f_0\Vert_{X^{\theta,q} }.
\end{align*}
\end{proposition}

To prove the proposition above, we will apply the $L^2_{x_3,v}-L^\infty_{x_3,v}$ bootstrap argument. We use standard notations for the backward exit time and backward exit position:
\begin{equation*}
\begin{split}
\tb(x,v) :   &  = \sup\{s\geq 0, x-sv \in \O\} , \\
  \xb(x,v)  : & = x - \tb(x,v)v.
\end{split}
\end{equation*}

We denote $t_0=t$ to be a fixed starting time. First, we define the stochastic cycle as follows.

\begin{definition}\label{def:sto_cycle}
We define the stochastic cycles as $(x^0,v^0)= (x,v) \in \bar{\O} \times \R^3$ and inductively
\begin{align}
&x^1:= \xb(x,v), \   v^1 \in \mathcal{V}_1:=\{v^1\in \mathbb{R}^3: v^1_3\times \text{sign}(x^1_3)>0\} , \notag\\
&  v^{k}\in \mathcal{V}_k:= \{v^{k}\in \mathbb{R}^3: v^k_3 \times \text{sign}(x^k_3)>0\}, \ \ \text{for} \  k \geq 1,
\notag
\\
 &x^{k+1} := \xb(x^k, v^k) , \ \tb^{k}:= \tb(x^k,v^k) \ \ \text{for} \  v^k \in \mathcal{V}_k, \notag
\end{align}
\begin{equation*}
t^k =  t_0-  \{ \tb + \tb^1 + \cdots + \tb^{k-1}\},  \ \ \text{for} \  k \geq 1.    
\end{equation*}

\end{definition}

With the stochastic cycles defined above, we apply the method of characteristics to have
\begin{align}
 f(t,x,v) &   =  \mathbf{1}_{t^1\leq 0}  e^{- \nu(v) t} f_0(x-t v, v) \label{initial}\\
  & + \mathbf{1}_{t^1 \leq 0} \int_0^{t} e^{-\nu(v) (t-s)}   \int_{\mathbb{R}^3}   f(s,x-(t-s)v,u) \mathbf{k}(v,u) \dd u \dd s \label{K_0}\\
  & +\mathbf{1}_{t^1>0} \int^{t}_{t^1} e^{-\nu(v) (t - s)}    \int_{\mathbb{R}^3}  f(s,x-(t-s)v,u) \mathbf{k}(v,u) \dd u \dd s  \label{K_1} \\
  & + \mathbf{1}_{t^1>0} e^{-\nu(v) (t - t^1)} f(t^1,x^1,v), \label{f_bdr}
\end{align}
where the contribution of the boundary is bounded as
\begin{align}
   & |\eqref{f_bdr}|\leq   e^{-\nu(v) (t-t^1)}  \sqrt{\mu(v)}  \notag\\
   & \times \int_{\prod_{j=1}^{k-1}\mathcal{V}_j} \bigg\{ \sum_{i=1}^{k-1}\mathbf{1}_{t^{i+1}\leq 0 < t^i} e^{-\nu(v^i) t^i} w(v^i)|f(0, x^i - t^i v^i, v^i)| \dd \Sigma_i     \label{bdr_initial_2d}\\
  & + \mathbf{1}_{t^k>0}  w(v^{k-1})|f(t^k,x^k,v^{k-1})| \dd \Sigma_{k-1} \label{bdr_tk_2d}\\
  & + \sum_{i=1}^{k-1} \mathbf{1}_{t^{i+1}\leq 0 < t^i}  \int^{t^i}_0 e^{-\nu(v^i) (t^i-s)}   w(v^i)  \int_{\mathbb{R}^3}   |f(s,x^i-(t^i-s)v^i, u)| \mathbf{k}(v^i,u) \dd u \dd s \dd \Sigma_i   \label{bdr_K_0_2d} \\
  & + \sum_{i=1}^{k-1} \mathbf{1}_{t^{i+1}>0} \int_{t^{i+1}}^{t^i} e^{-\nu(v^i) (t^i-s)}  w(v^i)
  \int_{\mathbb{R}^3} \mathbf{k}(v^i,u) |f(s,x^i-(t^i-s)v^i, u)| \dd u \dd s \dd \Sigma_i   \label{bdr_K_i_2d} .
\end{align}
Here $\dd \Sigma_i$ is defined as
\begin{equation}
\dd \Sigma_i = \Big\{\prod_{j=i+1}^{k-1} \dd \sigma_j \Big\}  \times  \Big\{  \frac{1}{w(v^i)\sqrt{\mu(v^i)}}   \dd \sigma_i \Big\} \times \Big\{\prod_{j=1}^{i-1}  e^{-\nu(v^j)(t^j-t^{j+1})}   \dd \sigma_j \Big\} , \label{Sigma_2d}
\end{equation}
where $\dd \sigma_i$ is a probability measure in $\mathcal{V}_i$ given by
\begin{equation}
\dd \sigma_i =  \sqrt{2\pi}\mu(v^i) |v^i_3|\dd v^i. \label{d_sigma_i_2d}
\end{equation}

Note that \eqref{bdr_tk_2d} corresponds to the scenario that the backward trajectory interacts with the diffuse boundary portion a large number of times. This term is controlled by the following lemma.

\begin{lemma}[\cite{G}]\label{lemma:tk_2d}
For $T_0>0$ sufficiently large, there exist constants $C_1,C_2>0$ independent of $T_0$ such that for $k = C_1 T_0^{5/4}$ and $(t^0,x^0,v^0) = (t,x,v)\in [0,T_0]\times \bar{\O}\times \mathbb{R}^3$,
\begin{equation*}
\int_{\prod_{j=1}^{k-1} \mathcal{V}_j}\mathbf{1}_{t^k>0}  \prod_{j=1}^{k-1} \dd \sigma_j \leq \Big( \frac{1}{2}\Big)^{C_2 T_0^{5/4}}.
\end{equation*}

\end{lemma}

First, we give an estimate to the boundary term \eqref{f_bdr}.

\begin{lemma}\label{lemma:bdr}
Let $t\leq T_0$ with $T_0$ introduced in Lemma \ref{lemma:tk_2d}, then for the boundary term~\eqref{f_bdr}, it holds that
\begin{equation*}
\begin{split}
& \Vert \mathbf{1}_{t^1>0}e^{-\nu(v)(t-t^1)}  f(t^1,x^1,v)\Vert_{\dot{H}^{\theta,\ell}} \\
   & \leq 4 e^{-\frac{\nu_0}{2} t}\Vert f_0\Vert_{\dot{H}^{\theta,\ell}}+ o(1) (1+t)^{-\sigma_{q,\ell}}\sup_{s\leq t}(1+s)^{\sigma_{q,\ell}}\Vert f(s)\Vert_{\dot{H}^{\theta,\ell}}  \\
   &+ C(T_0)[e^{-c_0 t}\Vert f_0\Vert_{L^2_{x_3,v}\dot{H}^\ell_{x_\parallel}}+(1+t)^{-\sigma_{q,\ell}} \Vert f_0\Vert_{X^{\theta,q}}].
\end{split}
\end{equation*}

When $\mathbf{P}_0f_0 = 0,$ it holds that
\begin{equation*}
\begin{split}
& \Vert \mathbf{1}_{t^1>0}e^{-\nu(v)(t-t^1)} f(t^1,x^1,v)\Vert_{\dot{H}^{\theta,\ell}} \\
   & \leq 4 e^{-\frac{\nu_0}{2} t}\Vert f_0\Vert_{\dot{H}^{\theta,\ell}}+ o(1) (1+t)^{-\sigma_{q,\ell+1}}\sup_{s\leq t}(1+s)^{\sigma_{q,\ell+1}}\Vert f(s)\Vert_{\dot{H}^{\theta,\ell}} \\
   &+ C(T_0)[e^{-c_0 t}\Vert f_0\Vert_{L^2_{x_3,v}\dot{H}^\ell_{x_\parallel}} + (1+t)^{-\sigma_{q,\ell+1}} \Vert f_0\Vert_{X^{\theta,q}}].
\end{split}
\end{equation*}

\end{lemma}

\begin{proof}
We begin with the estimate of \eqref{bdr_initial_2d}. We observe that $\text{sign}(x^i_3),t^i,v^i$ are independent of $x_\parallel$, by Minkowski inequality, for each $1\leq i\leq k-1$, we have
\begin{align*}
    &    \Big\Vert w(v)\eqref{bdr_initial_2d}(t,\cdot,x_3,v) \Big\Vert_{\dot{H}^\ell_{x_\parallel}} \lesssim  \mathbf{1}_{t^1>0}e^{-\nu(v)(t-t^1)} \int_{\prod_{j=1}^{k-1}\mathcal{V}_j} \mathbf{1}_{t^{i+1\leq 0\leq t^i}} e^{-\nu(v^i)t^i} w(v^i)\Vert f_0(\cdot, x^i_3-t^iv^i_3,v^i)\Vert_{\dot{H}^\ell_{x_\parallel}} \dd \Sigma_i  \\
    & \lesssim \Vert f_0\Vert_{\dot{H}^{\theta,\ell}} \mathbf{1}_{t^1>0}e^{-\nu(v)(t-t^1)} \int_{\prod_{j=1}^{k-1}\mathcal{V}_j} \mathbf{1}_{t^{i+1\leq 0\leq t^i}} e^{-\nu(v^i)t^i} \dd \Sigma_i  .
\end{align*}
Further taking the norm of $L^\infty_{x_3,v}$, we have
\begin{align*}
    & \Vert \eqref{bdr_initial_2d} \Vert_{\dot{H}^{\theta,\ell}} \leq 4 e^{-\nu_0 t} \Vert f_0\Vert_{\dot{H}^{\theta,\ell}}.
\end{align*}
Here we used that $\dd \sigma_i$ in~\eqref{d_sigma_i_2d} is a probability measure. The constant $4$ comes from
\[\int_{\mathcal{V}_i}   |v^i_3| \sqrt{\mu(v^i)} w^{-1}(v^i) \dd v^i\leq 4.\]
The exponential decay factor $e^{-\nu_0 t^1}$ comes from the decay factor from \eqref{Sigma_2d}, and the computation
\begin{align*}
    & e^{-\nu_0 t^i} e^{-\nu_0 (t^{i-1}-t^i)} \leq e^{-\nu_0 t^{i-1}}, \ e^{-\nu_0 t^{i-1}} e^{-\nu_0 (t^{i-2}-t^{i-1})} \leq e^{-\nu_0 t^{i-2}} \cdots.
\end{align*}

For~\eqref{bdr_tk_2d}, with $k=C_1 T_0^{5/4}$, we apply Minkowski inequality to have
\begin{align}
&  \Vert w(v)\eqref{bdr_tk_2d}(t,\cdot,x_3,v) \Vert_{\dot{H}^\ell_{x_\parallel}}    \leq  \mathbf{1}_{t^1>0} e^{-\nu(v)(t-t^1)} \int_{\prod_{j=1}^{k-1} \mathcal{V}_j} \mathbf{1}_{t^k>0}  w(v^{k-1})\Vert f(t^k,\cdot,x^k_3,v^{k-1})\Vert_{\dot{H}^\ell_{x_\parallel}} \dd \Sigma_{k-1} \notag \\
 &\leq \sup_{s\leq t} (1+s)^{\sigma_{q,\ell}} \Vert f(s)\Vert_{\dot{H}^{\theta,\ell}} \mathbf{1}_{t^1>0} e^{-\nu(v)(t-t^1)} \int_{\prod_{j=1}^{k-1} \mathcal{V}_j} \mathbf{1}_{t^k>0}  (1+t^k)^{-\sigma_{q,\ell}} \dd \Sigma_{k-1}.\notag
\end{align}
Taking the norm of $L^\infty_{x_3,v}$, we apply Lemma \ref{lemma:tk_2d} to have
\begin{align*}
    &   \Vert \eqref{bdr_tk_2d} \Vert_{\dot{H}^{\theta,\ell}} \lesssim (1+t)^{-\sigma_{q,\ell}} \sup_{s\leq t} (1+s)^{\sigma_{q,\ell}} \Vert f(s)\Vert_{\dot{H}^{\theta,\ell}} .
\end{align*}
Here we have used
\begin{align*}
    & e^{-\frac{\nu_0}{2} (t-t^1)} e^{-\frac{\nu_0}{2}(t^1-t^2)}\times \cdots \times e^{-\frac{\nu_0}{2}(t^{k-1}-t^k)} (1+t^k)^{-\sigma_{q,\ell}} \leq e^{-\frac{\nu_0}{2}(t - t^k)} (1+t^k)^{-\sigma_{q,\ell}} \lesssim (1+t)^{-\sigma_{q,\ell}}.
\end{align*}

Then we estimate \eqref{bdr_K_0_2d} and \eqref{bdr_K_i_2d}. We only estimate \eqref{bdr_K_i_2d}; the estimate for the other term is the same. We bound it by
\begin{align}
&\Big\Vert \mathbf{1}_{t^1>0} w(v)\sqrt{\mu(v)} e^{-\nu_0(t-t^1)}\int_{\prod_{j=1}^i \mathcal{V}_j}\mathbf{1}_{t^{i+1}>0 }  \prod_{j=1}^i\dd \sigma_j \mu^{-1/2}(v^i)    \notag\\
    & \times     \int_{t^{i+1}}^{t^i} \dd s e^{-\nu_0(t^i -s)} \int_{\mathbb{R}^3} \dd u\mathbf{k}(v^1,u) |f(s,x^i-(t^i-s)v^i,u)|  \Big\Vert_{\dot{H}^\ell_{x_\parallel}} \notag\\
    & \lesssim    \mathbf{1}_{t^1>0}e^{-\nu_0(t-t^1)}\int_{\prod_{j=1}^i \mathcal{V}_j}\mathbf{1}_{t^{i+1}>0 }  \prod_{j=1}^i\dd \sigma_j \mu^{-1/2}(v^i) \notag\\
    & \times \int^{t^i}_{t^{i+1}} \dd s e^{-\nu_0(t^i-s)} 
  \int_{\mathbb{R}^3} \dd u \mathbf{k}(v^1,u) \Vert f(s,\cdot,x^i_3 - (t^i-s)v^i_3)\Vert_{\dot{H}^\ell_{x_\parallel}}  .\label{iteration_i_2d}
\end{align}

First we consider the small time integration $\max\{t^i-\delta,t^{i+1}\}  \leq s\leq t^i$, and we have
\begin{align*}
    & \Vert \eqref{iteration_i_2d}\mathbf{1}_{\max\{t^1-\delta,t^2\}  \leq s\leq t^1} \Vert_{L^\infty_{x_3,v}}\\
    &\lesssim \sup_{x_3,v} e^{-\nu_0 (t-t^1)}\int_{\prod_{j=1}^{i-1}\mathcal{V}_j} \prod_{j=1}^{i-1}e^{-\nu_0(t^j-t^{j+1})}\dd \sigma_j \int_{\text{sign}(x_3^i)\times v^i_3 >0} \dd v^i \mu^{1/4}(v^i) \int^{t^i}_{t^i-\delta} \dd s e^{-\nu(v^i)(t^1-s)}  \\
    &\times \int_{\mathbb{R}^3} \dd u \frac{\mathbf{k}(v^1,u)}{w(u)} (1+s)^{-\sigma_{q,\ell}} \times \sup_{s\leq t} (1+s)^{\sigma_{q,\ell}}\Vert f(s)\Vert_{\dot{H}^{\theta,\ell}} \\
    & \lesssim (1+t)^{-\sigma_{q,\ell}} \sup_{s\leq t} (1+s)^{\sigma_{q,\ell}}\Vert f(s)\Vert_{\dot{H}^{\theta,\ell}} \int_{\prod_{j=1}^{i-1}\mathcal{V}_j} \prod_{j=1}^{i-1}\dd \sigma_j \int_{\text{sign}(x_3^i)\times v^i_3 >0} \dd v^i \mu^{1/4}(v^i) \int^{t^i}_{t^i-\delta} \dd s  \\
    & \lesssim o(1)(1+t)^{-\sigma_{q,\ell}} \sup_{s\leq t} (1+s)^{\sigma_{q,\ell}}\Vert f(s)\Vert_{\dot{H}^{\theta,\ell}}.
\end{align*}

Next we consider the $u$ integration with $|u|>N$ or $|v^i-u|<\frac{1}{N}$. We apply Lemma \ref{lemma:k_theta} to have
\begin{align*}
    &\Vert \eqref{iteration_i_2d}\mathbf{1}_{|u|>N \text{ or }|v^1-u|<\frac{1}{N}} \Vert_{L^\infty_{x_3,v}}\\
    &\lesssim \sup_{x_3,v} e^{-\nu_0 (t-t^1)}  \int_{\prod_{j=1}^{i-1}\mathcal{V}_j} \prod_{j=1}^{i-1}e^{-\nu_0(t^j-t^{j+1})}\dd \sigma_j \int_{\text{sign}(x_3^i)\times v^i_3 >0} \dd v^i \mu^{1/4}(v^i) \int^{t^i}_{0} \dd s e^{-\nu(v^i)(t^1-s)} \\
    &\times  \int_{|u|>N \text{ or }|v^i-u|<\frac{1}{N}} \frac{\mathbf{k}(v^i,u)}{w(u)} \dd u \times  \sup_{s\leq t} (1+s)^{\sigma_{q,\ell}}\Vert f(s)\Vert_{\dot{H}^{\theta,\ell}} \\
    & \lesssim o(1)(1+t)^{-\sigma_{q,\ell}} \sup_{s\leq t} (1+s)^{\sigma_{q,\ell}}\Vert f(s)\Vert_{\dot{H}^{\theta,\ell}}.
\end{align*}

For the rest cases, we have $|v^i-u|>\frac{1}{N}$, and thus $\mathbf{k}(v^i,u)\lesssim 1$ from Lemma \ref{lemma:k_theta}. We apply the change of variable
\begin{align*}
    & v^i_3 \to y = x^i_3 - (t^i-s)v_3^i, \ \frac{\dd y}{\dd v_3^i} = -(t^i-s), \ |t^i-s|>\delta.
\end{align*}
We further compute that
\begin{align*}
    &\Vert \eqref{iteration_i_2d} \mathbf{1}_{|u|<N, \ |v^i-u|>\frac{1}{N}, \ s<t^i-\delta}\Vert_{L^\infty_{x_3,v}} \\
    &\lesssim \sup_{x_3,v} e^{-\nu_0 (t-t^1)}  \int_{\prod_{j=1}^{i-1}\mathcal{V}_j} \prod_{j=1}^{i-1}e^{-\nu_0(t^j-t^{j+1})}\dd \sigma_j \int_{\text{sign}(x_3^i)\times v^i_3 >0} \dd v^i \mu^{1/4}(v^i) \int^{t^i-\delta}_{0} \dd s e^{-\nu(v^i)(t^1-s)} \\
    &\times \int_{|u|<N} \dd u \Vert f(s,\cdot,x^i_3-(t^i-s)v^i_3,u)\Vert_{\dot{H}^\ell_{x_\parallel}}  \\
    & \lesssim \sup_{x_3,v}\delta^{-1} e^{-\nu_0 (t-t^1)} \int_{\prod_{j=1}^{i-1}\mathcal{V}_j} \prod_{j=1}^{i-1}e^{-\nu_0(t^j-t^{j+1})}\dd \sigma_j   \\
    &\times \int_{\mathbb{R}^2} \dd v^i_\parallel \mu^{1/4}(v^i_\parallel) \int^{t^i-\delta}_0 \dd s e^{-\nu_0(t^i-s)} \int_{y\in (-1,1)} \int_{|u|<N} \dd u \Vert f(s,\cdot,y,u)\Vert_{\dot{H}^\ell_{x_\parallel}} \\
    & \lesssim \sup_{x_3,v} e^{-\nu_0 (t-t^1)} \int_{\prod_{j=1}^{i-1}\mathcal{V}_j} \prod_{j=1}^{i-1}\dd \sigma_j  \int^{t^1-\delta}_0 \dd s e^{-\nu_0(t-s)} \dd u \Vert f(s)\Vert_{L^2_{x_3,v}\dot{H}^\ell_{x_\parallel}} \\
    & \lesssim  e^{-c_0 t}\Vert f_0\Vert_{L^2_{x_3,v}\dot{H}^\ell_{x_\parallel}} + (1+t)^{-\sigma_{q,\ell}}\Vert f_0\Vert_{Z_q} .
\end{align*}
In the second last line, we have applied Proposition \ref{prop.Bdecay}.

Collecting the estimates above, we conclude that
\begin{align}
&\Vert \eqref{bdr_K_i_2d}\Vert_{\dot{H}^{\theta,\ell}} + \Vert \eqref{bdr_K_0_2d} \Vert_{\dot{H}^{\theta,\ell}} \notag \\
&\lesssim o(1)(1+t)^{-\sigma_{q,\ell}}\sup_{s\leq t}(1+s)^{\sigma_{q,\ell}}\Vert f(s)\Vert_{\dot{H}^{\theta,\ell}} + e^{-c_0 t}\Vert f_0\Vert_{\dot{H}^{\theta,\ell}} + (1+t)^{-\sigma_{q,\ell}} \Vert f_0\Vert_{X^{\theta,q}}. \notag
\end{align}
We conclude the lemma.
\end{proof}

\begin{proof}[\textbf{Proof of Proposition \ref{prop:linfty_decay_rate}}]

We mainly perform the a priori estimate. We first set $t\leq T_0$.

We directly compute \eqref{initial} as
\begin{align*}
    & \Vert \eqref{initial}\Vert_{\dot{H}^{\theta,\ell}} \leq \Vert \mathbf{1}_{t^1\leq 0} e^{-\nu_0 t} w \Vert f_0(\cdot,x_3-t v_3,v)\Vert_{\dot{H}^\ell_{x_\parallel}} \Vert_{L^\infty_{x_3,v}} \leq e^{-\frac{\nu_0 t}{2}} \Vert f_0\Vert_{\dot{H}^{\theta,\ell}}.
\end{align*}

For the boundary term we apply Lemma \ref{lemma:bdr} to have
\begin{align*}
    &     \Vert \eqref{f_bdr} \Vert_{\dot{H}^{\theta,\ell}} \leq 4 e^{-\frac{\nu_0}{2} t}\Vert f_0\Vert_{\dot{H}^{\theta,\ell}}+ o(1) (1+t)^{-\sigma_{q,\ell}}\sup_{s\leq t}(1+s)^{\sigma_{q,\ell}}\Vert f(s)\Vert_{\dot{H}^{\theta,\ell}}  \\
    &+ C(T_0)[e^{-c_0 t}\Vert f_0\Vert_{L^2_{x_3,v}\dot{H}^\ell_{x_\parallel}} + (1+t)^{-\sigma_{q,\ell}} \Vert f_0\Vert_{X^{\theta,q}}].
\end{align*}

Last we compute \eqref{K_0} and \eqref{K_1}. Again, we only estimate \eqref{K_1}; the estimate for the other term is the same. We apply the method of characteristic again to $f(s,x-(t-s)v,u)$. We denote 
\begin{align*}
   \tb^u:= \tb(x-(t-s)v,u), \text{ and } \xb^u:=  x-(t-s)v - \tb^u u.  
\end{align*}
We have 
\begin{align}
    & \eqref{K_1} = \mathbf{1}_{t^1>0} \int^{t}_{t^1} \dd s e^{-\nu(v)(t-s)}  \int_{\mathbb{R}^3} \dd u \mathbf{k}(v,u) \notag\\
    &\times \Big[\mathbf{1}_{s-\tb^u\leq 0} e^{-\nu(u)s} f_0(x-(t-s)v-su,u)  \label{K_initial}\\
    & + \mathbf{1}_{s-\tb^u\geq 0} e^{-\nu(u)\tb^u} f(s-\tb^u,\xb^u,u) \label{K_bdr}\\
    & + \int^{s}_{\max\{0,t-\tb^u\}} \dd s' e^{-\nu(u)(s-s')} \int_{\mathbb{R}^3} \dd u\mathbf{k}(u,u')f(s',x-(t-s)v-(s-s')u,u') \Big]. \label{K_K}
\end{align}

We compute \eqref{K_initial} using Lemma \ref{lemma:k_theta} as
\begin{align*}
    & \Vert \eqref{K_initial}\Vert_{\dot{H}^{\theta,\ell}} \\
    &\leq \Big\Vert \mathbf{1}_{t^1>0} w(v)\int^{t}_{t^1} \dd s e^{-\nu(v)(t-s)} e^{-\nu(u)s} \int_{\mathbb{R}^3} \dd u \mathbf{k}(v,u) \mathbf{1}_{s-\tb^u\leq 0} \Vert f_0(\cdot,x_3-(t-s)v_3-su_3,u)\Vert_{\dot{H}^\ell_{x_\parallel}} 
  \Big\Vert_{L^\infty_{x_3,v}} \\
  & \leq C(\nu_0) \Vert f_0\Vert_{\dot{H}^{\theta,\ell}} e^{-\frac{\nu_0}{2}t} \sup_v \int_{\mathbb{R}^3} \dd u \frac{w(v)}{w(u)} \mathbf{k}(v,u) \leq C(\nu_0,\theta)e^{-\frac{\nu_0}{2}t}\Vert f_0\Vert_{\dot{H}^{\theta,\ell}}.
\end{align*}

We compute \eqref{K_bdr} using Lemma \ref{lemma:k_theta} and Lemma \ref{lemma:bdr} as
\begin{align*}
    &  \Vert \eqref{K_bdr}\Vert_{\dot{H}^{\theta,\ell}} \leq \Big\Vert \mathbf{1}_{t^1>0} \int^{t}_{t^1} \dd s e^{-\nu(v)(t-s)} \int_{\mathbb{R}^3} \dd u \frac{w(v)}{w(u)}\mathbf{k}(v,u) \mathbf{1}_{s-\tb^u\geq 0} e^{-\nu(u)\tb^u} \Vert f(s-\tb^u)\Vert_{\dot{H}^{\theta,\ell}} \Big\Vert_{L^\infty_{x_3,v}} \\
    & \leq \sup_{x_3,v} \mathbf{1}_{t^1>0}\int^{t}_{t^1} \dd s e^{-\nu_0(t-s)}\int_{\mathbb{R}^3} \dd u \mathbf{k}_\theta(v,u) \\
    &\times \Big[4 e^{-\frac{\nu_0}{2}s}\Vert f_0\Vert_{\dot{H}^{\theta,\ell}} + o(1)(1+s)^{-\sigma_{q,\ell}}\sup_{s\leq t}(1+s)^{\sigma_{q,\ell}}\Vert f(s)\Vert_{\dot{H}^{\theta,\ell}} \\
    &+ C(T_0)[(1+s)^{-\sigma_{q,\ell}} \Vert f_0\Vert_{X^{\theta,q}} + e^{-c_0 s}\Vert f_0\Vert_{L^2_{x_3,v}\dot{H}_{x_\parallel}^\ell}]\Big] \\
    & \leq C(\nu_0,\theta)e^{-\frac{\nu_0s}{2}} \Vert f_0\Vert_{\dot{H}^{\theta,\ell}}  + o(1)(1+t)^{-\sigma_{q,\ell}} \sup_{s\leq t} (1+s)^{\sigma_{q,\ell}}\Vert f(s)\Vert_{\dot{H}^{\theta,\ell}} \\
    & +C(T_0)[e^{-c_0 t}\Vert f_0\Vert_{L^2_{x_3,v}\dot{H}_{x_\parallel}^\ell} +(1+t)^{-\sigma_{q,\ell}} \Vert f_0\Vert_{X^{\theta,q}}  ].
\end{align*}

For \eqref{K_K}, first, we apply Minkowski inequality to have
\begin{align*}
    & \Vert \eqref{K_K}(t,\cdot,x_3,v)\Vert_{\dot{H}^\ell_{x_\parallel}} \lesssim \mathbf{1}_{t^1>0} \int^{t}_{t^1} \dd s e^{-\nu(v)(t-s)}\int_{\mathbb{R}^3} \dd u \mathbf{k}(v,u) \\
    & \times \int^s_{\max\{0,t-\tb^u\}} \dd s' e^{-\nu(u)(s-s')} \int_{\mathbb{R}^3} \dd u \mathbf{k}(u,u') \Vert f(s',\cdot,x_3-(t-s)v_3-(s-s')u_3,u')\Vert_{\dot{H}^\ell_{x_\parallel}}.
\end{align*}

In the $\dd s'$ integration, we consider $s-\delta\leq s'\leq s$, and we have
\begin{align*}
    & \Vert \eqref{K_K}\mathbf{1}_{s-\delta\leq s'\leq s}\Vert_{\dot{H}^{\theta,\ell}} \lesssim \sup_{s\leq t}(1+s)^{\sigma_{q,\ell}}\Vert f(s)\Vert_{\dot{H}^{\theta,\ell}} \sup_{x_3,v} \mathbf{1}_{t^1>0} \int^{t}_{t^1} \dd s e^{-\nu_0 (t-s)} \int_{\mathbb{R}^3} \dd u \frac{w(v)}{w(u)} \mathbf{k}(v,u) \\
    &\times \int^s_{s-\delta} \dd s' e^{-\nu_0(s-s')} \int_{\mathbb{R}^3}\dd u' \frac{w(u)}{w(u')} \mathbf{k}(u,u')  (1+s')^{-\sigma_{q,0}}\\
    & \lesssim o(1)(1+t)^{-\sigma_{q,\ell}} \sup_{s\leq t}(1+s)^{\sigma_{q,\ell}}\Vert f(s)\Vert_{\dot{H}^{\theta,\ell}}.
\end{align*}

For the $\dd u \dd u'$ integration, we consider $|v-u|<\frac{1}{N} \text{ or }|u-u'|<\frac{1}{N} \text{ or }|u'|>N \text{ or }|u|>N \text{ or }|v|>N$. By Lemma \ref{lemma:k_theta}, we have
\begin{align*}
    & \Vert \eqref{K_K}\mathbf{1}_{|v-u|<\frac{1}{N} \text{ or }|u-u'|<\frac{1}{N} \text{ or }|u'|>N \text{ or }|u|>N \text{ or }|v|>N}\Vert_{\dot{H}^{\theta,\ell}} \\
    &\lesssim \sup_{s\leq t}(1+s)^{\sigma_{q,\ell}}\Vert f(s)\Vert_{\dot{H}^{\theta,\ell}} \sup_{x_3,v} \mathbf{1}_{t^1>0} \int^{t}_{t^1} \dd s e^{-\nu_0 (t-s)} \int_{\mathbb{R}^3} \mathbf{1}_{|v-u|<\frac{1}{N} \text{ or }|u-u'|<\frac{1}{N} \text{ or }|u'|>N \text{ or }|u|>N \text{ or }|v|>N}  \\
    &\times \dd u  \mathbf{k}_\theta(v,u)\int^s_{0} \dd s' e^{-\nu_0(s-s')} \int_{\mathbb{R}^3} \dd u'\mathbf{k}_\theta(u,u')  (1+s')^{-\sigma_{q,0}}\\
    & \lesssim o(1)(1+t)^{-\sigma_{q,\ell}} \sup_{s\leq t}(1+s)^{\sigma_{q,\ell}}\Vert f(s)\Vert_{\dot{H}^{\theta,\ell}}.
\end{align*}

For the rest of the cases, due to $|v|<N,|v-u|>\frac{1}{N},|u-u'|>\frac{1}{N}$, from Lemma \ref{lemma:k_theta}, we have 
\begin{align*}
    & w(v)\mathbf{k}(v,u)\mathbf{k}(u,u')\lesssim 1,
\end{align*}
and the following change of variable holds due to $s'\leq s-\delta$:
\begin{align*}
    &   u_3 \to y=x_3-(t-s)v_3-(s-s')u_3, \ \frac{\dd y}{\dd u_3} = -(s-s'), \ |s-s'|>\delta.
\end{align*}
We compute that
\begin{align*}
    & \Vert \eqref{K_K}\mathbf{1}_{|v-u|>\frac{1}{N} \text{ and }|u-u'|>\frac{1}{N} \text{ and }|u'|<N \text{ and }|u|<N \text{ and }|v|<N \text{ and }s'\leq s-\delta}\Vert_{\dot{H}^{\theta,\ell}} \\
    &\lesssim_\delta \sup_{x_3,v} \mathbf{1}_{t^1>0} \int^{t}_{t^1} \dd s e^{-\nu_0(t-s)} \int_{|u_\parallel|<N} \dd u_\parallel \int^s_{\max\{0,t-\tb^u\}} \dd s' e^{-\nu_0(s-s')} \int_{-1}^1 \dd y\int_{|u'|<N} \dd u' \Vert f(s',\cdot,y,u')\Vert_{\dot{H}^\ell_{x_\parallel}} \\
    & \lesssim  \int^{t}_{0} \dd s e^{-\nu_0(t-s)}  \int^s_{0} \dd s' e^{-\nu_0(s-s')} \Vert f(s')\Vert_{L^2_{x_3,v}\dot{H}^\ell_{x_\parallel}} \\
    &     \lesssim  e^{-c_0 t}\Vert f_0\Vert_{L^2_{x_3,v}\dot{H}^\ell_{x_\parallel}} + (1+t)^{-\sigma_{q,\ell}}\Vert f_0\Vert_{Z_q} .
\end{align*}
In the last line, we have applied Proposition \ref{prop.Bdecay}.

Collecting the estimates above and absorbing the terms $o(1)(1+t)^{-\sigma_{q,\ell}} \sup_{s\leq t}\{(1+s)^{\sigma_{q,\ell}}\Vert f(s)\Vert_{\dot{H}^{\theta,\ell}}\}$, we conclude that for $t\leq T_0$, it holds
\begin{align*}
    &  \Vert f(t)\Vert_{\dot{H}^{\theta,\ell}} \leq C(\nu_0,\theta)e^{-\frac{\nu_0t}{2}}\Vert f_0\Vert_{\dot{H}^{\theta,\ell}} + C(T_0,\delta,N)[e^{-c_0 t}\Vert f_0\Vert_{L^2_{x_3,v}\dot{H}^\ell_{x_\parallel}} + (1+t)^{-\sigma_{q,\ell}}\Vert f_0\Vert_{Z_q}] .
\end{align*}

Let $t=T_0\gg 1$ such that $C(\nu_0,\theta)e^{-\frac{\nu_0T_0}{2}}<e^{-\frac{\nu_0T_0}{4}}$. The estimate above implies 
\begin{align*}
    &  \Vert f(T_0)\Vert_{\dot{H}^{\theta,\ell}} \leq e^{-\frac{\nu_0T_0}{4}}\Vert f_0\Vert_{\dot{H}^{\theta,\ell}} + C(T_0,\delta,N)[e^{-c_0 T_0}\Vert f_0\Vert_{L^2_{x_3,v}\dot{H}^\ell_{x_\parallel}} + (1+T_0)^{-\sigma_{q,\ell}}\Vert f_0\Vert_{Z_q}] .
\end{align*}

It is standard to apply the inductive argument to $t=mT_0,m\geq 2$ and $mT_0\leq t\leq (m+1)T_0$. Hence, the same type of estimate holds for any $t>0$:
\begin{align*}
    &  \Vert f(t)\Vert_{\dot{H}^{\theta,\ell}} \lesssim e^{-\frac{\nu_0t}{4}}\Vert f_0\Vert_{\dot{H}^{\theta,\ell}} + e^{-c_0 t}\Vert f_0\Vert_{L^2_{x_3,v}\dot{H}^\ell_{x_\parallel}} + (1+t)^{-\sigma_{q,\ell}}\Vert f_0\Vert_{Z_q} .
\end{align*}
We refer to details in \cite{EGKM}.

Last, since $x_3$ is bounded and $w^{-1}(v)\in L^2_v$, we have $\Vert f_0\Vert_{L^2_{x_3,v}\dot{H}^\ell_{x_\parallel}} \leq \Vert f_0\Vert_{\dot{H}^{\theta,\ell}}$, and hence
\begin{align*}
    &  \Vert f(t)\Vert_{\dot{H}^{\theta,\ell}} \lesssim e^{-c_0t}\Vert f_0\Vert_{\dot{H}^{\theta,\ell}}  + (1+t)^{-\sigma_{q,\ell}}\Vert f_0\Vert_{Z_q} .
\end{align*}

We complete the estimate for the general case. The estimate under the case $\mathbf{P}_0f_0=0$ is similar. 

We have completed the a priori estimate. It is standard to justify the existence and uniqueness, we refer to the sequential argument in \cite{EGKM}.
\end{proof}

\subsection{Nonlinear estimate under $\Vert \cdot \Vert_{H^{\theta,\ell}}$.}\label{sec:nonliear_asym}

We consider the nonlinear problem \eqref{nonlinear_f}. The following lemma with respect to the decay rate is necessary to obtain the decay rate of the nonlinear operator.

\begin{lemma}\label{lemma:gamma_decay}
Recall the definition of $\sigma_{q,0}$ in \eqref{si-qm}, then it holds that
\begin{align*}
    &    \int_0^t (1+t-s)^{-1} (1+s)^{-2\sigma_{q,0}} \dd s\lesssim (1+t)^{-2\sigma_{q,0}} \log(2+t),
\end{align*}
and
\begin{align*}
    &  \int_0^t  (1+t-s)^{-1} (1+s)^{-2\sigma_{q,1}} \dd s \lesssim \begin{cases}
        (1+t)^{-1},  \text{ if } q<2 \\
        (1+t)^{-1}\log(2+t), \text{ if }q=2.
    \end{cases}
\end{align*}

\end{lemma}

By the Duhammel principle, the solution $f$ to \eqref{nonlinear_f} can be written as
\begin{align}
    & f(t,x,v) = e^{Bt}f_0 + \int_0^t e^{B(t-s)}[\nu \nu^{-1}\Gamma(f,f)(s,x,v)] \dd s. \label{Duhammel}
\end{align}
To study the second term, we denote
\begin{align}
    &    \Psi[\nu h](t,x,v) := \int_0^t e^{B(t-s)}(\nu h)(s,x,v)\dd s.
    \label{def.Psi}
\end{align}
We estimate this term with a general $\nu h$ that satisfies $\mathbf{P}_0(\nu h)=0$ in the following lemma. Then we will set $h = \nu^{-1}\Gamma(f,f)$.

\begin{lemma}\label{lemma:nonlinear_aprior_est}
Assume $\mathbf{P}_0(\nu h) = 0$, $1\leq q<2$, and 
\begin{align*}
& \sup_{s\leq t}  (1+s)^{2\sigma_{q,0}}\Vert  h(s)\Vert_{X^{\theta,1} } + \sup_{s\leq t} (1+s)^{2\sigma_{q,0}} \Vert h(s)\Vert_{H^{\theta,\ell}}<\infty, 
\end{align*}
then it holds that
\begin{align}
    & \Big\Vert \frac{w(v)}{\nu(v)}\Psi[\nu h](t) \Big\Vert_{L^\infty_{x_3,v} H^\ell_{x_\parallel} } \lesssim  (1+t)^{-\sigma_{q,0}} \big[\sup_{s\leq t} \{ (1+s)^{2\sigma_{q,0}}\Vert h(s)\Vert_{H^{\theta,\ell}}\} + \sup_{s\leq t} \{(1+s)^{2\sigma_{q,0}}\Vert h(s)\Vert_{X^{\theta,1} }\} \big]. \label{loss_weight_est}
\end{align}

Moreover, we have the gain in $\nu$:
\begin{align}
    &  \Big\Vert \Psi[\nu h](t) \Big\Vert_{H^{\theta,\ell}} \lesssim (1+t)^{-\sigma_{q,0}} \big[\sup_{s\leq t} \{ (1+s)^{2\sigma_{q,0}}\Vert h(s)\Vert_{H^{\theta,\ell}}\} + \sup_{s\leq t} \{(1+s)^{2\sigma_{q,0}}\Vert h(s)\Vert_{X^{\theta,1} }\} \big]. \label{recover_weight_est}
\end{align}

Similarly, for $1\leq q\leq 2$, if we assume
\begin{align*}
& \sup_{s\leq t}  (1+s)^{2\sigma_{q,1}}\Vert h(s)\Vert_{X^{\theta,1} } + \sup_{s\leq t} (1+s)^{2\sigma_{q,1}} \Vert h(s)\Vert_{H^{\theta,\ell}}<\infty, 
\end{align*}
then it holds that
\begin{align}
    & \Big\Vert \frac{w(v)}{\nu(v)}\Psi[\nu h](t) \Big\Vert_{L^\infty_{x_3,v} H^\ell_{x_\parallel} } \lesssim  (1+t)^{-\sigma_{q,1}} \big[\sup_{s\leq t} \{ (1+s)^{2\sigma_{q,1}}\Vert h(s)\Vert_{H^{\theta,\ell}}\} + \sup_{s\leq t} \{(1+s)^{2\sigma_{q,1}}\Vert h(s)\Vert_{X^{\theta,1} }\} \big], \label{loss_weight_est_faster}
\end{align}
and
\begin{align}
    &  \Big\Vert \Psi[\nu h](t) \Big\Vert_{H^{\theta,\ell}} \lesssim (1+t)^{-\sigma_{q,1}} \big[\sup_{s\leq t} \{ (1+s)^{2\sigma_{q,1}}\Vert h(s)\Vert_{H^{\theta,\ell}}\} + \sup_{s\leq t} \{(1+s)^{2\sigma_{q,1}}\Vert h(s)\Vert_{X^{\theta,1} }\} \big]. \notag
\end{align}

\end{lemma}

\begin{proof}
First we prove \eqref{loss_weight_est}. We compute that
\begin{align*}
    &\Big\Vert \frac{w(v)}{\nu(v)}\int_0^t e^{B(t-s)}(\nu h)(s,x,v) \dd s \Big\Vert_{L^\infty_{x_3,v} H^\ell_{x_\parallel} } \lesssim \int_0^t  \Big\Vert \frac{w(v)}{\nu(v)} e^{B(t-s)}(\nu h)(s,x,v) \Big\Vert_{L^\infty_{x_3,v} H^\ell_{x_\parallel} }  \dd s \\
    & \lesssim  \int_0^t    \Big[e^{-c_0(t-s)} \Big\Vert \frac{w(v)}{\nu(v)} \nu(v) h(s,x,v) \Big\Vert_{L^\infty_{x_3,v}H^\ell_{x_\parallel}}  + (1+t-s)^{-\sigma_{1,1}} \Big\Vert \frac{w(v)}{\nu(v)}\nu(v) h(s,x,v) \Big\Vert_{L^\infty_{x_3,v}L^1_{x_\parallel}}   \Big] \dd s\\
    & \lesssim \int_0^t \big[ e^{-c_0(t-s)} (1+s)^{-2\sigma_{q,0}}   + (1+t-s)^{-\sigma_{1,1}} (1+s)^{-2\sigma_{q,0}}  \big] \dd s \\
    &\times \big[\sup_{s\leq t} \{ (1+s)^{2\sigma_{q,0}}\Vert h(s)\Vert_{H^{\theta,\ell}}\} + \sup_{s\leq t} \{(1+s)^{2\sigma_{q,0}}\Vert h(s)\Vert_{X^{\theta,1} }\} \big] \\
    & \lesssim (1+t)^{-\sigma_{q,0}} \big[\sup_{s\leq t} \{ (1+s)^{2\sigma_{q,0}}\Vert h(s)\Vert_{H^{\theta,\ell}}\} + \sup_{s\leq t} \{(1+s)^{2\sigma_{q,0}}\Vert h(s)\Vert_{X^{\theta,1} }\} \big] .
\end{align*}
In the second line, we have applied Proposition \ref{prop:linfty_decay_rate}, with replacing the weight $w(v)$ by $\frac{w(v)}{\nu(v)}$. In the last line, we have applied Lemma \ref{lemma:gamma_decay} with $0<2\sigma_{q,0}\leq 1$ to have
\begin{align*}
    &    \int_{0}^t (1+t-s)^{-\sigma_{1,1}} (1+s)^{-2\sigma_{q,0}} \dd s = \int_0^t (1+t-s)^{-1} (1+s)^{-(\frac{2}{q}-1)} \dd s \\
    & \lesssim (1+t)^{-(\frac{2}{q}-1)} \log(2+t) \lesssim (1+t)^{-(\frac{1}{q}-\frac{1}{2})} = (1+t)^{-\sigma_{q,0}}.
\end{align*}

Next, we prove \eqref{recover_weight_est}. Recall $B = -v\cdot \nabla_x -\nu(v) + K$, by method of characteristic, we have
\begin{align*}
      &\int_0^t e^{B(t-s)}(\nu h)(s,x,v) \dd s \\
      &= \int_0^t [\mathbf{1}_{t^1<s}e^{-\nu(v)(t-s)} \nu(v) h(x-(t-s)v,v)  + \mathbf{1}_{t^1>s} e^{-\nu(v)\tb}e^{B(t^1-s)}(\nu h)(s,\xb,v)  ]  \dd s \\
    & + \int^t_{\max\{0,t^1\}} \dd s e^{-\nu(t-s)} \int^s_0 K e^{B(s-\tau)}(\nu h)(\tau,x-(t-s)v,v) \dd \tau.
\end{align*}
We control the first term by
\begin{align*}
    &   \Big\Vert \int_0^t \mathbf{1}_{t^1<s} e^{-\nu(v)(t-s)} \nu(v) h(s,x-(t-s)v,v) \dd s\Big\Vert_{H^{\theta,\ell}} \\
    &\lesssim \sup_{s\leq t} \{(1+s)^{2\sigma_{q,0}}\Vert h(s) \Vert_{H^{\theta,\ell}}\}\int_0^t e^{-\nu(v)(t-s)}\nu(v) (1+s)^{-2\sigma_{q,0}}  \dd s \\
    & \lesssim (1+t)^{-2\sigma_{q,0}} \sup_{s\leq t} \{(1+s)^{2\sigma_{q,0}}\Vert h(s)\Vert_{H^{\theta,\ell}}\}, 
\end{align*}
and control the third term using Lemma \ref{lemma:k_theta},
\begin{align*}
    & \Big\Vert   \int^t_{\max\{0,t^1\}} \dd s e^{-\nu(t-s)} \int^s_0 K e^{B(s-\tau)}(\nu h)(\tau,x-(t-s)v,v) \dd \tau   \Big\Vert_{H^{\theta,\ell}}  \\
    &= \Big\Vert \frac{w(v)}{\nu(v)} \nu(v) \int^t_{\max\{0,t^1\}} \dd s e^{-\nu(t-s)} \int^s_0 K e^{B(s-\tau)}(\nu h)(\tau,x-(t-s)v,v) \dd \tau      \Big\Vert_{L^\infty_{x_3,v} H^\ell_{x_\parallel}} \\
    & = \Big\Vert \int^t_{\max\{0,t^1\}} \dd s e^{-\nu(t-s)} \nu(v) \int^s_0 \frac{w(v)}{\nu(v)} \frac{\nu(u)}{w(u)} \int_{\mathbb{R}^3} \dd u  \mathbf{k}(v,u) \frac{w(u)}{\nu(u)} e^{B(s-\tau)}(\nu h)(\tau,x-(t-s)v,u) \dd \tau\Big\Vert_{L^\infty_{x_3,v} H^\ell_{x_\parallel}} \\
    & \lesssim  \Big\Vert \int^t_{\max\{0,t^1\}} \dd s e^{-\nu(t-s)} \nu(v)  \int_{\mathbb{R}^3} \dd u  \mathbf{k}_\theta(v,u) \frac{w(u)}{\nu(u)} \Big\Vert\int^s_0  e^{B(s-\tau)}(\nu h)(\tau,\cdot,x_3-(t-s)v_3,u) \dd \tau \Big\Vert_{H^\ell_{x_\parallel}}\Big\Vert_{L^\infty_{x_3,v}} \\
    & \lesssim \Big\Vert \int^t_{\max\{0,t^1\}} \dd s e^{-\nu(t-s)} \nu(v)  \int_{\mathbb{R}^3} \dd u  \mathbf{k}_\theta(v,u)   \Big\Vert \frac{w(u)}{\nu(u)}\int^s_0  e^{B(s-\tau)}(\nu h)(\tau,x_\parallel,x_3,u) \dd \tau \Big\Vert_{L^\infty_{x_3,u}H^\ell_{x_\parallel}} \Big\Vert_{L^\infty_{v}} \\
    & \lesssim  \sup_v \int^t_0 \dd s e^{-\nu (t-s)}\nu(v) (1+s)^{-\sigma_{q,0}} \times \big[\sup_{s\leq t} \{ (1+s)^{2\sigma_{q,0}}\Vert h(s)\Vert_{H^{\theta,\ell}}\} + \sup_{s\leq t} \{(1+s)^{2\sigma_{q,0}}\Vert h(s)\Vert_{X^{\theta,1} }\} \big] \\
    & \lesssim (1+t)^{-\sigma_{q,0}}\big[\sup_{s\leq t} \{ (1+s)^{2\sigma_{q,0}}\Vert h(s)\Vert_{H^{\theta,\ell}}\} + \sup_{s\leq t} \{(1+s)^{2\sigma_{q,0}}\Vert h(s)\Vert_{X^{\theta,1} }\} \big].
\end{align*}
In the second last line, we have applied \eqref{loss_weight_est}. 

Last we focus on the boundary term. We apply the boundary condition to have
\begin{align*}
    &     \Big\Vert  \int^t_{0} \mathbf{1}_{t^1>s} e^{-\nu \tb} \sqrt{\mu(v)}   \int_{\text{sign}(\xb)\times u_3>0}  \sqrt{\mu(u)}|u_3| e^{B(t^1-s)}(\nu h)(s,\xb,u)   \dd u   \dd s \Big\Vert_{H^{\theta,\ell}} \\
    & = \Big\Vert    e^{-\nu \tb} \sqrt{\mu(v)} \int_{\text{sign}(\xb)\times u_3>0} \dd u\sqrt{\mu(u)} |u_3|\frac{\nu(u)}{w(u)} \frac{w(u)}{\nu(u)} \int_0^{t^1} e^{B(t^1-s)}(\nu h)(s,\xb,u) \dd s \Big\Vert_{H^{\theta,\ell}} \\
    & \lesssim   \Big\Vert w(v)e^{-\nu \tb} \sqrt{\mu(v)} \int_{\text{sign}(\xb)\times u_3>0} \dd u w^{-1}(u) \Big\Vert \frac{w(u)}{\nu(u)} \int_0^{t^1} e^{B(t^1-s)}(\nu h)(s,\xb,u) \dd s \Big\Vert_{L^\infty_{x_3,u}H^\ell_{x_\parallel}}    \Big\Vert_{L^\infty_{x_3,v}} \\
    & \lesssim \sup_{x_3,v} e^{-\nu (t-t^1)} (1+t^1)^{-\sigma_{q,0}} \big[\sup_{s\leq t} \{ (1+s)^{2\sigma_{q,0}}\Vert h(s)\Vert_{H^{\theta,\ell}}\} + \sup_{s\leq t} \{(1+s)^{2\sigma_{q,0}}\Vert h(s)\Vert_{X^{\theta,1} }\} \big] \\
    & \lesssim (1+t)^{-\sigma_{q,0}}\big[\sup_{s\leq t} \{ (1+s)^{2\sigma_{q,0}}\Vert h(s)\Vert_{H^{\theta,\ell}}\} + \sup_{s\leq t} \{(1+s)^{2\sigma_{q,0}}\Vert h(s)\Vert_{X^{\theta,1} }\} \big].
\end{align*}
In the second last line, we have applied \eqref{loss_weight_est}. We conclude \eqref{recover_weight_est}.

The proof of \eqref{loss_weight_est_faster} is similar, the only difference lies in the time integration:
\begin{align*}
    & \int^t_0 (1+t-s)^{-\sigma_{1,1}} (1+s)^{-2\sigma_{q,1}} \dd s = \int_0^t (1+t-s)^{-1} (1+s)^{- (\frac{2}{q})} \dd s \\
    & \lesssim \begin{cases}
        \frac{1}{1+t}, \text{ if } q<2 \\
        \frac{\log(2+t)}{1+t} ,\text{ if } q=2
    \end{cases} \lesssim \begin{cases}
        \frac{1}{(1+t)^{1/q}}, \text{ if } q<2 \\
        \frac{1}{(1+t)^{1/2}} ,\text{ if } q=2
    \end{cases} = (1+t)^{-\sigma_{q,1}}.
\end{align*}
We conclude the lemma.
\end{proof}

Now we set $h=\nu^{-1}\Gamma(f,f)$. To apply Lemma \ref{lemma:nonlinear_aprior_est}, we analyze the nonlinear operator $\nu^{-1}\Gamma(f,f)$ under the norm $\Vert \cdot \Vert_{H^{\theta,\ell}}$ and $\Vert \cdot \Vert_{X^{\theta,1}}$.

\begin{lemma}\label{lemma:gamma_control}
We have the following time-weighted estimate for the nonlinear term:
\begin{align}
    &     (1+t)^{2\sigma_{q,0}}\Vert \nu^{-1} \Gamma(f,f)(t)\Vert_{H^{\theta,\ell}} + (1+t)^{2\sigma_{q,0}} \Vert \nu^{-1}\Gamma(f,f)(t)\Vert_{X^{\theta,1}} \notag\\
    & \lesssim \Big( (1+t)^{\sigma_{q,0}} \Vert f(t)\Vert_{H^{\theta,\ell}} \Big)^2.
    \label{lemma:gamma_controlp1}
\end{align}

Recall \eqref{def.Psi}. Thus for $1\leq q< 2$,
\begin{align}
    &  \Big\Vert  \Psi[\nu \nu^{-1}\Gamma(f,f)](t)  \Big\Vert_{H^{\theta,\ell}} \lesssim (1+t)^{-\sigma_{q,0}} \Big(\sup_{s\leq t} (1+s)^{\sigma_{q,0}} \Vert f(s)\Vert_{H^{\theta,\ell}} \Big)^2. \label{psi_gamma_bdd}
\end{align}

Similarly, if $\sup_{s\leq t}(1+s)^{\sigma_{q,1}}\Vert f(s)\Vert_{H^{\theta,\ell}}<\infty,$ then
\begin{align}
    &     (1+t)^{2\sigma_{q,1}}\Vert \nu^{-1} \Gamma(f,f)(t)\Vert_{H^{\theta,\ell}} + (1+t)^{2\sigma_{q,1}} \Vert \nu^{-1}\Gamma(f,f)(t)\Vert_{X^{\theta,1}} \notag\\
    & \lesssim \Big( (1+t)^{\sigma_{q,1}} \Vert f(t)\Vert_{H^{\theta,\ell}} \Big)^2,
    \label{lemma:gamma_controlp3}
\end{align}
and for $1\leq q\leq 2$, it holds that
\begin{align}
    &  \Big\Vert  \Psi[\nu \nu^{-1}\Gamma(f,f)](t)  \Big\Vert_{H^{\theta,\ell}} \lesssim (1+t)^{-\sigma_{q,1}} \Big(\sup_{s\leq t} (1+s)^{\sigma_{q,1}} \Vert f(s)\Vert_{H^{\theta,\ell}} \Big)^2. \label{psi_gamma_bdd_2}
\end{align}

\end{lemma}

\begin{proof}
We compute that
\begin{align*}
    & \Vert \nu^{-1}\Gamma(f,f)\Vert_{H^{\theta,\ell}} \lesssim \Vert \nu^{-1} w(v) \Gamma(\Vert f\Vert_{H^\ell_{x_\parallel}}, \Vert f\Vert_{H^\ell_{x_\parallel}})  \Vert_{L^\infty_{x_3,v}} \lesssim  \Vert w f \Vert_{L^\infty_{x_3,v}H^\ell_{x_\parallel}}^2.
\end{align*}
Here we have applied Lemma \ref{lemma:gamma_property}.

Next we compute that
\begin{align*}
    & \Vert \nu^{-1}\Gamma(f,f)(t)\Vert_{X^{\theta,1}} \leq \Vert w\nu^{-1}\Gamma(\Vert f\Vert_{L^2_{x_\parallel}}, \Vert f\Vert_{L^2_{x_\parallel}})(t,x_3,v) \Vert_{L^\infty_{x_3,v}} \\
    & \lesssim \Vert wf(t)\Vert_{L^\infty_{x_3,v}L^2_{x_\parallel}}^2 \lesssim \Vert f(t)\Vert_{H^{\theta,\ell}}^2.
\end{align*}
In the first inequality, we have applied the H\"older inequality.

Collecting the estimates above, \eqref{lemma:gamma_controlp1} follows. Moreover, applying Lemma \ref{lemma:nonlinear_aprior_est}, we conclude \eqref{psi_gamma_bdd} as well as \eqref{lemma:gamma_controlp3}. The proof of \eqref{psi_gamma_bdd_2} is similar.  
\end{proof}

Now, coming back to the nonlinear problem \eqref{Duhammel}, we have the following a priori estimate.

\begin{proposition}\label{prop:nonlinear_apriori}
Let $f$ be a solution to the nonlinear problem \eqref{nonlinear_f}, then for $1\leq q< 2$, we have
\begin{align}
      \Vert f(t)\Vert_{H^{\theta,\ell}} & \lesssim e^{-c_0 t}\Vert f_0\Vert_{H^{\theta,\ell}} + (1+t)^{-\sigma_{q,0}} \Vert f_0\Vert_{X^{\theta,q}} \notag\\
    & \quad + (1+t)^{-\sigma_{q,0}} \big(\sup_{s\leq t} \{(1+s)^{\sigma_{q,0}}\Vert f(s)\Vert_{H^{\theta,\ell}}\}   \big)^2.
    \label{prop:nonlinear_apriori.p1}
\end{align}

If $\mathbf{P}_0 f_0 = 0$, then for $1\leq q\leq 2$, it holds that
\begin{align}
        \Vert f(t)\Vert_{H^{\theta,\ell}} &\lesssim e^{-c_0 t} \Vert f_0\Vert_{H^{\theta,\ell}} + (1+t)^{-\sigma_{q,1}}\Vert f_0\Vert_{X^{\theta,q}}\notag \\
    &\quad + (1+t)^{-\sigma_{q,1}} (\sup_{s\leq t}\{(1+s)^{\sigma_{q,1}}\Vert f(s)\Vert_{H^{\theta,\ell}}\})^2.
    \label{prop:nonlinear_apriori.p2}
\end{align}
\end{proposition}

\begin{proof}
From \eqref{Duhammel} together with \eqref{def.Psi}, it is clear that
\begin{align*}
    & \Vert f(t)\Vert_{H^{\theta,\ell}} \leq \Vert e^{Bt}f_0\Vert_{H^{\theta,\ell}} + \Vert \Psi[\nu \nu^{-1}\Gamma(f,f)](t)\Vert_{H^{\theta,\ell}}.
\end{align*}
The first term is estimated by Proposition \ref{prop:linfty_decay_rate} as
\begin{align*}
    & \Vert e^{Bt}f_0\Vert_{H^{\theta,\ell}} \lesssim e^{-c_0 t}\Vert f_0\Vert_{H^{\theta,\ell}} + (1+t)^{-\sigma_{q,0}} \Vert f_0\Vert_{X^{\theta,q}}.
\end{align*}
The second term is estimated by \eqref{psi_gamma_bdd} in Lemma \ref{lemma:gamma_control}. Then, \eqref{prop:nonlinear_apriori.p1} is proved.

The proof of \eqref{prop:nonlinear_apriori.p2} when $\mathbf{P}_0f_0 = 0$ is similar. We conclude the proposition.
\end{proof}

\begin{proof}[\textbf{Proof of Theorem \ref{thm:asymptotic_stability}}]

Given the a priori estimate in Proposition \ref{prop:nonlinear_apriori}, we impose the smallness assumption on the initial condition \eqref{initial_condition}. It is standard to apply the sequential argument in the perturbation framework with the linear well-posedness in Proposition \ref{prop:linfty_decay_rate} to conclude the existence and uniqueness of \eqref{nonlinear_f}. Establishing the positivity of the solution \eqref{proF} in the case of hard sphere model is also standard by the usual sequential argument. For details, we refer to \cite{EGKM} and \cite{ukai2006mathematical}. 
\end{proof}

\subsection{Leading order asymptotic behavior of the solution}\label{sec:leading}

Recall \eqref{nonlinear_f} with initial data $f_0$. Let $f_{\mathbf{p}}$ be the solution to the following problem, where the initial condition is defined as the projection onto the $x_3$-average mass.
\begin{align}
\begin{cases}
       & \p_t f_{\mathbf{p}} + v\cdot \nabla_x f_{\mathbf{p}} + \mathcal{L}f_{\mathbf{p}} = 0, \\
    & f_{\mathbf{p}}(0,x,v) = \mathbf{P}_0 f_0, \\
    & f_{\mathbf{p}}|_{\gamma_-} = P_\gamma f_{\mathbf{p}}. 
\end{cases} \label{linear_leading}
\end{align}

We first prove that the leading asymptotic profile of the nonlinear problem \eqref{nonlinear_f} is given by the linearized problem \eqref{linear_leading}.

\begin{lemma}\label{lemma:f1_leading}
Let $f$ be the solution to \eqref{nonlinear_f} and $f_{\mathbf{p}}$ be the solution to \eqref{linear_leading}. When $1\leq q<2$, it holds that
\begin{align*}
    &  \Vert f-f_{\mathbf{p}}\Vert_{H^{\theta,\ell}}  \lesssim (1+t)^{-2\sigma_{q,0}}\log(2+t)\big[\Vert f_0\Vert_{H^{\theta,\ell}} + \Vert f_0\Vert_{X^{\theta,q}}   \big].
\end{align*}

\end{lemma}

\begin{proof}
Let $g= f-f_{\mathbf{p}}$, then the equation of $g$ is given by
\begin{align*}
\begin{cases}
    &  \p_t g + v\cdot \nabla_x g + \mathcal{L}g = \Gamma(f,f) ,\\
    &g(0,x,v) = f_0 - \mathbf{P}_0 f_0, \\
    & g|_{\gamma_-} = P_\gamma g.
\end{cases}
\end{align*}

Applying Theorem \ref{thm:asymptotic_stability} and Proposition \ref{prop:nonlinear_apriori}, we have the following decay estimate for $f_{\mathbf{p}}$ and $f$:
\begin{align*}
    &\Vert f(t)\Vert_{H^{\theta,\ell}} \lesssim (1+t)^{-\sigma_{q,0}} [\Vert  f_0\Vert_{X^{\theta,q}} + \Vert f_0\Vert_{H^{\theta,\ell}}], \\
    &\Vert f_{\mathbf{p}}(t)\Vert_{H^{\theta,\ell}} \lesssim (1+t)^{-\sigma_{q,0}} [\Vert  \mathbf{P}f_0\Vert_{X^{\theta,q}} + \Vert \mathbf{P}f_0\Vert_{H^{\theta,\ell}}].
\end{align*}

Applying the Duhammel principle to $g$, since $\mathbf{P}_0(f_0-\mathbf{P}_0f_0) = 0$, we have
\begin{align*}
    &  \Vert g(t)\Vert_{H^{\theta,\ell}} \lesssim (1+t)^{-\sigma_{q,1}} [\Vert f_0\Vert_{H^{\theta,\ell}} + \Vert f_0\Vert_{X^{\theta,q}}] + \Vert \Psi[\nu \nu^{-1}\Gamma(f,f)]\Vert_{H^{\theta,\ell}}  \\
    & \lesssim (1+t)^{-\sigma_{q,1}}[\Vert f_0\Vert_{H^{\theta,\ell}} + \Vert f_0\Vert_{X^{\theta,q}}]  \\
    &+ \int_0^t (1+t-s)^{-\sigma_{1,1}} (1+s)^{-2\sigma_{q,0} } \dd s \times \big(\sup_{s\leq t}(1+s)^{\sigma_{q,0}}\Vert f(s)\Vert_{H^{\theta,\ell}}\big)^2 \\
    & \lesssim (1+t)^{-\sigma_{q,1}}[\Vert f_0\Vert_{H^{\theta,\ell}} + \Vert f_0\Vert_{X^{\theta,q}}]  + (1+t)^{-2\sigma_{q,0}}\log(2+t)\times \big(\sup_{s\leq t}(1+s)^{\sigma_{q,0}}\Vert f(s)\Vert_{H^{\theta,\ell}} \big)^2 .
\end{align*}
In the second line, we have applied Lemma \ref{lemma:gamma_control}. In the last line, we have applied Lemma \ref{lemma:gamma_decay}.

Since $q<2$, we have
\begin{align*}
&  -\sigma_{q,1} = -\frac{1}{q} \leq -2\sigma_{q,0} = 1-\frac{2}{q} , \ (1+t)^{-2\sigma_{q,0}}\log(2+t) \lesssim (1+t)^{-\sigma_{q,0}}.
\end{align*}
Applying the Theorem \ref{thm:asymptotic_stability} to $\sup_{s\leq t}(1+s)^{\sigma_{q,0}}\Vert f(s)\Vert_{H^{\theta,\ell}} $, we conclude the lemma.
\end{proof}

Next, we investigate \eqref{linear_leading}. We denote $\mathbf{P}_0 f_0 = \rho_0(x_\parallel)\sqrt{\mu}$. We prove that the leading asymptotic profile of \eqref{linear_leading} is given by the solution to a heat equation with a specific diffusion coefficient, thereby concluding the proof of Theorem \ref{thm:leading_behavior}.
\begin{proposition}\label{proposition:heat_eqn_limit}
For $1\leq q<2$, it holds that
\begin{align*}
     \Vert  f_{\mathbf{p}}(t,x,v) - \rho(t,x_\parallel)\sqrt{\mu(v)} \Vert_{L^2_{x,v}}  & \lesssim (1+t)^{-\sigma_{q,1}}\Vert \rho_0\Vert_{L^q_{x_\parallel}} + e^{-c_0 t}\Vert \rho_0\Vert_{L^2_{x_\parallel}}  \\
    &\lesssim (1+t)^{-\sigma_{q,1}}\Vert f_0\Vert_{Z_q} + e^{-c_0 t} \Vert f_0\Vert_{L^2_{x,v}}.
\end{align*}
Here $\rho(t,x_\parallel)$ is the solution to the heat equation in $\mathbb{R}^2$:
\begin{align*}
\begin{cases}
        & \p_t \rho = \lambda^* \Delta_{x_\parallel} \rho, \\
    & \rho(0,x_\parallel) = \rho_0(x_\parallel),
\end{cases}
\end{align*}
where $\lambda^*>0$ is defined in Proposition \ref{prop:eigenvalue}.

\end{proposition}

\begin{proof}
We define $f_{\mathbf{p}}^n$ to be the solution to the regularized linear problem 
\begin{align*}
\begin{cases}
    & \p_t f_{\mathbf{p}}^n + v\cdot \nabla_x f_{\mathbf{p}}^n + \mathcal{L}_n f_{\mathbf{p}}^n = 0, \\
    & f_{\mathbf{p}}^n(0,x,v) = \mathbf{P}_0 f_0, \\
    & f_{\mathbf{p}}^n|_{\gamma_-} = P_{\gamma} f_{\mathbf{p}}^n,
\end{cases}
\end{align*}
and define $\rho^n$ to be the solution to the 2D heat equation
\begin{align*}
\begin{cases}
    & \p_t \rho^n = \lambda^*_n \Delta_{x_\parallel} \rho^n, \\
    & \rho^n(0,x_\parallel) = \rho_0(x_\parallel).
\end{cases}
\end{align*}
Here $\lambda_n^*$ is defined in Proposition \ref{prop:eigenvalue}.

We rewrite
\begin{align*}
    &    f_{\mathbf{p}} - \rho(t,x_\parallel)\sqrt{\mu(v)} = f_{\mathbf{p}} - f_{\mathbf{p}}^n + f_{\mathbf{p}}^n - \rho^n(t,x_\parallel)\sqrt{\mu(v)} + \rho^n(t,x_\parallel)\sqrt{\mu(v)} - \rho(t,x_\parallel)\sqrt{\mu(v)}.
\end{align*}

Similar to the proof in Proposition \ref{prop.Bdecay}, the first subtraction converges to $0$:
\begin{align*}
    &   \lim_{n\to \infty} \Vert  f_{\mathbf{p}} - f_{\mathbf{p}}^n\Vert_{L^2_{x,v}} = 0.
\end{align*}
From Proposition \ref{prop:eigenvalue}, $\lambda_n^* \to \lambda^*$, thus the third subtraction also converges to $0$:
\begin{align*}
    &  \lim_{n\to\infty}\Vert \rho^n \sqrt{\mu} - \rho \sqrt{\mu}\Vert_{L^2_{x,v}} = 0.
\end{align*}

For the second subtraction, we apply Proposition \ref{prop:inverse_laplace} to have
\begin{align*}
  \hat{f}_{\mathbf{p}}^n(t,k,x_3,v)  &   = e^{-\lambda^*_n |k|^2 t}  \hat{\rho}_0 \sqrt{\mu} + e^{\lambda_n(k) t} (|k| \mathbf{P}_1(k) + |k|^2 \mathbf{P}_2(k) )[ \hat{\rho}_0 \sqrt{\mu}]\mathbf{1}_{|k|<\kappa} \\
    &  + \Phi(k)[\hat{\rho}_0 \sqrt{\mu}] - e^{-\lambda^*_n|k|^2 t} \hat{\rho}_0 \sqrt{\mu} \mathbf{1}_{|k|\geq \kappa} + [e^{-\lambda_n(k) t} - e^{-\lambda^*_n |k|^2 t}] \hat{\rho}_0\sqrt{\mu} \mathbf{1}_{|k|<\kappa}.
\end{align*}

Recall $\lambda_n(k) = -\lambda^*_n |k|^2 +  C^n_\lambda |k|^3$, for $|k|<\kappa$, we have
\begin{align*}
    &  |e^{\lambda_n(k) t}| \leq e^{-\frac{\lambda^*_n}{2} |k|^2 t}, \\
    & |e^{-\lambda_n(k) t} - e^{-\lambda^*_n |k|^2 t}| \leq e^{-\lambda^*_n |k|^2 t} |1- e^{C|k|^3t }| \leq e^{-\lambda^*_n |k|^2 t} | C|k|^3 t| e^{\eta} \lesssim e^{-\frac{\lambda^*_n}{4} |k|^2 t} |k|  .
\end{align*}
Here $0<|\eta| < C|k|^3 t<\frac{\lambda^*_n}{4} |k|^2 t$ for small $|k|<\kappa$. And we have used 
\begin{align*}
    & e^{-\lambda^*_n |k|^2 t} |C|k|^3 t| e^{C|k|^3 t} \lesssim e^{-\frac{\lambda^*_n}{2}|k|^2} |k| e^{\frac{\lambda^*_n}{4}|k|^2 t} \leq  e^{-\frac{\lambda^*_n}{4}|k|^2} |k|.
\end{align*}

Thus, we conclude that
\begin{align*}
    & |\hat{f}_{\mathbf{p}}^n(t,k,x_3,v) - e^{-\lambda^*_n |k|^2 t}\hat{\rho}_0\sqrt{\mu}| \leq \big| e^{\lambda_n(k) t} (|k| \mathbf{P}_1(k) + |k|^2 \mathbf{P}_2(k) )[ \hat{\rho}_0 \sqrt{\mu}]\mathbf{1}_{|k|<\kappa} \big| \\
    &  + |\Phi(k)[\hat{\rho}_0\sqrt{\mu}]| + |e^{-\lambda^*_n \kappa^2 t} \hat{\rho}_0\sqrt{\mu}| + |e^{-\frac{\lambda^*_n}{4}|k|^2 t} |k|  \hat{\rho}_0 \sqrt{\mu}  \mathbf{1}_{|k|<\kappa}|.
\end{align*}

With the extra $|k|$, following the same proof of Lemma \ref{lemma:linear_decay}, we conclude that
\begin{align*}
    & \Vert f_{\mathbf{p}}^n - \rho^n \sqrt{\mu(v)} \Vert_{L^2_{x,v}} \lesssim (1+t)^{-\sigma_{q,1}}\Vert \rho_0\Vert_{L^q_{x_\parallel}} + e^{-c_0 t}\Vert \rho_0\Vert_{L^2_{x_\parallel}}.
\end{align*}

We complete the proof.
\end{proof}

Now we are ready to prove Theorem \ref{thm:leading_behavior} for convergence to the heat equation.

\begin{proof}[\textbf{Proof of Theorem \ref{thm:leading_behavior}}]

Combining Lemma \ref{lemma:f1_leading} and Proposition \ref{proposition:heat_eqn_limit}, we have
\begin{align*}
    & \Vert f-\rho \sqrt{\mu}\Vert_{L^2_{x,v}} \lesssim \Vert f-f_{\mathbf{p}}\Vert_{L^2_{x,v}} + \Vert f_{\mathbf{p}} - \rho \sqrt{\mu}\Vert_{L^2_{x,v}} \lesssim \Vert f-f_{\mathbf{p}}\Vert_{H^{\theta,\ell}} + \Vert f_{\mathbf{p}}-\rho\sqrt{\mu}\Vert_{L^2_{x,v}} \\
    & \lesssim (1+t)^{-2\sigma_{q,0}} \log(2+t)[\Vert f_0\Vert_{H^{\theta,\ell}}+\Vert f_0\Vert_{X^{\theta,q}}] + (1+t)^{-\sigma_{q,1}} [\Vert f_0\Vert_{Z,q} + \Vert f_0\Vert_{L^2_{x,v}} ] \\
    & \lesssim (1+t)^{-2\sigma_{q,0}}\log(2+t)[\Vert f_0\Vert_{H^{\theta,\ell}}+\Vert f_0\Vert_{X^{\theta,q}}].
\end{align*}
In the last inequality, we have used $1\leq q<2$ to have
\begin{align*}
    &   (1+t)^{-\sigma_{q,1}} \lesssim (1+t)^{-2\sigma_{q,0}}\log(2+t).
\end{align*}
We complete the proof.
\end{proof}

\appendix

\section{Trace estimate in the domain $D(A)$}\label{sec:trace}

Denote
\begin{align*}
    & \dd \gamma:= |v_3| \dd v \dd x_1 \dd x_2.
\end{align*}
\begin{lemma}\label{lemma:Ukai_trace}
Define the near-grazing set of $\gamma_+$ and $\gamma_-$ as
\begin{align*}
    & \gamma^\e_\pm := \{(x,v)\in \gamma_{\pm}: |v_3|\leq \e, \text{ or }|v|\geq \frac{1}{\e} \}.
\end{align*}
Then 
\begin{align*}
    &   \int_{\gamma_+\backslash \gamma_+^\e} |f| \dd \gamma \lesssim \Vert f\Vert_{L^1_{x,v}} + \Vert v\cdot \nabla_x  f\Vert_{L^1_{x,v}}.
\end{align*}
    
\end{lemma}

A consequence of the trace theorem above is the following control over $|f|_{L^2_{\gamma_\pm}}$.

\begin{lemma}
Let $f|_{\gamma_-} = P_\gamma f$, then it holds that
\begin{align*}
    &    |f|_{L^2_{\gamma_\pm}} \lesssim  \Vert f\Vert_{L^2_{x,v}} + \Vert v\cdot \nabla_x  f\Vert_{L^2_{x,v}} + |(I-P_\gamma)f|_{L^2_{\gamma_+}}.
\end{align*}
\end{lemma}

\begin{proof}
Without loss of generality, we compute the estimate at $x_3=1$:
    \begin{align*}
    & |f|_{L^2_{\gamma_-}} =  \int_{\gamma_-} \mu(v)\Big| \int_{u_3 > 0 } f \sqrt{\mu(u)} u_3 \dd u \Big|^2 \dd \gamma \\
    &\lesssim  \int_{\mathbb{R}^2} \int_{u_3>0} |f|^2 u_3\dd u \dd x_1 \dd x_2 = |f|_{L^2_{\gamma_+}} =  \int_{\mathbb{R}^2}  \Big[ \int_{|u_3|<\e \text{ or }|u|>\e^{-1}} + \int_{|u_3|>\e , |u|<\e^{-1}}\Big] |f|^2  u_3 \dd u  \dd x_1 \dd x_2 \\
    &  = \int_{\mathbb{R}^2} \int_{|u_3|<\e \text{ or }|u|>\e^{-1}} [|(I-P_\gamma) f|^2 + |P_\gamma f|^2 ]  u_3 \dd u \dd x_1 \dd x_2+  \int_{\gamma_+\backslash \gamma_+^\e} |f|^2 \dd \gamma \\
    & \lesssim  \int_{\gamma_+} |(I-P_\gamma)f|^2 \dd \gamma + \int_{\gamma_+\backslash \gamma_+^\e} |f|^2 \dd \gamma \\
    &+  \int_{\mathbb{R}^2} \int_{|v_3|<\e \text{ or }|v|>\e^{-1}}  \mu(v)v_3 \Big|\int_{u_3>0} f \sqrt{\mu(u)}u_3 \dd u\Big|^2 \dd v \dd x_1 \dd x_2 \\
    &\lesssim  \int_{\gamma_+} |(I-P_\gamma)f|^2 \dd \gamma +  \int_{\gamma_+\backslash \gamma_+^\e} |f|^2 \dd \gamma + \e \int_{\mathbb{R}^2} \int_{u_3>0} |f|^2 u_3 \dd u  \dd x_1 \dd x_2.
\end{align*}
Absorbing the last term by the second line, we obtain
\begin{align*}
    & |f|_{L^2_{\gamma_+}} = \int_{\mathbb{R}^2} \int_{u_3>0} |f|^2 u_3 \dd u \dd x_1 \dd x_2 \lesssim \int_{\gamma_+} |(I-P_\gamma)f|^2 \dd \gamma + \int_{\gamma_+\backslash \gamma_+^\e} |f|^2\dd \gamma.
\end{align*}
Applying Lemma \ref{lemma:Ukai_trace}, we conclude the lemma.
\end{proof}

After taking the Fourier transform in the horizontal direction, we can have a similar trace control.

\begin{lemma}\label{lemma:trace_k}
Let $\hat{f}|_{\gamma_-} = P_\gamma \hat{f}$, then it holds that
\begin{align*}
    &  |\hat{f}|_{L^2_{\gamma_\pm}} \lesssim \Vert \hat{f}\Vert_{L^2_{x_3,v}} + \Vert \{i(k_1v_1+k_2v_2) + v_3 \p_{x_3} \}\hat{f} \Vert_{L^2_{x_3,v}} + |(I-P_\gamma)\hat{f}|_{L^2_{\gamma_+}}.
\end{align*}
\end{lemma}

\noindent {\bf Acknowledgment:}\,
HXC thanks the host from the Department of Mathematics, The Chinese University of Hong Kong since 2023. The research of RJD was partially supported by the General Research Fund (Project No.~14303523) from RGC of Hong Kong and also by the grant from the National Natural Science Foundation of China (Project No.~12425109). The research of SQL was partially supported by the grant from the National Natural Science Foundation of China (Project No.~12325107). 

\medskip
\noindent\textbf{Data Availability Statement:}
Data sharing is not applicable to this article as no datasets were generated or analysed during the current study.

\noindent\textbf{Conflict of Interest:}
The authors declare that they have no conflict of interest.

\hide

It remains to prove that the following problem only has a trivial solution under the condition that $|\tau|> 1, |\sigma|<c_0\ll 1$, and  $\int_{-1}^1\int_{\mathbb{R}^3}\hat{f}\sqrt{\mu}\dd v\dd x_3 = 0$:
\begin{align*}
    &   (\sigma+i\tau) \hat{f} + i(k_1v_1+k_2v_2) \hat{f} + v_3 \p_{x_3}\hat{f} + \mathcal{L}\hat{f} = 0, \\
    & \hat{f}|_{\gamma_-} = P_{\gamma}\hat{f}.
\end{align*}

\Blue{When $k=0$, we have known that $\lambda = 0$ is the only eigenvalue on the imaginary axis. For $\sigma<0$, is it possible to apply perturbation theory to conclude that there exists at most finite eigenvalue? If this is true, when $k\to 0$, is it possible to apply perturbation theory to conclude that the eigenvalue must converge to the eigenvalue of $k=0$? If this true, then we can conclude that except the eigenvalue we find, the other eigenvalue, if exists, must admit an upper bound for the real part, since their limit corresponds to the eigenvalue of $k=0$, which is away from the imaginary axis(except $\lambda=0$). }

\unhide

\hide

We focus on the regularity estimate to the following problem
\begin{align*}
\begin{cases}
      &   (\sigma+ i\tau) \hat{f} + i(k_1v_1+k_2v_2)\hat{f} + v_3 \p_{x_3}\hat{f} + \mathcal{L}\hat{f} = 0 ,\\
    & \hat{f}_{\gamma_-} = P_{\gamma}\hat{f}.   
\end{cases}
\end{align*}
We denote a weight as 
\begin{align*}
    & \alpha(v) = \max\{|v_3|^{0.1},|v_3|\}.
\end{align*}

Denote
\begin{align*}
    & \tb(x_3,v) = \begin{cases}
        & \frac{x_3+1}{v_3} \text{ if }v_3 > 0 \\
        & \frac{x_3-1}{v_3} \text{ if }v_3<0
    \end{cases}, \ \xb(x_3,v) = \begin{cases}
         & 1 \text{ if }v_3 < 0 \\
         & -1 \text{ if }v_3 >0
    \end{cases}
\end{align*}
and
\begin{align*}
    & \tilde{\nu}(v) = \nu(v) + \sigma + i(\tau + k_1v_1 + k_2v_2) .
\end{align*}
It holds that
\begin{align*}
    &\nu(v)\geq |\tilde{\nu}(v)|\geq \frac{\nu(v)}{2}.
\end{align*}

By method of characteristic, we have
\begin{align*}
    &    f(x_3,v) = e^{-\tilde{\nu}(v)\tb(x_3,v)}f(\xb(x_3,v),v) \\
    & + \int^t_{t-\tb(x_3,v)} \dd s  e^{-\tilde{\nu}(v)(t-s)}\int_{\mathbb{R}^3} \mathbf{k}(v,u)f(x_3-(t-s)v_3,u) \dd u.
\end{align*}
Taking derivative in $x_3$, we obtain
\begin{align*}
    &    \p_{x_3} f(x_3,v) =  -\tilde{\nu}(v)\p_{x_3}\tb(x_3,v) e^{-\tilde{\nu}(v)\tb(x_3,v)} f(\xb(x_3,v),v) \\
    & + \p_{x_3}\tb(x_3,v) e^{-\tilde{\nu}(v)\tb(x_3,v)}\int_{\mathbb{R}^3} \mathbf{k}(v,u) f(\xb(x_3,v),u) \dd u \\
    & + \int^t_{t-\tb(x_3,v)} \dd s e^{-\nu(v)(t-s)}\int_{\mathbb{R}^3} \mathbf{k}(v,u) \p_{x_3} f(x_3-(t-s)v_3,u) \dd u.
\end{align*}
Here we applied the fact that $\xb(x_3,v)$ does not depend on $x_3$.

Applying the formulation to $f(x_3-(t-s)v_3,u)$, denoting $y_3 = x_3-(t-s)v_3$, we obtain
\begin{align}
&\alpha(v)\p_{x_3}f(x_3,v) =- \alpha(v)\tilde{\nu}(v)\p_{x_3}\tb(x_3,v) e^{-\tilde{\nu}(v)\tb(x_3,v)} f(\xb(x_3,v),v) \label{regu:1}\\
& + \alpha(v)\p_{x_3}\tb(x_3,v) e^{-\tilde{\nu}(v)\tb(x_3,v)}\int_{\mathbb{R}^3} \mathbf{k}(v,u) f(\xb(x_3,v),u) \dd u  \label{regu:2}\\
& - \alpha(v)\int^t_{t-\tb(x_3,v)} \dd s e^{-\tilde{\nu}(v)(t-s)}\int_{\mathbb{R}^3} \dd u \mathbf{k}(v,u)  \p_{x_3} \tb(y_3,u) \tilde{\nu}(u) e^{-\tilde{\nu}(u)\tb(y_3,u)} f(\xb(y_3,u),u) \label{regu:3} \\
& + \alpha(v)\int^t_{t-\tb(x_3,v)} \dd s e^{-\tilde{\nu}(v)(t-s)} \int_{\mathbb{R}^3} \dd u \mathbf{k}(v,u) \p_{x_3}\tb(y_3,u) e^{-\tilde{\nu}(u)\tb(y_3,u)} \int_{\mathbb{R}^3}\mathbf{k}(u,u') f(\xb(y_3,u),u') \dd u' \label{regu:4} \\
& + \alpha(v)\int^t_{t-\tb(x_3,v)} \dd s e^{-\tilde{\nu}(v)(t-s)} \int_{\mathbb{R}^3} \dd u \mathbf{k}(v,u) \int^s_{s-\tb(y_3,u)} \dd s' e^{-\tilde{\nu}(u)(s-s')} \int_{\mathbb{R}^3} \dd u'\mathbf{k}(u,u') \p_{x_3} f(y_3-(s-s')u_3,u'). \label{regu:5} 
\end{align}

We compute
\begin{align*}
    & \Vert \eqref{regu:1} \Vert_{L^2_{x_3,v}} \lesssim \Big\Vert \frac{\alpha(v)}{|v_3|} \tilde{\nu}(v)e^{-\tilde{\nu}(v)\tb(x_3,v)}   \sqrt{\mu(v)} \int_{n(x_3)\cdot u >0} \sqrt{\mu(u)} f(x_3,u)(n(x_3)\cdot u) \dd u    \Big\Vert_{L^2_{x_3,v}}\\
    & \lesssim \Vert wf\Vert_{L^\infty_{x_3,v}} \Vert \frac{1}{|v_3|^{0.9}} e^{-\nu(v)}
\end{align*}

\unhide

\hide

The semi-group $e^{At}$ has the following property:
\begin{lemma}\label{lemma:At_ell_beta_decay}

\begin{align*}
    &   \Vert e^{At} u\Vert_{\beta,\ell} \lesssim e^{-\frac{\nu_0}{2}t}\Vert u\Vert_{\beta,\ell}, \\
    & \Vert G_j u\Vert_{\beta,\ell} \lesssim e^{-\frac{\nu_0}{2}t}\Vert u\Vert_{\beta,\ell}.
\end{align*}
    
\end{lemma}

\begin{proof}

\end{proof}

Next we investigate the property of the semi-group $e^{Bt}$ on $H_{\beta,\ell}$. We express $e^{Bt}$ as
\begin{align*}
    &  e^{Bt} = e^{At} + (e^{At}K)*e^{tB}, \\
    &  e^{Bt} = \sum_{k=0}^N G_j(t) + (G_N(t)K)*e^{Bt}, \\
    & G_0(t) = e^{At}, \ G_j(t) = (e^{At}K)*G_{j-1}.
\end{align*}

Applying the decomposition that $e^{Bt} = E_1(t) + E_2(t)$, we further have
\begin{align*}
    & e^{Bt} = D_1(t) + D_2(t) , \\ 
    &D_1(t) = (G_N(t)K)*E_1(t) , \  D_2(t) = \sum_{j=0}^N G_j(t) + (G_N(t)K)*E_2(t).
\end{align*}

\begin{lemma}\label{lemma:Hbl_semi_group}
\begin{align*}
    &     \Vert \p_{x_\parallel}^\alpha D_1(t)u\Vert_{\beta,\ell} \lesssim (1+t)^{-\sigma_{q,m}} [\Vert \p_x^{\alpha'} u\Vert_{Z_q} + \Vert \p_{x_\parallel}^\alpha u\Vert_{\beta,\ell}   ], \\
    & \Vert \p_{x_\parallel}^\alpha D_1(t)(\mathbf{I}-\mathbf{P})u\Vert_{\beta,\ell} \lesssim (1+t)^{-\sigma_{q,m+1}} [\Vert \p_x^{\alpha'} u\Vert_{Z_q} + \Vert \p_{x_\parallel}^\alpha u\Vert_{\beta,\ell}   ], 
\end{align*}

\begin{align*}
    &   \Vert \p_{x_\parallel}^\alpha D_2(t)u\Vert_{\beta,\ell} \lesssim e^{-c_0 t} [\Vert \p_{x_\parallel}^{\alpha} u\Vert_{\beta,\ell} + \Vert \p_{x_\parallel}^\alpha u\Vert_{H^\ell_{x_\parallel} L^2_{x_3,v}} ]. 
\end{align*}

\end{lemma}

\begin{proof}

\end{proof}

\unhide

\bibliographystyle{siam}


\end{document}